\documentclass[12pt,thmsa]{article}
\usepackage{amssymb}
\usepackage{amsfonts}
\usepackage{enumerate}
\usepackage{amsmath}
\usepackage{cite}
\usepackage[top=3.8cm,bottom=3.8cm,textwidth=16cm,centering,a4paper]{geometry}

\newtheorem{theorem}{Theorem}[section]

\newtheorem{proposition}[theorem]{Proposition}
\newtheorem{corollary}[theorem]{Corollary}
\newtheorem{remark}[theorem]{Remark}

\newtheorem{lemma}[theorem]{Lemma}

\newtheorem{definition}[theorem]{Definition}
\newenvironment{proof}[1][Proof]{\noindent\textbf{#1.} }{\ \rule{0.5em}{0.5em}}

\begin{document}

\title{Standing waves for the NLS equation with competing nonlocal and local
nonlinearities: the double $L^{2}$-supercritical case}
\date{}
\author{Shuai Yao$^{a}$\thanks{%
E-mail address: shyao2019@163.com (S. Yao)}, Hichem Hajaiej$^{b}$\thanks{%
E-mail address: hhajaie@calstatela.edu (H. Hajaiej)}, Juntao Sun$^{c}$%
\thanks{%
E-mail address: jtsun@sdut.edu.cn (J. Sun)}, Tsung-fang Wu$^{d}$\thanks{%
E-mail address: tfwu@nuk.edu.tw (T.-F. Wu)} \\
{\footnotesize $^{a}$\emph{School of Mathematics and Statistics, Central
South University, Changsha 410083, PR China}}\\
{\footnotesize $^{b}$\emph{Department of Mathematics, California State
University at Los Angeles, Los Angeles, CA 90032, USA}}\\
$^{c}${\footnotesize \emph{School of Mathematics and Statistics, Shandong
University of Technology, Zibo 255049, PR China}}\\
{\footnotesize $^{d}$\emph{Department of Applied Mathematics, National
University of Kaohsiung, Kaohsiung 811, Taiwan}}}
\maketitle

\begin{abstract}
We investigate the NLS equation with competing Hartree-type and power-type
nonlinearities
\begin{equation*}
\begin{array}{ll}
i\partial _{t}\psi +\Delta \psi +\gamma (I_{\alpha }\ast |\psi |^{p})|\psi
|^{p-2}\psi +\mu |\psi |^{q-2}\psi =0, & \text{ }\forall (t,x)\in \mathbb{%
R\times R}^{N},%
\end{array}%
\end{equation*}%
where $\gamma \mu <0$. We establish conditions for the local well-posedness
in the energy space. Under the double $L^{2}$-supercritical case, we prove
the existence and multiplicity of standing waves with prescribed mass by
developing a constraint method when $\gamma <0,\mu >0$ and $\gamma >0,\mu
<0, $ respectively. Moreover, we prove weak orbital stablility and strong
instability of standing waves by considering a suitable local minimization
problem and by analyzing the fibering mapping, respectively. A new analysis
of the fibering mapping is performed in this work. We believe that it is
innovative as it was not discussed at all in any previous results. The lower
bound rate of blow-up solutions for the Cauchy problem is given as well. Due
to the different \textquotedblleft strength" of the two types of
nonlinearities, we find some essential differences in our results between
these two competing cases. We will be dealing with two major scenarios that
are totally different from each other due to their diverse geometric
structure. This leads to surprising findings. Additionally, the competing
pure power-type nonlinearities case can be derived from our study thanks to
a good choice of the kernel of the Hartree term.
\end{abstract}

\tableofcontents



\section{Introduction}

Consider the Cauchy problem for the NLS equation with combined Hartree-type
and power-type nonlinearities:%
\begin{equation}
\left\{
\begin{array}{l}
i\partial _{t}\psi +\Delta \psi +\gamma (I_{\alpha }\ast |\psi |^{p})|\psi
|^{p-2}\psi +\mu |\psi |^{q-2}\psi =0, \\
\psi (0,x)=\psi _{0}(x),%
\end{array}%
\right.  \label{e1-0}
\end{equation}%
where $\psi (t,x)$ is the complex valued function in the spacetime $\mathbb{%
R\times R}^{N}$ ($N\geq 1$), $2_{\alpha }\leq p\leq 2_{\alpha }^{\ast }$ ($%
2_{\alpha }=\frac{N+\alpha }{N},$ $2_{\alpha }^{\ast }=\frac{N+\alpha }{N-2}$
if $N\geq 3$ and $2_{\alpha }^{\ast }=\infty $ if $N=1,2$), $2<q\leq 2^{\ast
}$ ($2^{\ast }=\frac{2N}{N-2}$ if $N\geq 3$ and $2^{\ast }=\infty $ if $%
N=1,2 $), the parameters $\gamma ,\mu \in \mathbb{R}$ and $I_{\alpha }$ is
the Riesz potential of order $\alpha \in (0,N)$ defined by
\begin{equation*}
I_{\alpha }=\frac{A(N,\alpha )}{|x|^{N-\alpha }}\quad \text{with}\quad
A(N,\alpha )=\frac{\Gamma (\frac{N-\alpha }{2})}{\pi ^{N/2}2^{\alpha }\Gamma
(\frac{\alpha }{2})}\text{ for each }x\in \mathbb{R}^{N}\backslash \{0\},
\end{equation*}%
and $\ast $ the convolution product on $\mathbb{R}^{N}$. The exponent $%
2_{\alpha }$ ($2_{\alpha }^{\ast }$) is the lower (upper) critical exponent
in the sense of Hardy-Littlewood-Sobolev inequality, and $2^{\ast }$ is the
Sobolev critical exponent in the sense of Sobolev inequality.

Equation (\ref{e1-0}) appears in the context of various physical models. It
seems to be first introduced in the pioneering work of Fr\"{o}hlich and
Pekar's to describe the model of the polaron, where free electrons in an
ionic lattice interact with phonons associated with deformations of the
lattice or with the polarisation that it creates in the medium (interaction
of an electron with its own hole) \cite{F1,F2}. It is a large system model
of non-relativistic bosonic atoms and molecules under an attractive
interaction that is weaker and has a longer range than the one of the
nonlinear Schr\"{o}dinger equation, where the interaction potential $%
I_{\alpha }$ is formally Dirac's delta at the origin \cite{FL}. Moreover, (%
\ref{e1-0}) also arises as a mean-field limit of a bosonic system with
attractive two-body interactions which can be taken rigorously in many cases
\cite{FL,LNR}.

For the Cauchy problem (\ref{e1-0}), when the initial data $\psi _{0}\in
H^{1}(\mathbb{R}^{N})$, we can study the finite Hamiltonian or the finite
energy solutions, which satisfy energy (Hamiltonian), mass and momentum
conservations:
\begin{equation*}
E_{\gamma ,\mu }[\psi (t)]:=\frac{1}{2}\int_{\mathbb{R}^{N}}|\nabla \psi
|^{2}dx-\frac{\gamma }{2p}\int_{\mathbb{R}^{N}}(I_{\alpha }\ast |\psi
|^{p})|\psi |^{p}dx-\frac{\mu }{q}\int_{\mathbb{R}^{N}}|\psi
|^{q}dx=E_{\gamma ,\mu }[\psi _{0}],
\end{equation*}%
\begin{equation*}
M[\psi (t)]:=\int_{\mathbb{R}^{N}}|\psi |^{2}dx=M[\psi _{0}],
\end{equation*}%
\begin{equation*}
P[\psi (t)]:=\mathrm{Im}\int_{\mathbb{R}^{N}}\bar{\psi}\nabla \psi dx=P[\psi
_{0}].
\end{equation*}

An important feature related to the nonlinear evolution equations such as (%
\ref{e1-0}) is the study of standing waves. A standing wave of (\ref{e1-0})
is a solution of the form%
\begin{equation*}
\psi (t,x)=e^{i\lambda t}u(x),
\end{equation*}%
where $\lambda \in \mathbb{R}$ and $u$ satisfies the stationary equation
\begin{equation}
-\Delta u+\lambda u=\gamma (I_{\alpha }\ast \left\vert u\right\vert
^{p})|u|^{p-2}u+\mu |u|^{q-2}u\quad \text{in}\quad \mathbb{R}^{N}.
\label{e1-1}
\end{equation}%
At this point, there are two different ways to deal with (\ref{e1-1})
according to the role of the frequency $\lambda $:\newline
$(i)$ the frequency $\lambda $ is a fixed and assigned parameter;\newline
$(ii)$ the frequency $\lambda $ is an unknown of the problem.

For case $(i),$ one can see that solutions of (\ref{e1-1}) can be obtained
as critical points of the functional defined in $H^{1}(\mathbb{R}^{N})$ by%
\begin{equation*}
I(u):=\frac{1}{2}\int_{\mathbb{R}^{N}}|\nabla u|^{2}dx+\frac{\lambda }{2}%
\int_{\mathbb{R}^{N}}|u|^{2}dx-\frac{\gamma }{2p}\int_{\mathbb{R}%
^{N}}(I_{\alpha }\ast |u|^{p})|u|^{p}dx-\frac{\mu }{q}\int_{\mathbb{R}%
^{N}}|u|^{q}dx.
\end{equation*}%
This case has been extensively studied by many authors in the last years,
see for example \cite{AN,GS,MS,R,SW,SW1,SW3} and the references therein.

Alternatively one can look for solutions of (\ref{e1-1}) with $\lambda $
unknown. In this case $\lambda \in \mathbb{R}$ appears as a Lagrange
multiplier and $L^{2}$-norms of solutions are prescribed, which are usually
called normalized solutions. This study seems particularly meaningful from
the physical point of view, since standing waves of (\ref{e1-0}) conserve
their mass along time, and physicists are very interested in the stability.
Here we focus on this issue. For $c>0$ given, we consider the problem of
finding solutions to
\begin{equation}
\left\{
\begin{array}{ll}
-\Delta u+\lambda u=\gamma (I_{\alpha }\ast \left\vert u\right\vert
^{p})|u|^{p-2}u+\mu |u|^{q-2}u & \quad \text{in}\quad \mathbb{R}^{N}, \\
\int_{\mathbb{R}^{N}}|u|^{2}dx=c. &
\end{array}%
\right.  \tag{$P_{c}$}
\end{equation}%
Solutions of $(P_{c})$ can be obtained as critical points of the energy
functional $E_{\gamma ,\mu }:H^{1}(\mathbb{R}^{N})\rightarrow \mathbb{R}$
given by%
\begin{equation*}
E_{\gamma ,\mu }(u):=\frac{1}{2}\int_{\mathbb{R}^{N}}|\nabla u|^{2}dx-\frac{%
\gamma }{2p}\int_{\mathbb{R}^{N}}(I_{\alpha }\ast |u|^{p})|u|^{p}dx-\frac{%
\mu }{q}\int_{\mathbb{R}^{N}}|u|^{q}dx  \label{e1-3}
\end{equation*}%
on the constraint
\begin{equation*}
S(c):=\left\{ u\in H^{1}(\mathbb{R}^{N})\text{ }|\text{ }\int_{\mathbb{R}%
^{N}}|u|^{2}dx=c\right\} .  \label{e1-4}
\end{equation*}%
It is easy to show that $E_{\gamma ,\mu }$ is a well-defined and $C^{1}$
functional on $S(c)$ for $2_{\alpha }\leq p\leq 2_{\alpha }^{\ast }$ and $%
2<q\leq 2^{\ast }.$

As we will see, the number and properties of normalized solutions to $%
(P_{c}) $ are strongly affected by further assumptions on the exponents $p,q$%
, the parameters $\gamma ,\mu $ and the mass $c.$ When the problem involves
pure power-type nonlinearity, i.e. $\gamma =0$ and $\mu >0,$ from the
variational point of view, $E_{0,\mu }$ is bounded from below on $S\left(
c\right) $ for $2<q<\overline{q}:=2+4/N,$ here $\overline{q}$ is called the $%
L^{2}$-critical exponent for the power nonlinearity. Thus, for each $c>0$, a
ground state can be found as a global minimizer of $E_{0,\mu }$ on $S\left(
c\right) $, see \cite{L0,S0,S}. Moreover, the set of ground states is
orbitally stable \cite{CL,HS,S0}. For $\overline{q}<q<2^{\ast }$, on the
contrary, $E_{0,\mu }$ is unbounded from below on $S\left( c\right) $.
Jeanjean \cite{J} has developed an approach to show that we can still define
normalized ground state and they do exist for $c>0$ also in this case.
Recently, Hajaiej and Song \cite{HS1} provided a general approach to prove
the existence of normalized solutions for this kind of problem.

When the problem only involves the Hartree-type nonlinearity, i.e. $\gamma
>0 $ and $\mu =0,$ (\ref{e1-1}) is usually called the nonlinear Choquard or
Choquard--Pekar equation. It has several physical origins. In the case where
$N=3$ and $\alpha =p=2,$ the problem%
\begin{equation}
-\Delta u+\lambda u=\gamma \left( I_{2}\ast u^{2}\right) u\quad \text{in}%
\quad \mathbb{R}^{3}  \label{e1-2}
\end{equation}%
appears as a model in quantum theory of a polaron at rest \cite{P}. The
time-dependent form of (\ref{e1-2}) also describes the self-gravitational
collapse of a quantum mechanical wave-function \cite{P1}. Lieb \cite{L5}
proved the existence and uniqueness of normalized solutions for (\ref{e1-2})
on $S\left( c\right) $ by using symmetrization techniques, and Lions \cite{L}
studied the existence and stability issues of normalized solutions for (\ref%
{e1-2}) on $S\left( c\right) $. Recently, for (\ref{e1-1}) with $\gamma >0$
and $\mu =0,$ Li and Ye \cite{LY} showed that if $N\geq 3$ and $\overline{p}%
:=\frac{N+\alpha +2}{N}<p<2_{\alpha }^{\ast },$ then $E_{\gamma ,0}$ has a
MP geometry on $S(c)$ and there exists a normalized solution for each $c>0.$
Usually, $\overline{p}$ is called the $L^{2}$-critical exponent for the
Hartree-type nonlinearity.

When the problem involves both Hartree-type and a power-type nonlinearities,
the situation seems to become more complicated. We summarize two different
cases as follows.\newline
\textbf{$\bullet $ The mixed attractive case ($\gamma >0,\mu >0$).} Li \cite%
{L1,L3} studied the upper critical case of $(P_{c}),$ i.e. $p=2_{\alpha
}^{\ast }$ and $2<q<2^{\ast }.$ She concluded that $(P_{c})$ has one radial
normalized solution when $\overline{q}<q<2^{\ast }$ while two radial
normalized solutions when $2<q<\overline{q}.$\ Moreover, qualitative
properties and stability of normalized solutions are described as well. For
the lower critical case of $(P_{c}),$ i.e. $p=2_{\alpha }$ and $2<q\leq
2^{\ast },$ the first and third authors of the current paper proved several
nonexistence and existence results by introducing a method of adding mass
term and Sobolev subcritical approximation in \cite{YCRS}.\newline
\textbf{$\bullet $ The mixed competing case ($\gamma \mu <0$).} In the
specific case of $N=3,\gamma <0,\mu >0$ and $\alpha =p=2,$ (\ref{e1-0}) is
reduced to the well-known Schr\"{o}dinger--Poisson equations, appearing in
the physical literature as an approximation of the Hartree--Fock model of a
quantum many--body system of electrons, see \cite{LL0} for a mathematical
introduction to Hartree--Fock method. Clearly, the Hartree-type nonlinearity
in Schr\"{o}dinger--Poisson equations belongs to the $L^{2}$-subcritical
case, since $2=p<\overline{p}=7/3.$ Bellazzini and Siciliano \cite{BS1} and
Jeanjean and Luo \cite{JL} obtained a ground state as a global minimizer of $%
E_{\gamma ,\mu }$ on $S(c)$ for $2<q<3$ and $c>0$ sufficiently small, and
for $3<q<10/3$ and $c>0$ sufficiently large, respectively. Bellazzini et al.
\cite{BJL} found a mountain-pass normalized solution for $10/3<q<6$ and $c>0$
sufficiently small. Moreover, the strong instability of standing waves at
the mountain pass energy level was established. Later, Feng et al. \cite{FCW}
proved that the standing waves are strongly unstable by blow up for $q=10/3.$

The study of standing waves with prescribed mass for NLS equation with
competing Hartree-type and power-type nonlinearities has attracted many
mathematicians. However, so far only a very few results exist in our
context, \cite{BJL,BS1,FCW,JL}, and many interesting cases are still open,
such as the double $L^{2}$-supercritical case, since its geometric structure
is totally different from others, leading us to some surprising results.

In this paper, our aim is to develop a method that will enable us to
establish optimal results on the existence, multiplicity and dynamics of
standing waves with a prescribed mass for (\ref{e1-0}) with $\gamma \mu <0$
and $p,q$ being both $L^{2}$-supercritical exponents. More precisely, we
firstly establish the local well-posedness of the Cauchy problem (\ref{e1-0}%
) in the energy space $H^{1}(\mathbb{R}^{N}).$ Secondly, the existence and
multiplicity of normalized solutions of $(P_{c})$ are studied under the
cases of $\gamma <0,\mu >0$, and $\gamma >0,\mu <0$. As a result, the
corresponding standing waves with prescribed mass of (\ref{e1-0}) are
obtained. We note that there are some essential differences in the geometry
of the corresponding functionals $E_{\gamma ,\mu }$ between the two above
competing cases, since \textquotedblleft the strength" of Hartree-type and
power-type nonlinearities are different. Thirdly, we study the stability and
instability of standing waves with prescribed mass for (\ref{e1-0}).
Finally, we compute the lower bound rate of blow-up solutions for the Cauchy
problem (\ref{e1-0}). It is worth emphasizing that one has to face more
challenges due to the interaction of the double $L^{2}$-supercritical
nonlinearities. In order to overcome these considerable difficulties, new
ideas and techniques have been explored. More details will be discussed in
the next subsection.

\subsection{Main results}

Before stating our main results, we agree that when $p=2_{\alpha }^{\ast }$
or $q=2^{\ast }$ is involved, we always assume that $N\geq 3.$ Firstly, we
deal with the local well-posedness of the Cauchy problem (\ref{e1-0}) in the
energy space $H^{1}(\mathbb{R}^{N})$ with standard contraction mapping
argument.

\begin{theorem}
\label{t1} Let $N\geq 1,$ $\gamma $ and $\mu $ be nonzero real constants and
the initial data $\psi (0,x)=\psi _{0}(x)\in H^{1}(\mathbb{R}^{N}).$ If $%
\alpha \in (\max \{0,N-4\},N)$, $2\leq p\leq 2_{\alpha }^{\ast }$ and $%
2<q\leq 2^{\ast }$ (with additional assumption of smallness of $\Vert \psi
_{0}\Vert _{H^{1}}$ for $p=2_{\alpha }^{\ast }$ or $q=2^{\ast }$), then
there exists $T_{max}$ and a unique solution $\psi (t,x)\in C([0,T_{\max
});H^{1})$ of the Cauchy problem (\ref{e1-0}) such that either $T_{\max
}=\infty $ (\textbf{global existence}) or $T_{\max }<\infty $ and $\Vert
\psi (t,x)\Vert _{H^{1}}\rightarrow \infty $ as $t\uparrow T_{\max }$ (%
\textbf{blow up}). Moreover, if $|x|\psi _{0}\in L^{2}(\mathbb{R}^{N})$,
then the function $V(t):=\int_{\mathbb{R}^{N}}|x|^{2}|\psi (t,x)|^{2}dx$
belongs to $C^{2}([0,T_{\max }))$, and the following Virial identities
\begin{equation*}
V^{\prime }(t)=2\mathrm{Re}\int_{\mathbb{R}^{N}}|x|^{2}\bar{\psi}\partial
_{t}\psi dx=4\mathrm{Im}\int_{\mathbb{R}^{N}}\bar{\psi}x\cdot \nabla \psi dx
\end{equation*}%
and%
\begin{equation*}
V^{\prime \prime }(t)=8\left( \int_{\mathbb{R}^{N}}|\nabla \psi |^{2}dx-%
\frac{\gamma (Np-N-\alpha )}{2p}\int_{\mathbb{R}^{N}}(I_{\alpha }\ast |\psi
|^{p})|\psi |^{p}dx-\frac{\mu N(q-2)}{2q}\int_{\mathbb{R}^{N}}|\psi
|^{q}dx\right)
\end{equation*}%
hold for all $0\leq t<T_{\max }$.
\end{theorem}

\begin{remark}
\label{r1}For the non-perturbation case $\mu =0$ and $p\geq 2$, the local
well-posedness of $H^{1}$-solutions of (\ref{e1-0}) was established in \cite%
{AR,GV}. For the perturbation case $\mu \neq 0$ and $p=2$, Fang and Han
obtained the local well-posedness in the energy space $H^{1}(\mathbb{R}^{N})$
(see \cite[Propositions 3.1 and 3.2]{FH}). Theorem \ref{t1} generalizes the
results in \cite{AR,FH,GV}, in particular, where the double energy-critical
case is involved.
\end{remark}

Next, we focus on the existence and multiplicity of standing waves with
prescribed mass.

\begin{theorem}
\label{t4} Let $N\geq 2,\gamma <0$ and $\mu >0$. Assume that one of the two
following conditions holds:\newline
$(a)$ $\overline{p}<p<2_{\alpha }^{\ast },\overline{q}<q<2^{\ast }$ with $%
2Np-Nq-2\alpha \leq 0$ and $|\gamma |c^{\frac{N+\alpha -(N-2)p}{2}}$ is
sufficiently small;\newline
$(b)$ $\overline{p}<p\leq 2_{\alpha }^{\ast },\overline{q}<q<2^{\ast }$ with
$2Np-Nq-2\alpha >0$ and $|\gamma |$ is sufficiently small.\newline
Then $(P_{c})$ has a radially symmetric solution $(\lambda ^{-},u^{-})\in
\mathbb{R}^{+}\times H^{1}(\mathbb{R}^{N})$ satisfying $E_{\gamma ,\mu
}(u^{-})>0.$ Moreover, $u^{-}$ has exponential decay at infinity:
\begin{equation*}
|u^{-}|\leq Ce^{-\delta |x|},\quad \text{for }|x|\geq R,
\end{equation*}%
for some $C>0$, $\delta >0$ and $R>0$.
\end{theorem}

We note that under the assumptions of Theorem \ref{t4}, the functional $%
E_{\gamma ,\mu }$ is unbounded from below on $S(c).$ Then it will not be
possible to find a global minimizer. In order to seek for critical points of
$E_{\gamma ,\mu }$ restricted to $S\left( c\right) ,$ we shall use the
Pohozaev manifold $\mathcal{M}\left( c\right) $ that contains all the
critical points of $E_{\gamma ,\mu }$ restricted to $S(c).$ It is given by%
\begin{equation}
\mathcal{M}(c):=\left\{ u\in S(c)\text{ }|\text{ }Q(u)=0\right\} ,
\label{e1-13}
\end{equation}%
where $Q(u)=0$ is the Pohozaev type identity corresponding to (\ref{e1-1}).
For more details, we refer the reader to Section 2. As we shall see, for
Theorem \ref{t4} $(a)$, the functional $E_{\gamma ,\mu }$ is bounded from
below on $\mathcal{M}\left( c\right) ,$ a natural constraint of $E_{\gamma
,\mu }$. We expect to find critical points with positive level by applying
the minimax method to $E_{\gamma ,\mu }$ restricted to $\mathcal{M}(c)\cap
H_{r}^{1}(\mathbb{R}^{N}),$ introduced by Bartsch and Soave \cite{BS4}.
However, it is unclear if the solution $u^{-}$ obtained by us is a ground
state of $(P_{c})$ since it is difficult to verify that $\inf_{u\in \mathcal{%
M}(c)}E_{\gamma ,\mu }(u)=\inf_{u\in \mathcal{M}(c)\cap H_{r}^{1}(\mathbb{R}%
^{N})}E_{\gamma ,\mu }(u).$

For Theorem \ref{t4} $(b),$ due to the complicated competing effect between
Hartree type nonlinearity and the power one, we can not guarantee that an
arbitrary $u\in S(c)$ can always be projected onto $\mathcal{M}(c)$ for $c>0$%
. Furthermore, it seems difficult to prove that the functional $E_{\gamma
,\mu }$ is bounded from below on $\mathcal{M}(c)$. For these reasons,
inspired by \cite{CJ}, we shall introduce a submanifold of codimension $1$
in $S(c)$ given by%
\begin{equation*}
\Lambda (c):=V\cap \mathcal{M}(c),
\end{equation*}%
where $V$ is an open subset of $S(c)$ defined as (\ref{e4-34}) in Section 4
satisfying the dilation $u_{s}:=s^{N/2}u(sx)\in V$ for any $u\in V$ and $s>0$%
, and then apply the minimax method to $E_{\gamma ,\mu }$ restricted to $%
\Lambda (c)\cap H_{r}^{1}(\mathbb{R}^{N}).$ Since $V$ has a boundary, we
need to make sure that our deformation arguments take place inside $V$ by
imposing on the parameter $\gamma $.

\begin{remark}
\label{r2}In recent years, there has been much attention on the existence of
normalized solution for the following Schr\"{o}dinger equations with van der
Waals type potentials:
\begin{equation*}
\left\{
\begin{array}{ll}
-\Delta u+\lambda u=(|x|^{-\beta _{1}}\ast \left\vert u\right\vert
^{2})u+\gamma (|x|^{-\beta _{2}}\ast \left\vert u\right\vert ^{2})u & \quad
\text{in}\quad \mathbb{R}^{N}, \\
\int_{\mathbb{R}^{N}}|u|^{2}dx=c, &
\end{array}%
\right.
\end{equation*}%
where $\beta _{1}\neq \beta _{2}$ denotes the two-body potentials with
different width. We refer the reader to \cite{CJL,JL5} for $\gamma >0$ and
to \cite{BGH} for $\gamma <0.$ In particular, Bhimani et al. \cite{BGH}
established the existence of ground states for the above problem when $%
\gamma <0$ and $0<\beta _{2}<\beta _{1}<\min \{N,4\}$. However, there is an
unresolved case, i.e., $\gamma <0$ and $2<\beta _{1}<\beta _{2}<\min \{N,4\}$%
. The method of Theorem \ref{t4} $(b)$ can be applied to address this case,
which can be regarded as a supplement to the previous works.
\end{remark}

\begin{theorem}
\label{t6} Let $N\geq 2,\gamma >0$ and $\mu <0.$ Then the following
statements are true.

\begin{itemize}
\item[$(a)$] Assume that one of the four following conditions holds:\newline
$(i)$ $\overline{p}<p<2_{\alpha }^{\ast },\overline{q}<q<2^{\ast }$ with $%
2Np-Nq-\alpha q>0$ and $|\mu |c^{\frac{2N-(N-2)q}{4}}$ is sufficiently small;%
\newline
$(ii)$ $\overline{p}<p<2_{\alpha }^{\ast },\overline{q}<q\leq 2^{\ast }$
with $2Np-Nq-\alpha q\leq 0\leq 2Np-Nq-2\alpha .$\newline
Then $(P_{c})$ has a ground state $(\lambda ^{-},u^{-})\in \mathbb{R}%
^{+}\times H^{1}(\mathbb{R}^{N})$ satisfying $E_{\gamma ,\mu }(u^{-})>0.$ In
particular, such solution $u^{-}$ is radially symmetric.

\item[$(b)$] If $\overline{p}<p<2_{\alpha }^{\ast }$ and $\overline{q}<q\leq
2^{\ast }$ with $2Np-Nq-2\alpha <0,$ then for $c>0$ sufficiently small, $%
(P_{c})$ has two positive radially symmetric solutions $(\lambda ^{\pm
},u^{\pm })\in \mathbb{R}^{+}\times H^{1}(\mathbb{R}^{N})$ satisfying $%
E_{\gamma ,\mu }(u^{+})<0<E_{\gamma ,\mu }(u^{-})$ and $\Vert \nabla
u^{-}\Vert _{2}<\Vert \nabla u^{+}\Vert _{2}.$
\end{itemize}

Moreover, $u^{+}$ and $u^{-}$ have exponential decay at infinity.
\end{theorem}

\begin{remark}
\label{r3}If the response function is a Dirac-delta function, i.e. $%
I_{\alpha }(x)=\delta (x)$, then $(P_{c})$ is reduced to the following NLS
equation with combined power-type nonlinearities:%
\begin{equation}
\left\{
\begin{array}{ll}
-\Delta u+\lambda u=\gamma |u|^{p-2}u+\mu |u|^{q-2}u & \quad \text{in}\quad
\mathbb{R}^{N}, \\
\int_{\mathbb{R}^{N}}|u|^{2}dx=c. &
\end{array}%
\right.  \label{e1-5}
\end{equation}%
When $2<p\leq \overline{q}\leq q<2^{\ast }$ with $p\neq q,\gamma =1$ and $%
\mu \in \mathbb{R}$, the existence and stability/instability of normalized
ground states for (\ref{e1-5}) was studied by Soave \cite{S1}. When $%
2<p<q=2^{\ast },\gamma =1$ and $\mu >0,$ the existence, multiplicity and
stability/instability of normalized solutions for (\ref{e1-5}) was proved in
\cite{JL0,S2,WW}, which can be viewed as a counterpart of the Br\'{e}%
zis-Nirenberg problem in the context of normalized solutions. When $%
\overline{q}<p<q<2^{\ast },\gamma >0$ and $\mu <0,$ two radial normalized
solutions of (\ref{e1-5}) were found for $c>c^{\ast }$ by Jeanjean and Lu as
a special case in \cite[Theorem 1.3 and Remark 1.2]{JL4}. Our main result, Theorem \ref{t6} $(b)$, can
also be applied to the case of competing power-type nonlinearities with both
$L^{2}$-supercritical exponents, i.e. (\ref{e1-5}) with $\overline{q}%
<p<q<2^{\ast },\gamma >0$ and $\mu <0$. We can conclude that there exist two
radial normalized solutions for $c>0$ sufficiently small, which partly
extends the results in \cite{JL4}.
\end{remark}

Contrary to Theorem \ref{t4}, we can apply the minimax method to $E_{\gamma
,\mu }$ restricted to $\mathcal{M}(c)\cap H_{r}^{1}(\mathbb{R}^{N}) $ to
obtain a ground state of $(P_{c})$ with positive level in Theorem \ref{t6} $%
(a)$ by proving the equality $\inf_{u\in \mathcal{M}(c)}E_{\gamma ,\mu
}(u)=\inf_{u\in \mathcal{M}(c)\cap H_{r}^{1}(\mathbb{R}^{N})}E_{\gamma ,\mu
}(u).$ For Theorem \ref{t6} $(b)$, the solution $u^{-}$ with positive level
can be found by adopting the argument of Theorem \ref{t4} $(b)$, via a
replacement of $V$ with another open subset $U$ of $S(c)$ (see (\ref{e5-3})
in Section 4), and the solution $u^{+}$ with negative level can be obtained
as an interior local minimizer of $E_{\gamma ,\mu }$ on the set
\begin{equation*}
\mathbf{A}_{r_{0}}:=\left\{ u\in U\text{ }|\text{ }\Vert \nabla u\Vert
_{2}^{2}<r_{0}\right\} \text{ for some }r_{0}>0\text{ large enough.}
\end{equation*}%
Here we are not sure whether $u^{+}$ is a ground state of $(P_{c})$,
although its energy level is negative. In fact, we can prove that $E_{\gamma
,\mu }$ is bounded from below by a negative constant and that this energy
functional is coercive on $S(c)$ under the assumption of Theorem \ref{t6} $%
(b)$ (see Lemma \ref{L5-6} in Section 4), but the compactness is a delicate
problem due to the inconsistent scaling transformation, which at the moment
we could not solve.

By comparing Theorem \ref{t4} $(b)$ with Theorem \ref{t6} $(b),$ we find
that $(P_{c})$ has different number of solutions, although they are both
competing and double $L^{2}$-supercritical case. It shows that the
power-type nonlinearity is in some sense stronger that the Hartree-type one.
Note that two normalized solutions are found in Theorem \ref{t6} $(b)$. From
the geometry of the functional $E_{\gamma ,\mu }$, we find that the critical
point $u^{-}$ with positive level is closer to the origin than the critical
point $u^{+}$ with negative level.

Finally, we give the dynamic behavior of standing waves for Cauchy problem (%
\ref{e1-0}).

\begin{theorem}
\label{t8} Under the assumptions of Theorem \ref{t6} $(b)$ with $p\geq 2$,
the set of local minimizers $\mathcal{M}_{c}^{r_{0}}:=\left\{ u\in S(c)\cap
\mathbf{A}_{r_{0}}\text{ }|\text{ }E_{\gamma ,\mu }(u)=m^{+}(c)\right\} \neq
\emptyset $ is stable under the flow corresponding to the Cauchy problem (%
\ref{e1-0}). That is to say, for any $\varepsilon >0$, there exists $\delta
>0$ such that for any $\varphi \in H^{1}$ satisfing
\begin{equation*}
dist_{X}(\varphi ,\mathcal{M}_{c}^{r_{0}})<\delta ,
\end{equation*}%
the solution $\psi (t,\cdot )$ of the Cauchy problem (\ref{e1-0}) with $\psi
(0,\cdot )=\varphi $ satisfies
\begin{equation*}
\sup_{t\in \lbrack 0,T)}dist_{X}(\psi (t,\cdot ),\mathcal{M}%
_{c}^{r_{0}})<\varepsilon ,
\end{equation*}%
where $T$ is the maximal existence time for $\psi (t,\cdot ).$
\end{theorem}

We say that a standing wave $e^{i\lambda t}u(x)$ is strongly unstable if for
every $\varepsilon >0$ there exists $\psi _{0}\in H^{1}$ such that $\Vert
u-\psi _{0}\Vert _{H^{1}}<\varepsilon $, and $\psi (t)$ blows up in finite
time.

\begin{theorem}
\label{t7} $(i)$ Under the assumptions in Theorems \ref{t4} $(a)$ or
Theorems \ref{t6} $(a)$, if $\alpha \in (\max \left\{ 0,N-4\right\} ,N)$ and
$p\geq 2$, then the standing waves $\psi (t,x)=e^{i\lambda ^{-}t}u^{-}(x)$
to the Cauchy problem (\ref{e1-0}) is strongly unstable;\newline
$(ii)$ Under the assumptions in Theorems \ref{t4} $(b)$, if $\alpha \in
(\max \left\{ 0,N-4\right\} ,N)$, $p\geq 2$ and $\mu <\overline{\mu }$, then
the standing waves $\psi (t,x)=e^{i\lambda ^{-}t}u^{-}(x)$ to the Cauchy
problem (\ref{e1-0}) is strongly unstable;\newline
$(iii)$ Under the assumptions in Theorems \ref{t6} $(b)$, if $\alpha \in
(\max \left\{ 0,N-4\right\} ,N)$, $p\geq 2$ and $\gamma <\overline{\gamma }$%
, then the standing waves $\psi (t,x)=e^{i\lambda ^{-}t}u^{-}(x)$ to the
Cauchy problem (\ref{e1-0}) is strongly unstable.
\end{theorem}

\begin{remark}
\label{r4}$(i)$ The proofs of the above instability results depend on the
detailed analysis of the corresponding fibering mapping. In Theorems \ref{t7}
$(ii)$ and $(iii)$, the fibering mapping presents a new geometric structure,
which seems to be the first result for this kind of problem.\newline
$(ii)$ For the special case of $p=2,\gamma =1$ and $\mu =\left\{
-1,1\right\} $, Tian and Zhu \cite{TZ} showed the sharp energy threshold for
the global existence and blow-up to the Cauchy problem (\ref{e1-0}) by
utilizing comprehensive potential-well structures. In particular, when $%
\gamma =1$ and $\mu =-1$, they only considered the case of $Nq\leq
4N-2\alpha $. However, for the case of $Nq>4N-2\alpha $, it is more
difficult to obtain the instability/stability of standing waves for the
Cauchy problem (\ref{e1-0}). In Theorem \ref{t7}, we give a positive answer
by a detailed description of the geometric structure corresponding to the
functional.
\end{remark}

Moreover, we obtain the lower bound rate of blow-up solutions for the Cauchy
problem (\ref{e1-0}) as follows.

\begin{theorem}
\label{t10} Let $\alpha \in (\max \{0,N-4\},N)$, $2\leq p<2_{\alpha }^{\ast
} $ and $\overline{q}<q<2^{\ast }$ with $2(Np-N-\alpha )>(N-2)(q-2)p.$ If $%
\psi _{0}\in H^{1}$ and the corresponding solution $\psi (t,x)$ of the
Cauchy problem (\ref{e1-0}) blows up at the finite time $0<T<+\infty $, then
there exists a positive constant $C$ such that for all $0<t<T,$
\begin{equation*}
\Vert \nabla \psi \Vert _{L_{x}^{2}}\gtrsim C(T-t)^{-\frac{N+\alpha -(N-2)p}{%
2(Np-N-\alpha )(p-1)}}.
\end{equation*}
\end{theorem}

The rest of this paper is organized as follows. In Sect. 2, we introduce
some preliminary results. In Sect. 3, we give the proof of the local
well-posedness. In Sect. 4, we study the existence of standing waves and
prove Theorems \ref{t4} and \ref{t6}. Finally, dynamic behavior of standing
waves will be shown in Sect. 5.

\section{Preliminary results}

For the reader's convenience, we set
\begin{equation*}
A(u)=\int_{\mathbb{R}^{N}}|\nabla u|^{2}dx,\text{ }B(u)=\int_{\mathbb{R}%
^{N}}\left( I_{\alpha }\ast |u|^{p}\right) |u|^{p}dx,\text{ }C(u)=\int_{%
\mathbb{R}^{N}}|u|^{q}dx\quad \text{and }D(u)=\int_{\mathbb{R}^{N}}|u|^{2}dx.
\end{equation*}

\subsection{Several important inequalities and a property}

In what follows, we introduce several important inequalities.

\textbf{(1) Gagliardo-Nirenberg inequality of power type (\cite{W}):} Let $%
N\geq 1$ and $r\in (2,2^{\ast }]$. Then there exists a sharp constant $%
\overline{S}:=\overline{S}(N,r)>0$ such that
\begin{equation}
\Vert u\Vert _{r}\leq \overline{S}^{1/r}\Vert \nabla u\Vert _{2}^{\frac{%
N(r-2)}{2r}}\Vert u\Vert _{2}^{\frac{2N-r(N-2)}{2r}}.  \label{e2-1}
\end{equation}

\textbf{(2) Gagliardo-Nirenberg inequality of Hartree type (\cite{MS,Y}):}
Let $N\geq 1$ and $2_{\alpha}\leq p\leq2_{\alpha }^{\ast }$. Then there
exists a best constant $\overline{B}:=\overline{B}(N,\alpha,p)>0$ such that
\begin{equation}
\int_{\mathbb{R}^{N}}(I_{\alpha }\ast |u|^{p})|u|^{p}dx\leq \overline{B}%
\Vert \nabla u\Vert _{2}^{Np-N-\alpha }\Vert u\Vert _{2}^{N+\alpha -p(N-2)}.
\label{e2-3}
\end{equation}

\textbf{(3) Hardy-Littlewood-Sobolev inequality (\cite{L}):} For $0<\alpha<N$
and $1<p, q<\infty$, there exists a sharp constant $C(N,p,\alpha)>0$ such
that
\begin{equation}  \label{e2-4}
\left\|\int_{\mathbb{R}^{N}}\frac{u(y)}{|x-y|^{N-\alpha}}dy\right\|_{L^{q}}%
\leq C(N,p,\alpha)\|u\|_{L^{p}},
\end{equation}
where $\frac{1}{q}=\frac{1}{p}-\frac{\alpha}{N}$ and $p<\frac{N}{\alpha}$.

\begin{corollary}
\label{R1} By the semi-group identity for the Riesz potential \cite{L9}, for
$p\in [\frac{N+\alpha}{N}, \frac{N+\alpha}{N-2}] $, there exists $\widetilde{%
C}(N,\alpha)=\pi^{\frac{N-\alpha}{2}}\frac{\Gamma(\frac{\alpha}{2})} {%
\Gamma(N-\frac{N-\alpha}{2})}\left[\frac{\Gamma(\frac{N}{2})}{\Gamma(N)}%
\right]^{-\frac{\alpha}{N}}A(N,\alpha)>0$ such that
\begin{eqnarray*}
\int_{\mathbb{R}^{N}}( I_{\alpha }\ast | u| ^{p}) | u| ^{p}dx=\int_{\mathbb{R%
}^{N}}\left\vert I_{\alpha/2}\ast| u|^{p}\right\vert^{2}dx \leq \widetilde{C}%
\left( \int_{\mathbb{R}^{N}}| u| ^{\frac{2Np}{N+\alpha }}dx\right) ^{\frac{%
N+\alpha }{N}}.  \label{e2-9}
\end{eqnarray*}
\end{corollary}

Next, we show a splitting property for the functional $B$ and its derivative
$B^{\prime }$, which is similar to the Br\'{e}zis-Lieb type Lemma for
nonlocal nonlinearities \cite{A,AN}.

\begin{lemma}
\label{L2-4} Let $\{u_{n}\}$ be a bounded sequence in $H^{1}(\mathbb{R}^{N})$%
. If $u_{n}\rightarrow u$ a.e. in $\mathbb{R}^{N}$, then\newline
$(i)$ $B(u_{n}-u)=B(u_{n})-B(u)+o(1);$\, $(ii)$ $B^{\prime }\left(
u_{n}-u\right) =B^{\prime }(u_{n})-B^{\prime }(u)+o(1)$ in $H^{-1}(\mathbb{R}%
^{N}).$
\end{lemma}

\subsection{Decomposition of Pohozaev manifold $\mathcal{M}(c)$}

\begin{lemma}
\label{L2-1} Let $u$ be a weak solution to (\ref{e1-1}). Then $u$ satisfies
the Pohozaev identity
\begin{equation*}
Q(u):=A(u)-\frac{\gamma (Np-N-\alpha )}{2p}B(u)-\frac{\mu N(q-2)}{2q}C(u)=0.
\end{equation*}
\end{lemma}

\begin{proof}
The proof is similar to that of \cite[Proposition 3.1]{MS}, we omit it here.
\end{proof}

Following the idea of Soave \cite{S1} and Cingolani and Jeanjean \cite{CJ},
we will introduce a natural constraint manifold $\mathcal{M}(c)$ that
contains all the critical points of the functional $E_{\gamma ,\mu }$
restricted to $S(c)$. For each $u\in H^{1}(\mathbb{R}^{N})\backslash \{0\}$
and $t>0$, we set $u_{s}(x):=s^{N/2}u(sx)$. Then a direct calculation shows
that $\Vert u_{s}\Vert _{2}^{2}=\Vert u\Vert _{2}^{2},$ $A(u_{s})=s^{2}A(u),$
$B(u_{t})=s^{Np-N-\alpha }B(u)\quad $and $C(u_{s})=s^{N(q-2)/2}C(u).$ Define
the fibering map $s\in (0,\infty )\mapsto g_{u}(s):=E_{\gamma ,\mu }(u_{s})$
given by
\begin{equation*}
g_{u}(s)=\frac{s^{2}}{2}A(u)-\frac{\gamma s^{Np-N-\alpha }}{2p}B(u)-\frac{%
\mu s^{N(q-2)/2}}{q}C(u).  \label{e2-10}
\end{equation*}%
By calculating the first and second derivatives of $g_{u}(s),$ we have $%
\frac{d}{ds}E_{\gamma ,\mu }(u_{s})=g_{u}^{\prime }(s)=\frac{Q(u_{s})}{s},$.
Then by (\ref{e1-13}) one has
\begin{equation}
\mathcal{M}(c)=\left\{ u\in S(c)\text{ }|\text{ }Q(u)=0\right\} =\left\{
u\in S(c)\text{ }|\text{ }g_{u}^{\prime }(1)=0\right\} ,  \label{e2-19}
\end{equation}%
which appears as a natural constraint. We also recognize that for any $u\in
S(c)$, the dilated function $u_{s}(x)$ belongs to the Pohozaev manifold $%
\mathcal{M}(c)$ if and only if $s\in \mathbb{R}$ is a critical value of the
fibering map $s\in (0,\infty )\mapsto g_{u}(s)$, namely, $g_{u}^{\prime
}(s)=0$. Thus, it is natural split $\mathcal{M}(c)$ into three parts
corresponding to local minima, local maxima and points of inflection.
Following \cite{T}, we define
\begin{eqnarray*}
\mathcal{M}^{+}(c) &:&=\left\{ u\in S(c)\text{ }|\text{ }g_{u}^{\prime
}(1)=0,g_{u}^{\prime \prime }(1)>0\right\} ; \\
\mathcal{M}^{-}(c) &:&=\left\{ u\in S(c)\text{ }|\text{ }g_{u}^{\prime
}(1)=0,g_{u}^{\prime \prime }(1)<0\right\} ; \\
\mathcal{M}^{0}(c) &:&=\left\{ u\in S(c)\text{ }|\text{ }g_{u}^{\prime
}(1)=0,g_{u}^{\prime \prime }(1)=0\right\} .
\end{eqnarray*}%
Then it holds
\begin{eqnarray}
g_{u}^{\prime \prime }(1) &=&(N+\alpha +2-Np)A(u)-\frac{\mu
N(q-2)(Nq-2Np+2\alpha )}{4q}C(u)  \label{e2-7} \\
&=&-\frac{N(q-2)-4}{2}A(u)-\frac{\gamma (Np-N-\alpha )(2Np-qN-2\alpha )}{4p}%
B(u).  \label{e2-8}
\end{eqnarray}%
Furthermore, following the argument of Soave \cite{S1}, we have the
following lemma.

\begin{lemma}
\label{L2-2} If $\mathcal{M}^{0}(c)=\emptyset $, then $\mathcal{M}(c)$ is a
smooth submanifold of codimension $2$ of $H^{1}(\mathbb{R}^{N})$ and a
submanifold of codimension $1$ in $S(c).$
\end{lemma}

\subsection{A general minimax theorem}

In this subsection, we shall give a general minimax theorem based on the
homotopy stable family to establish the existence of a Palais--Smale
sequence.

\begin{definition}
\label{D2-1}\cite[Definition 3.1]{G} Let $\Theta $ be a closed subset of a
metric space $X\subset H^{1}(\mathbb{R}^{N})$. We say that a class $\mathcal{%
F}$ of compact subsets of $X$ is a homotopy-stable family with closed
boundary $\mathcal{B}$ provided that

\begin{itemize}
\item[$(a)$] every set in $\mathcal{F}$ contains $\Theta$;

\item[$(b)$] for any set $H\in\mathcal{F}$ and any $\eta\in C([0,1]\times
X,X)$ satisfying $\eta(s,x)=x$ for all $(s,x)\in (\{0\}\times
X)\cup([0,1]\times \Theta)$, we have that $\eta(\{1\}\times H)\in\mathcal{F}$%
.
\end{itemize}
\end{definition}

\begin{lemma}
\label{L2-6}\cite[Theorem 3.2]{G} Let $\varphi $ be a $C^{1}$-functional on
a complete connected $C^{1}$-Finsler manifold $X$ (without boundary) and
consider a homotopy stable family $\mathcal{F}$ of compact subsets of $X$
with a closed boundary $\mathcal{B}$. Set
\begin{equation*}
\bar{c}=\bar{c}(\varphi ,\mathcal{F})=\inf_{H\in \mathcal{F}}\max_{u\in
H}\varphi (u)
\end{equation*}%
and suppose that $\sup \varphi (\Theta )<\bar{c}.$ Then for any sequence of
sets $\{H_{n}\}$ in $\mathcal{F}$ such that $\lim_{n\rightarrow \infty
}\sup_{H_{n}}\varphi =\bar{c}$, there exists a sequence $\{u_{n}\}$ in $X$
such that\newline
$(i)$ $\lim_{n\rightarrow \infty }\varphi (u_{n})=\bar{c};$ $(ii)$ $%
\lim_{n\rightarrow \infty }\left\Vert \varphi ^{\prime }(u_{n})\right\Vert
=0;$ $(iii)$ $\lim_{n\rightarrow \infty }\text{dist}(u_{n},H_{n})=0.$\newline
Furthermore, if $\varphi^{\prime} $ is uniformly continuous, then $u_{n}$
can be chosen to be in $H_{n}$ for each $n$.
\end{lemma}

\subsection{Compactness lemma for Palais-Smale sequences}

\begin{lemma}
\label{L2-9} Let $\{u_{n}\}\subset \mathcal{M}(c)\cap H_{r}^{1}(\mathbb{R}%
^{N})$ be a bounded Palais-Smale sequence for $E_{\gamma ,\mu }(u)|_{S(c)}$
at level $\overline{\beta }\neq 0$. Then, up to a subsequence, $%
u_{n}\rightarrow u$ strongly in $H_{r}^{1}(\mathbb{R}^{N})$, and $u\in S(c)$
is a solution of $(P_{c})$ for some $\lambda >0$.
\end{lemma}

\begin{proof}
By the principle of symmetric criticality \cite[Theorem 1.28]{W1}, the
solutions in radially symmetric function space $H_{r}^{1}(\mathbb{R}^{N})$
are also those in the space $H^{1}(\mathbb{R}^{N})$.

\textbf{Step 1. The boundedness of $\lambda_{n}$.} Since $\left\{
u_{n}\right\} \subset \mathcal{M}(c)$ is bounded and the embedding $%
H_{r}^{1}(\mathbb{R}^{N})\hookrightarrow L^{r}(\mathbb{R}^{N})$ is compact
for $r\in (2,2^{\ast })$, there exists $u\in H_{r}^{1}(\mathbb{R}^{N})$ such
that $u_{n}\rightharpoonup u$ weakly in $H_{r}^{1}(\mathbb{R}^{N})$, $%
u_{n}\rightarrow u$ strongly in $L^{s}(\mathbb{R}^{N})$ for $s\in (2,2^{\ast
})$, and a.e. in $\mathbb{R}^{N}$. By the Lagrange multipliers rule, there
exists $\lambda _{n}\in \mathbb{R}$ such that for every $\varphi \in H^{1}(%
\mathbb{R}^{N})$,
\begin{equation}
\int_{\mathbb{R}^{N}}\nabla u_{n}\nabla \varphi dx+\lambda _{n}\int_{\mathbb{%
R}^{N}}u_{n}\varphi dx-\gamma \iint\limits_{\mathbb{R}^{N}\times \mathbb{R}%
^{N}}\frac{|u_{n}|^{2p-2}u_{n}}{|x-y|^{N-\alpha }}\varphi dxdy-\mu \int_{%
\mathbb{R}^{N}}|u_{n}|^{q-2}u_{n}\varphi dx=o(1)\Vert \varphi \Vert .
\label{e2-11}
\end{equation}%
In particular, we have
\begin{equation*}
\lambda _{n}c=-A(u_{n})+\gamma B(u_{n})+\mu C(u_{n})+o(1),
\end{equation*}%
which implies that $\{\lambda _{n}\}$ is bounded, and so we can assume that $%
\lambda _{n}\rightarrow \lambda \in \mathbb{R}$. Since $Q(u_{n})=0$, we have
\begin{equation*}
\lambda _{n}c=\gamma \left( 1-\frac{Np-N-\alpha }{2p}\right) B(u_{n})+\mu
\left( 1-\frac{N(q-2)}{2q}\right) C(u_{n})+o(1).  \label{e2-14}
\end{equation*}

\textbf{Step 2. The sign of $\lambda$.} We now divide the arguments into
several cases as follows.

Case $(a):\gamma <0,\mu >0, \overline{p}<p\leq2_{\alpha }^{\ast }$ and $%
\overline{q}<q<2^{\ast }$. Since $Q(u_{n})=0$, we have
\begin{equation*}
A(u_{n})=\frac{\gamma (Np-N-\alpha )}{2p}B(u_{n})+\frac{\mu N(q-2)}{2q}%
C(u_{n})\leq \frac{\mu N(q-2)}{2q}C(u_{n}).
\end{equation*}%
Using this, together with (\ref{e2-1}) and the fact of $u_{n}\rightarrow u$
strongly in $L^{r}(\mathbb{R}^{N})$ for $r\in (2,2^{\ast })$, gives%
\begin{eqnarray}
A(u)\leq \liminf_{n\rightarrow\infty}A(u_{n})&\leq&\frac{\mu N(q-2)}{2q}%
\liminf_{n\rightarrow\infty}C(u_{n})=\frac{\mu N(q-2)}{2q}C(u)  \notag \\
&\leq& \frac{\mu N(q-2)\overline{S}}{2q}\Vert u\Vert _{2}^{\frac{2N-q(N-2)}{2%
}}A(u)^{\frac{N(q-2)}{4}}.  \label{e2-13}
\end{eqnarray}%
Since $E_{\gamma ,\mu }(u_{n})\rightarrow \overline{\beta }\neq 0$, we have $%
u\neq 0.$ Hence, using (\ref{e2-13}) and the fact of $\Vert u\Vert
_{2}^{2}\leq c$ obtained by the weak lower semi-continuity, we deduce that
\begin{equation}
A(u)\geq \left( \frac{2q}{\mu N(q-2)\overline{S}c^{\frac{2N-q(N-2)}{4}}}%
\right) ^{\frac{4}{N(q-2)-4}}.  \label{e2-15}
\end{equation}

Since $\lambda _{n}\rightarrow \lambda $ and $u_{n}\rightarrow u\neq 0$
weakly in $H_{r}^{1}(\mathbb{R}^{N})$, and strongly in $L^{q}(\mathbb{R}%
^{N}) $, (\ref{e2-11}) implies that $u$ is a weak solution to (\ref{e1-1}).
By Lemma \ref{L2-1}, we infer that $Q(u)=0.$ Moreover, we have
\begin{eqnarray}
\lambda \Vert u\Vert _{2}^{2} &\geq &\frac{2N-q(N-2)}{N(q-2)}A(u)+\frac{%
\gamma \overline{B}(Nq+\alpha q-2Np)}{Np(q-2)}c^{\frac{N+\alpha -p(N-2)}{2}%
}A(u)^{\frac{Np-N-\alpha }{2}}  \notag \\
&\geq &A(u)^{\frac{Np-N-\alpha }{2}}\left[ \frac{2N-q(N-2)}{N(q-2)}\left(
\frac{2q}{\mu N(q-2)\overline{S}c^{\frac{2N-(N-2)q}{4}}}\right) ^{\frac{%
2(N+\alpha +2-Np)}{(q-2)N-4}}\right.  \notag \\
&&\left. +\frac{\gamma \overline{B}(Nq+\alpha q-2Np)}{Np(q-2)c^{-\frac{%
N+\alpha -p(N-2)}{2}}}\right] .  \label{e2-16}
\end{eqnarray}%
Case $(a-i):Nq+\alpha q-2Np\leq 0$. It is clear that $\lambda >0,$ by (\ref%
{e2-15})-(\ref{e2-16}).\newline
Case $(a-ii):$ $Nq+\alpha q-2Np>0$. Since $\{u_{n}\}$ is a bounded sequence
in $H^{1}(\mathbb{R}^{N})$, there exists a constant $D_{1}>0$ such that $%
A(u)\leq \lim_{n\rightarrow \infty }A(u_{n})\leq D_{1}$. Let us consider
\begin{equation*}
f(s):=\frac{2N-q(N-2)}{N(q-2)}s+\frac{\gamma \overline{B}(Nq+\alpha q-2Np)}{%
Np(q-2)}c^{\frac{N+\alpha -p(N-2)}{2}}s^{\frac{Np-N-\alpha }{2}}\text{ for }%
s>0.
\end{equation*}%
A direct calculation shows that $f(A(u))>0$ for
\begin{equation*}
0<A(u)<\left( \frac{(2N-q(N-2))p}{|\gamma |\overline{B}(Nq+\alpha q-2Np)c^{%
\frac{N+\alpha -p(N-2)}{2}}}\right) ^{\frac{2}{Np-N-\alpha -2}}.
\end{equation*}%
Taking $|\gamma |c^{\frac{N+\alpha -p(N-2)}{2}}$ small enough such that
\begin{equation*}
\left( \frac{(2N-q(N-2))p}{|\gamma \overline{B}|(Nq+\alpha q-2Np)c^{\frac{%
N+\alpha -p(N-2)}{2}}}\right) ^{\frac{2}{Np-N-\alpha -2}}>D_{1}>\left( \frac{%
2q}{\mu N(q-2)\overline{S}c^{\frac{2N-q(N-2)}{4}}}\right) ^{\frac{4}{N(q-2)-4%
}}.
\end{equation*}%
Thus, one has
\begin{equation*}
\lambda \Vert u\Vert _{2}^{2}\geq \frac{2N-q(N-2)}{N(q-2)}A(u)+\frac{\gamma
\overline{B}(Nq+\alpha q-2Np)}{Np(q-2)}c^{\frac{N+\alpha -p(N-2)}{2}}A(u)^{%
\frac{Np-N-\alpha }{2}}>0,
\end{equation*}%
which shows that $\lambda >0$.

Case $(b):\gamma >0,\mu <0.$ Since $Q(u_{n})=o(1)$, it holds
\begin{equation*}
A(u_{n})\leq \frac{\gamma (Np-N-\alpha )}{2p}B(u_{n})+o(1).
\end{equation*}%
Using this, together with (\ref{e2-3}) and the fact of $B(u_{n})\rightarrow
B(u)$ as $n\rightarrow \infty $, leads to
\begin{equation*}
A(u)\leq \frac{\gamma (Np-N-\alpha )}{2p}B(u)\leq \frac{\gamma \overline{B}%
(Np-N-\alpha )}{2p}\Vert u\Vert _{2}^{N+\alpha -(N-2)p}A(u)^{\frac{%
Np-N-\alpha }{2}}.
\end{equation*}%
As in Case $(a)$, we have $u\neq 0.$ Thus, using that $\Vert u\Vert
_{2}^{2}\leq c$ obtained by the weak lower semi-continuity, we deduce that
\begin{equation}
A(u)\geq \left( \frac{2p}{\gamma \overline{B}(Np-N-\alpha )c^{\frac{N+\alpha
-(N-2)p}{2}}}\right) ^{\frac{2}{Np-N-\alpha -2}}.  \label{e2-17}
\end{equation}%
Since $\lambda _{n}\rightarrow \lambda $ and $u_{n}\rightarrow u\neq 0$
weakly in $H_{r}^{1}(\mathbb{R}^{N})$, and strongly in $L^{q}(\mathbb{R}%
^{N}) $, (\ref{e2-11}) indicates that $u$ is a weak solution to (\ref{e1-1}%
). It follows from Lemma \ref{L2-1} that $Q(u)=0$ and
\begin{eqnarray}
\lambda \Vert u\Vert _{2}^{2} &\geq &\frac{N+\alpha -p(N-2)}{Np-N-\alpha }%
A(u)+\frac{\mu \overline{S}(2Np-Nq-\alpha q)}{q(Np-N-\alpha )}c^{\frac{%
2N-q(N-2)}{4}}A(u)^{\frac{N(q-2)}{4}}  \notag \\
&\geq &A(u)^{\frac{N(q-2)}{4}}\left[ \frac{N+\alpha -p(N-2)}{Np-N-\alpha }%
\left( \frac{2p}{\gamma \overline{B}(Np-N-\alpha )c^{\frac{N+\alpha -p(N-2)}{%
2}}}\right) ^{\frac{4-N(q-2)}{2(Np-N-\alpha -2)}}\right.  \notag \\
&&\left. +\frac{\mu \overline{S}(2Np-Nq-\alpha q)}{q(Np-N-\alpha )c^{-\frac{%
2N-q(N-2)}{4}}}\right] .  \label{e2-18}
\end{eqnarray}%
\newline
Case $(b-i):2Np-Nq-\alpha q\leq 0.$ By (\ref{e2-17})-(\ref{e2-18}) it is
easy to get $\lambda >0$.\newline
Case $(b-ii):2Np-Nq-\alpha q>0$. Note that $\{u_{n}\}$ is a bounded sequence
in $H^{1}(\mathbb{R}^{N})$. Then there exists a constant $D_{2}>0$ such that
$A(u)\leq \lim_{n\rightarrow \infty }A(u_{n})\leq D_{2}$. Now, we consider
the function
\begin{equation*}
f(s):=\frac{N+\alpha -p(N-2)}{Np-N-\alpha }s+\frac{\mu \overline{S}%
(2Np-Nq-\alpha q)}{q(Np-N-\alpha )}c^{\frac{2N-q(N-2)}{4}}s^{\frac{N(q-2)}{4}%
}\text{ for }s>0.
\end{equation*}%
It is not difficult to verify that $f(A(u))>0$ for
\begin{equation*}
0<A(u)<\left( \frac{q(N+\alpha -p(N-2))}{|\mu |\overline{S}(2Np-Nq-\alpha
q)c^{\frac{2N-q(N-2)}{4}}}\right) ^{\frac{4}{N(q-2)-4}}.
\end{equation*}%
Let us take $|\mu |c^{\frac{2N-(N-2)q}{4}}$ small enough such that
\begin{equation*}
\left( \frac{q(N+\alpha -p(N-2))}{|\mu |\overline{S}(2Np-Nq-\alpha q)c^{%
\frac{2N-(N-2)q}{4}}}\right) ^{\frac{4}{N(q-2)-4}}>D_{2}>\left( \frac{2\Vert
Q_{p}\Vert _{2}^{2p-2}}{\gamma (Np-N-\alpha )c^{\frac{N+\alpha -(N-2)p}{2}}}%
\right) ^{\frac{2}{Np-N-\alpha -2}}.
\end{equation*}%
Thus, one has
\begin{equation*}
\lambda \Vert u\Vert _{2}^{2}\geq \frac{N+\alpha -p(N-2)}{Np-N-\alpha }A(u)+%
\frac{\mu \overline{S}(2Np-Nq-\alpha q)}{q(Np-N-\alpha )}c^{\frac{2N-q(N-2)}{%
4}}A(u)^{\frac{N(q-2)}{4}}>0,
\end{equation*}%
which indicates that $\lambda >0$.

\textbf{Step 3. $H_{r}^{1}$-strong convergence.} By (\ref{e2-11}), we have
\begin{equation}
E_{\gamma ,\mu }^{\prime }( u) [\varphi ]+\lambda \int_{\mathbb{R}%
^{N}}u\varphi dx=0  \label{e2-12}
\end{equation}
for every $\varphi \in H^{1}(\mathbb{R}^{N})$. Taking $\varphi =u_{n}-u$ in (%
\ref{e2-11}) and (\ref{e2-12}), we get
\begin{equation*}
( E_{\gamma ,\mu }^{\prime }( u_{n}) -E_{\gamma ,\mu }^{\prime }( u) )
[u_{n}-u]+\lambda \int_{\mathbb{R}^{N}}| u_{n}-u| ^{2}dx=o( 1) .
\end{equation*}%
Since $u_{n}\rightarrow u$ in $L^{r}(\mathbb{R}^{N})$ for $r\in ( 2,2^{\ast
}) $, we infer that $B( u_{n}-u) =o(1)$ for $p\in
(2_{\alpha},2^{*}_{\alpha}) $ and $C( u_{n}-u) =o(1)$ for $q\in(2,2^{*})$.
Then we obtain that
\begin{equation*}
\int_{\mathbb{R}^{N}}\left(| \nabla ( u_{n}-u) | ^{2}+\lambda| u_{n}-u|
^{2}\right) dx=o( 1) ,
\end{equation*}%
which implies that $u_{n}\rightarrow u$ in $H_{r}^{1}(\mathbb{R}^{N}),$
since $\lambda >0.$

In addition, for the case of $\gamma <0,\mu >0$ and $2_{\alpha }^{\ast }$,
we have
\begin{equation*}
\int_{\mathbb{R}^{N}}\left( |\nabla (u_{n}-u)|^{2}+\lambda
|u_{n}-u|^{2}\right) dx+|\gamma |B(u_{n}-u)=o(1).
\end{equation*}%
This implies that $u_{n}\rightarrow u$ in $H_{r}^{1}(\mathbb{R}^{N}),$ since
$\lambda >0.$ Similarly, for the case of $\gamma >0,\mu <0$ and $q=2^{\ast }$%
, we can also obtain $u_{n}\rightarrow u$ in $H_{r}^{1}(\mathbb{R}^{N})$.
Consequently, the proof is complete.
\end{proof}

\section{Local well-posedness for the Cauchy problem}

Denoting $g:\mathbb{C}\rightarrow \mathbb{C}$ by $g(\psi )=g_{1}(\psi
)+g_{2}(\psi )$, where $g_{1}(\psi )=\gamma (I_{\alpha }\ast |\psi
|^{p})|\psi |^{p-2}\psi $ and $g_{2}(\psi )=\mu |\psi |^{q-2}\psi $. (\ref%
{e1-0}) reads as
\begin{equation}
\left\{
\begin{array}{l}
i\partial _{t}\psi +\Delta \psi +g(\psi )=0, \\
\psi (0,x)=\psi _{0}.%
\end{array}%
\right.  \label{e3-1}
\end{equation}

\begin{definition}
\label{D3.1} The pair $(s,r)$ is referred to be as an Strichartz admissible
if
\begin{equation*}
\frac{2}{s}+\frac{N}{r}=\frac{N}{2}\quad \text{for }s,r\in \lbrack 2,\infty ]%
\text{ and }(s,r,N)\neq (2,\infty ,2).
\end{equation*}
\end{definition}

We introduce the spaces $L_{t}^{s}([0,T);L_{x}^{r})$ and $%
L_{t}^{s}([0,T);W_{x}^{1,r})$ (as $L_{t}^{s}L_{x}^{r}$ and $%
L_{t}^{s}W_{x}^{1,r}$ respectively to simplify the notation) equipped with
the Strichartz norms:
\begin{equation*}
\Vert w(t,x)\Vert _{L_{t}^{s}L_{x}^{r}}=\left( \int_{0}^{T}\Vert w(t,\cdot
)\Vert _{r}^{s}dt\right) ^{\frac{1}{s}}\text{ and }\Vert w(t,x)\Vert
_{L_{t}^{s}W_{x}^{1,r}}=\left( \int_{0}^{T}\Vert w(t,\cdot )\Vert _{W^{1,r}(%
\mathbb{R}^{N})}^{s}dt\right) ^{\frac{1}{s}},
\end{equation*}%
where for the function $w(t,x)$ defined on the time-space strip $[0,T)\times
\mathbb{R}^{N}$.

\begin{definition}
\label{D3.2}Let $T>0$. We say that $\psi (t,x)$ is an integral solution of
the Cauchy problem (\ref{e3-1}) on the time interval $[0,T)$ if\newline
$(i)$ $\psi \in \mathcal{C}([0,T);H_{x}^{1})\cap
L_{t}^{s}([0,T);W_{x}^{1,r});$\newline
$(ii)$ for all $t\in [0,T)$ it holds $\psi (t)=e^{it\Delta }\psi_{0}
+i\int_{0}^{t}e^{i(t-\rho)\Delta }g(\psi (\rho))d\rho.$
\end{definition}

Let us recall the following well-known Strichartz's estimates that will be
useful in the next proof, see \cite{C,KT}.

\begin{proposition}
\label{P3.4} For every admissible pairs $(s,r)$ and $(\tilde{s},\tilde{r})$,
there exists a constant $C>0$ such that for every $T>0$, the following
properties hold:\newline
$(i)$ For every $\varphi \in L_{x}^{2}(\mathbb{R}^{N})$, the function $%
t\mapsto e^{it\Delta }\varphi $ belongs to $L_{t}^{s}([0,T);L_{x}^{r})\cap
\mathcal{C(}[0,T);L_{x}^{2}(\mathbb{R}^{N}))$ and
\begin{equation*}
\left\Vert e^{it\Delta }\varphi \right\Vert _{L_{t}^{s}L_{x}^{r}}\leq C\Vert
\varphi \Vert _{L_{x}^{2}}.
\end{equation*}%
$(ii)$ Let $F\in L_{t}^{\tilde{s}^{\prime }}([0,T);L_{x}^{\tilde{r}^{\prime
}})$, where we use a prime to denote conjugate indices. Then the function
\begin{equation*}
t\mapsto \Phi _{F}(t):=\int_{0}^{t}e^{i(t-\rho)\Delta }F(\rho)d\rho
\end{equation*}%
belongs to $L_{t}^{s}([0,T);L_{x}^{r})\cap \mathcal{C}([0,T);L_{x}^{2}(%
\mathbb{R}^{N}))$ and
\begin{equation*}
\left\Vert \Phi _{F}\right\Vert _{L_{t}^{s}L_{x}^{r}}\leq C\Vert F\Vert
_{L_{t}^{\tilde{s}^{\prime }}L_{x}^{\tilde{r}^{\prime }}}.
\end{equation*}%
$(iii)$ For every $\varphi \in H_{x}^{1}(\mathbb{R}^{N})$, the function $%
t\mapsto e^{it\Delta }\varphi $ belongs to $L_{t}^{s}([0,T);W_{x}^{1,r})\cap
\mathcal{C}([0,T);L_{x}^{2}(\mathbb{R}^{N}))$ and
\begin{equation*}
\left\Vert e^{it\Delta }\varphi \right\Vert _{L_{t}^{s}W_{x}^{1,r}}\leq
C\Vert \varphi \Vert _{H_{x}^{1}}.
\end{equation*}
\end{proposition}

We choose the $L^{2}$-admissible pair $(s_{1},r_{1})$ and $(s_{2},r_{2})$
defined as follows
\begin{equation*}
(s_{1},r_{1})=\left( \frac{4p}{Np-N-\alpha },\frac{2Np}{N+\alpha }\right)
\text{ and }(s_{2},r_{2})=\left( \frac{4q}{(q-2)N},q\right) .
\end{equation*}%
Let $\Psi (\psi (t))=\Psi _{1}(\psi (t))+\Psi _{2}(\psi (t))$, where
\begin{equation*}
\Psi _{1}(\psi (t)):=\frac{1}{2}e^{it\Delta }\psi_{0}
+i\int_{0}^{t}e^{i(t-\rho )\Delta }g_{1}(\psi (\rho ))d\rho \text{ and }\Psi
_{2}(\psi (t)):=\frac{1}{2}e^{it\Delta }\psi_{0} +i\int_{0}^{t}e^{i(t-\rho
)\Delta }g_{2}(\psi (\rho ))d\rho .
\end{equation*}

\begin{lemma}
\label{L3-4} Assume that $\alpha \in (\max \{0,N-4\},N)$, $2\leq p\leq
2_{\alpha }^{\ast }$ and $2<q\leq 2^{\ast }$. Then for every admissible
couple $(\tilde{s},\tilde{r})$ there exists a constant $C>0$ such that for
every $T>0$ the following estimates hold:
\begin{eqnarray}
&&\Vert \Psi _{1}(\psi (t))-\Psi _{1}(v(t))\Vert _{L_{t}^{\tilde{s}}L_{x}^{%
\tilde{r}}}  \notag \\
&\leq &CT^{\theta }[\Vert \psi \Vert _{L_{t}^{\infty }H_{x}^{1}}^{p-1}\Vert
\psi \Vert _{L_{t}^{s_{1}}L_{x}^{r_{1}}}(\Vert \psi \Vert _{L_{t}^{\infty
}H_{x}^{1}}^{p-2}+\Vert v\Vert _{L_{t}^{\infty }H_{x}^{1}}^{p-2})\Vert \psi
-v\Vert _{L_{t}^{\infty }H_{x}^{1}}  \notag \\
&&+(\Vert \psi \Vert _{L_{t}^{\infty }H_{x}^{1}}^{p-1}+\Vert v\Vert
_{L_{t}^{\infty }H_{x}^{1}}^{p-1})\Vert \psi -v\Vert
_{L_{t}^{s_{1}}L_{x}^{r_{1}}}\Vert v\Vert _{L_{t}^{\infty }H_{x}^{1}}],
\label{e3-2}
\end{eqnarray}%
and%
\begin{eqnarray}
&&\Vert \nabla \lbrack \Psi _{1}(\psi (t))-\Psi _{1}(v(t))]\Vert _{L_{t}^{%
\tilde{s}}L_{x}^{\tilde{r}}}  \notag \\
&\leq &CT^{\theta }[\Vert \nabla \psi \Vert
_{L_{t}^{s_{1}}L_{x}^{r_{1}}}\Vert \psi \Vert _{L_{t}^{\infty
}H_{x}^{1}}^{p-1}(\Vert \psi \Vert _{L_{t}^{\infty }H_{x}^{1}}^{p-2}+\Vert
v\Vert _{L_{t}^{\infty }H_{x}^{1}}^{p-2})\Vert \psi -v\Vert _{L_{t}^{\infty
}H_{x}^{1}}  \notag \\
&&+\Vert \psi \Vert _{L_{t}^{\infty }H_{x}^{1}}^{p}(\Vert \psi \Vert
_{L_{t}^{\infty }H_{x}^{1}}^{p-2}+\Vert v\Vert _{L_{t}^{\infty
}H_{x}^{1}}^{p-2})\Vert \nabla (\psi -v)\Vert _{L_{t}^{s_{1}}L_{x}^{r_{1}}}
\notag \\
&&+(\Vert \psi \Vert _{L_{t}^{\infty }H_{x}^{1}}^{p-1}+\Vert v\Vert
_{L_{t}^{\infty }H_{x}^{1}}^{p-1})(\Vert \nabla (\psi -v)\Vert
_{L_{t}^{s_{1}}L_{x}^{r_{1}}}\Vert v\Vert _{L_{t}^{\infty }H_{x}^{1}}^{p-1}
\notag \\
&&+\Vert \psi -v\Vert _{L_{t}^{\infty }H_{x}^{1}}\Vert v\Vert
_{L_{t}^{\infty }H_{x}^{1}}^{p-2}\Vert \nabla v\Vert
_{L_{t}^{s_{1}}L_{x}^{r_{1}}})],  \label{e3-3}
\end{eqnarray}%
and%
\begin{equation}
\Vert \Psi _{2}(\psi (t))-\Psi _{2}(v(t))\Vert _{L_{t}^{\tilde{s}}L_{x}^{%
\tilde{r}}}\leq CT^{\delta }\left[ (\Vert \psi \Vert _{L_{t}^{\infty
}H_{x}^{1}}^{q-2}+\Vert v\Vert _{L_{t}^{\infty }H_{x}^{1}}^{q-2})\Vert \psi
-v\Vert _{L_{t}^{s_{2}}L_{x}^{r_{2}}}\right] ,  \label{e3-4}
\end{equation}%
and
\begin{equation}
\Vert \nabla \lbrack \Psi _{2}(\psi (t))-\Psi _{2}(v(t))]\Vert _{L_{t}^{%
\tilde{s}}L_{x}^{\tilde{r}}}\leq CT^{\delta }\left[ (\Vert \psi \Vert
_{L_{t}^{\infty }H_{x}^{1}}^{q-2}+\Vert v\Vert _{L_{t}^{\infty
}H_{x}^{1}}^{q-2})\Vert \nabla (\psi -v)\Vert _{L_{t}^{s_{2}}L_{x}^{r_{2}}}%
\right] ,  \label{e3-5}
\end{equation}%
where $\theta :=\frac{N+\alpha -(N-2)p}{2p}\geq 0$ and $\delta :=\frac{%
2N-(N-2)q}{2q}\geq 0$.
\end{lemma}

\begin{proof}
By using Proposition \ref{P3.4}, we have
\begin{equation}
\left\{
\begin{array}{ll}
\Vert \Psi _{1}(\psi (t))-\Psi _{1}(v(t))\Vert _{L_{t}^{\tilde{s}}L_{x}^{%
\tilde{r}}}\lesssim \Vert g_{1}(\psi )-g_{1}(v)\Vert _{L_{t}^{s_{1}^{\prime
}}L_{x}^{r_{1}^{\prime }}}, &  \\
\Vert \nabla \lbrack \Psi _{1}(\psi (t))-\Psi _{1}(v(t))]\Vert _{L_{t}^{%
\tilde{s}}L_{x}^{\tilde{r}}}\lesssim \Vert \nabla \lbrack g_{1}(\psi
)-g_{1}(v)]\Vert _{L_{t}^{s_{1}^{\prime }}L_{x}^{r_{1}^{\prime }}}, &  \\
\Vert \Psi _{2}(\psi (t))-\Psi _{2}(v(t))\Vert _{L_{t}^{\tilde{s}}L_{x}^{%
\tilde{r}}}\lesssim \Vert g_{2}(\psi )-g_{2}(v)\Vert _{L_{t}^{s_{2}^{\prime
}}L_{x}^{r_{2}^{\prime }}}, &  \\
\Vert \nabla \lbrack \Psi _{2}(\psi (t))-\Psi _{2}(v(t))]\Vert _{L_{t}^{%
\tilde{s}}L_{x}^{\tilde{r}}}\lesssim \Vert \nabla \lbrack g_{2}(\psi
)-g_{2}(v)]\Vert _{L_{t}^{s_{2}^{\prime }}L_{x}^{r_{2}^{\prime }}}. &
\end{array}%
\right.  \label{e3-6}
\end{equation}%
By H\"{o}lder inequality in time in (\ref{e3-6}), we obtain
\begin{equation*}
\left\{
\begin{array}{ll}
\Vert g_{1}(\psi (t))-g_{1}(v(t))\Vert _{L_{t}^{s_{1}^{\prime
}}L_{x}^{r_{1}^{\prime }}}\lesssim T^{\theta }\Vert g_{1}(\psi
)-g_{1}(v)\Vert _{L_{t}^{s_{1}}L_{x}^{r_{1}^{\prime }}}, &  \\
\Vert \nabla \lbrack g_{1}(\psi (t))-g_{1}(v(t))]\Vert
_{L_{t}^{s_{1}^{\prime }}L_{x}^{r_{1}^{\prime }}}\lesssim T^{\theta }\Vert
\nabla \lbrack g_{1}(\psi )-g_{1}(v)]\Vert
_{L_{t}^{s_{1}}L_{x}^{r_{1}^{\prime }}}, &  \\
\Vert g_{2}(\psi (t))-g_{2}(v(t))\Vert _{L_{t}^{s_{2}^{\prime
}}L_{x}^{r_{2}^{\prime }}}\lesssim T^{\delta }\Vert g_{2}(\psi
)-g_{2}(v)\Vert _{L_{t}^{s_{2}}L_{x}^{r_{2}^{\prime }}}, &  \\
\Vert \nabla \lbrack g_{2}(\psi (t))-g_{2}(v(t))]\Vert
_{L_{t}^{s_{2}^{\prime }}L_{x}^{r_{2}^{\prime }}}\lesssim T^{\delta }\Vert
\nabla \lbrack g_{2}(\psi )-g_{2}(v)]\Vert
_{L_{t}^{s_{2}}L_{x}^{r_{2}^{\prime }}}, &
\end{array}%
\right.  \label{e3-7}
\end{equation*}%
where $\theta =\frac{N+\alpha -(N-2)p}{2p}$ and $\delta =\frac{2N-(N-2)q}{2q}
$. Following the methods of estimate in \cite{AR}, by using H\"{o}lder
inequality, Hardy-Littlewood-Sobolev inequality and Sobolev inequality, we
have
\begin{eqnarray*}
\Vert g_{1}(\psi (t))-g_{1}(v(t))\Vert _{L_{t}^{s_{1}}L_{x}^{r_{1}^{\prime
}}} &\lesssim &\Vert (|x|^{-(N-\alpha )}\ast |\psi |^{p})(|\psi |^{p-2}\psi
-|v|^{p-2}v)\Vert _{L_{t}^{s_{1}}L_{x}^{r_{1}^{\prime }}} \\
&&+\Vert (|x|^{-(N-\alpha )}\ast (|\psi |^{p}-|v|^{p}))|v|^{p-2}v\Vert
_{L_{t}^{s_{1}}L_{x}^{r_{1}^{\prime }}} \\
&\lesssim &\Vert (|x|^{-(N-\alpha )}\ast |\psi |^{p})\Vert
_{L_{t}^{s_{1}}L_{x}^{\frac{2N}{N-\alpha }}}\Vert |\psi |^{p-2}\psi
-|v|^{p-2}v\Vert _{L_{t}^{\infty }L_{x}^{\frac{2Np}{(N+\alpha )(p-1)}}} \\
&&+\Vert |x|^{-(N-\alpha )}\ast (|\psi |^{p}-|v|^{p})\Vert
_{L_{t}^{s_{1}}L_{x}^{\frac{2N}{N-\alpha }}}\Vert |v|^{p-2}v\Vert
_{L_{t}^{\infty }L_{x}^{\frac{2Np}{(N+\alpha )(p-1)}}} \\
&\lesssim &\Vert \psi \Vert _{L_{t}^{s_{1}p}L_{x}^{r_{1}}}^{p}(\Vert \psi
\Vert _{L_{t}^{\infty }L_{x}^{r_{1}}}^{p-2}+\Vert v\Vert _{L_{t}^{\infty
}L_{x}^{r_{1}}}^{p-2})\Vert \psi -v\Vert _{L_{t}^{\infty }L_{x}^{r_{1}}} \\
&&+\Vert |\psi |^{p}-|v|^{p}\Vert _{L_{t}^{s_{1}}L_{x}^{\frac{2N}{N+\alpha }%
}}\Vert v\Vert _{L_{t}^{\infty }L_{x}^{r_{1}}}^{p-1} \\
&\lesssim &\Vert \psi \Vert _{L_{t}^{\infty }H_{x}^{1}}^{p-1}\Vert \psi
\Vert _{L_{t}^{s_{1}}L_{x}^{r_{1}}}(\Vert \psi \Vert _{L_{t}^{\infty
}H_{x}^{1}}^{p-2}+\Vert v\Vert _{L_{t}^{\infty }H_{x}^{1}}^{p-2})\Vert \psi
-v\Vert _{L_{t}^{\infty }H_{x}^{1}} \\
&&+(\Vert \psi \Vert _{L_{t}^{\infty }H_{x}^{1}}^{p-1}+\Vert v\Vert
_{L_{t}^{\infty }H_{x}^{1}}^{p-1})\Vert \psi -v\Vert
_{L_{t}^{s_{1}}L_{x}^{r_{1}}}\Vert v\Vert _{L_{t}^{\infty }H_{x}^{1}}^{p-1}
\end{eqnarray*}%
and
\begin{eqnarray*}
\Vert \nabla \lbrack g_{1}(\psi (t))-g_{1}(v(t))]\Vert
_{L_{t}^{s_{1}}L_{x}^{r_{1}^{\prime }}} &\lesssim &\Vert \nabla \lbrack
(|x|^{-(N-\alpha )}\ast |\psi |^{p})(|\psi |^{p-2}\psi -|v|^{p-2}v)]\Vert
_{L_{t}^{s_{1}}L_{x}^{r_{1}^{\prime }}} \\
&&+\Vert \nabla \lbrack (|x|^{-(N-\alpha )}\ast (|\psi
|^{p}-|v|^{p}))|v|^{p-2}v]\Vert _{L_{t}^{s_{1}}L_{x}^{r_{1}^{\prime }}}:=A+B.
\end{eqnarray*}%
To estimate $A$ and $B$, we have
\begin{eqnarray*}
A &\lesssim &\Vert |x|^{-(N-\alpha )}\ast (|\psi |^{p-1}\nabla \psi )\Vert
_{L_{t}^{s_{1}}L_{x}^{\frac{2N}{N-\alpha }}}\Vert |\psi |^{p-2}\psi
-|v|^{p-2}v\Vert _{L_{t}^{\infty }L_{x}^{\frac{2Np}{(N+\alpha )(p-1)}}} \\
&&+\Vert |x|^{-(N-\alpha )}\ast |\psi |^{p}\Vert _{L_{t}^{\infty }L_{x}^{%
\frac{2N}{N-\alpha }}}\Vert \nabla (|\psi |^{p-2}\psi -|v|^{p-2}v)\Vert
_{L_{t}^{s_{1}}L_{x}^{\frac{2Np}{(N+\alpha )(p-1)}}} \\
&\lesssim &\Vert |\psi |^{p-1}\nabla \psi \Vert _{L_{t}^{s_{1}}L_{x}^{\frac{%
2N}{N+\alpha }}}(\Vert \psi \Vert _{L_{t}^{\infty
}L_{x}^{r_{1}}}^{p-2}+\Vert v\Vert _{L_{t}^{\infty
}L_{x}^{r_{1}}}^{p-2})\Vert \psi -v\Vert _{L_{t}^{\infty }L_{x}^{r_{1}}} \\
&&+\Vert \psi \Vert _{L_{t}^{\infty }L_{x}^{r_{1}}}^{p}(\Vert \psi \Vert
_{L_{t}^{\infty }L_{x}^{r_{1}}}^{p-2}+\Vert v\Vert _{L_{t}^{\infty
}L_{x}^{r_{1}}}^{p-2})\Vert \nabla (\psi -v)\Vert
_{L_{t}^{s_{1}}L_{x}^{r_{1}}} \\
&\lesssim &\Vert \nabla \psi \Vert _{L_{t}^{s_{1}}L_{x}^{r_{1}}}\Vert \psi
\Vert _{L_{t}^{\infty }L_{x}^{r_{1}}}^{p-1}(\Vert \psi \Vert _{L_{t}^{\infty
}L_{x}^{r_{1}}}^{p-2}+\Vert v\Vert _{L_{t}^{\infty
}L_{x}^{r_{1}}}^{p-2})\Vert \psi -v\Vert _{L_{t}^{\infty }L_{x}^{r_{1}}} \\
&&+\Vert \psi \Vert _{L_{t}^{\infty }L_{x}^{r_{1}}}^{p}(\Vert \psi \Vert
_{L_{t}^{\infty }L_{x}^{r_{1}}}^{p-2}+\Vert v\Vert _{L_{t}^{\infty
}L_{x}^{r_{1}}}^{p-2})\Vert \nabla (\psi -v)\Vert
_{L_{t}^{s_{1}}L_{x}^{r_{1}}} \\
&\lesssim &\Vert \nabla \psi \Vert _{L_{t}^{s_{1}}L_{x}^{r_{1}}}\Vert \psi
\Vert _{L_{t}^{\infty }H_{x}^{1}}^{p-1}(\Vert \psi \Vert _{L_{t}^{\infty
}H_{x}^{1}}^{p-2}+\Vert v\Vert _{L_{t}^{\infty }H_{x}^{1}}^{p-2})\Vert \psi
-v\Vert _{L_{t}^{\infty }H_{x}^{1}} \\
&&+\Vert \psi \Vert _{L_{t}^{\infty }H_{x}^{1}}^{p}(\Vert \psi \Vert
_{L_{t}^{\infty }H_{x}^{1}}^{p-2}+\Vert v\Vert _{L_{t}^{\infty
}H_{x}^{1}}^{p-2})\Vert \nabla (\psi -v)\Vert _{L_{t}^{s_{1}}L_{x}^{r_{1}}}
\end{eqnarray*}%
and
\begin{eqnarray*}
B &\lesssim &\Vert |x|^{-(N-\alpha )}\ast (\nabla (|\psi
|^{p}-|v|^{p}))\Vert _{L_{t}^{s_{1}}L_{x}^{\frac{2N}{N-\alpha }}}\Vert
|v|^{p-2}v\Vert _{L_{t}^{\infty }L_{x}^{\frac{2Np}{(N+\alpha )(p-1)}}} \\
&&+\Vert |x|^{-(N-\alpha )}\ast (|\psi |^{p}-|v|^{p})\Vert _{L_{t}^{\infty
}L_{x}^{\frac{2N}{N-\alpha }}}\Vert \nabla (|v|^{p-2}v)\Vert
_{L_{t}^{s_{1}}L_{x}^{\frac{2Np}{(N+\alpha )(p-1)}}} \\
&\lesssim &\Vert |\nabla (\psi |^{p}-|v|^{p})\Vert _{L_{t}^{s_{1}}L_{x}^{%
\frac{2N}{N+\alpha }}}\Vert v\Vert _{L_{t}^{\infty
}L_{x}^{r_{1}}}^{p-1}+\Vert |\psi |^{p}-|v|^{p}\Vert _{L_{t}^{\infty }L_{x}^{%
\frac{2N}{N+\alpha }}}\Vert v\Vert _{L_{t}^{\infty
}L_{x}^{r_{1}}}^{p-2}\Vert \nabla v\Vert _{L_{t}^{s_{1}}L_{x}^{r_{1}}} \\
&\lesssim &\left( \Vert \psi \Vert _{L_{t}^{\infty
}L_{x}^{r_{1}}}^{p-1}+\Vert v\Vert _{L_{t}^{\infty
}L_{x}^{r_{1}}}^{p-1}\right) \Vert \nabla (\psi -v)\Vert
_{L_{t}^{s_{1}}L_{x}^{r_{1}}}\Vert v\Vert _{L_{t}^{\infty
}L_{x}^{r_{1}}}^{p-1} \\
&&+\left( \Vert \psi \Vert _{L_{t}^{\infty }L_{x}^{r_{1}}}^{p-1}+\Vert
v\Vert _{L_{t}^{\infty }L_{x}^{r_{1}}}^{p-1}\right) \Vert \psi -v\Vert
_{L_{t}^{\infty }L_{x}^{r_{1}}}\Vert v\Vert _{L_{t}^{\infty
}L_{x}^{r_{1}}}^{p-2}\Vert \nabla v\Vert _{L_{t}^{s_{1}}L_{x}^{r_{1}}} \\
&\lesssim &\left( \Vert \psi \Vert _{L_{t}^{\infty }H_{x}^{1}}^{p-1}+\Vert
v\Vert _{L_{t}^{\infty }H_{x}^{1}}^{p-1}\right) \left( \Vert \nabla (\psi
-v)\Vert _{L_{t}^{s_{1}}L_{x}^{r_{1}}}\Vert v\Vert _{L_{t}^{\infty
}H_{x}^{1}}^{p-1}+\Vert \psi -v\Vert _{L_{t}^{\infty }H_{x}^{1}}\Vert v\Vert
_{L_{t}^{\infty }H_{x}^{1}}^{p-2}\Vert \nabla v\Vert
_{L_{t}^{s_{1}}L_{x}^{r_{1}}}\right) .
\end{eqnarray*}%
Thus we can get
\begin{eqnarray*}
&&\Vert \nabla \lbrack g_{1}(\psi (t))-g_{1}(v(t))]\Vert
_{L_{t}^{s_{1}}L_{x}^{r_{1}^{\prime }}} \\
&\lesssim &\Vert \nabla \psi \Vert _{L_{t}^{s_{1}}L_{x}^{r_{1}}}\Vert \psi
\Vert _{L_{t}^{\infty }H_{x}^{1}}^{p-1}\left( \Vert \psi \Vert
_{L_{t}^{\infty }H_{x}^{1}}^{p-2}+\Vert v\Vert _{L_{t}^{\infty
}H_{x}^{1}}^{p-2}\right) \Vert \psi -v\Vert _{L_{t}^{\infty }H_{x}^{1}} \\
&&+\Vert \psi \Vert _{L_{t}^{\infty }H_{x}^{1}}^{p}\left( \Vert \psi \Vert
_{L_{t}^{\infty }H_{x}^{1}}^{p-2}+\Vert v\Vert _{L_{t}^{\infty
}H_{x}^{1}}^{p-2}\right) \Vert \nabla (\psi -v)\Vert
_{L_{t}^{s_{1}}L_{x}^{r_{1}}} \\
&&+\left( \Vert \psi \Vert _{L_{t}^{\infty }H_{x}^{1}}^{p-1}+\Vert v\Vert
_{L_{t}^{\infty }H_{x}^{1}}^{p-1}\right) \left( \Vert \nabla (\psi -v)\Vert
_{L_{t}^{s_{1}}L_{x}^{r_{1}}}\Vert v\Vert _{L_{t}^{\infty
}H_{x}^{1}}^{p-1}+\Vert \psi -v\Vert _{L_{t}^{\infty }H_{x}^{1}}\Vert v\Vert
_{L_{t}^{\infty }H_{x}^{1}}^{p-2}\Vert \nabla v\Vert
_{L_{t}^{s_{1}}L_{x}^{r_{1}}}\right) .
\end{eqnarray*}%
Using H\"{o}lder inequality and Sobolev inequality again leads to
\begin{eqnarray*}
\Vert g_{2}(\psi (t))-g_{2}(v(t))\Vert _{L_{t}^{s_{2}}L_{x}^{r_{2}^{\prime
}}} &\lesssim &\Vert (|\psi |^{q-2}+|v|^{q-2})(\psi -v)\Vert
_{L_{t}^{s_{2}}L_{x}^{r_{2}^{\prime }}} \\
&\lesssim &\left( \Vert \psi \Vert _{L_{t}^{\infty
}L_{x}^{r_{2}}}^{q-2}+\Vert v\Vert _{L_{t}^{\infty
}L_{x}^{r_{2}}}^{q-2}\right) \Vert \psi -v\Vert _{L_{t}^{s_{2}}L_{x}^{r_{2}}}
\\
&\lesssim &\left( \Vert \psi \Vert _{L_{t}^{\infty }H_{x}^{1}}^{q-2}+\Vert
v\Vert _{L_{t}^{\infty }H_{x}^{1}}^{q-2}\right) \Vert \psi -v\Vert
_{L_{t}^{s_{2}}L_{x}^{r_{2}}},
\end{eqnarray*}%
and
\begin{equation*}
\Vert \nabla \lbrack g_{2}(\psi (t))-g_{2}(v(t))]\Vert
_{L_{t}^{s_{2}}L_{x}^{r_{2}^{\prime }}}\lesssim \left( \Vert \psi \Vert
_{L_{t}^{\infty }H_{x}^{1}}^{q-2}+\Vert v\Vert _{L_{t}^{\infty
}H_{x}^{1}}^{q-2}\right) \Vert \nabla (\psi -v)\Vert
_{L_{t}^{s_{2}}L_{x}^{r_{2}}}.
\end{equation*}%
The inequality (\ref{e3-2})-(\ref{e3-5}) follows by applying the above
estimates.
\end{proof}

\textbf{Now we give the proof of Theorem \ref{t1}:} As we mentioned before,
we are going to prove the theorem by means of a contraction mapping
argument. For $T>0$, we define
\begin{equation*}
\vartheta (\psi )=\max \left\{ \sup_{t\in \lbrack 0,T)}\Vert \psi \Vert
_{H_{x}^{1}},\Vert \psi \Vert _{L_{t}^{s_{1}}W_{x}^{1,r_{1}}},\Vert \psi
\Vert _{L_{t}^{s_{2}}W_{x}^{1,r_{2}}}\right\} ,
\end{equation*}%
and
\begin{equation*}
\mathcal{B}_{R,T}:=\left\{ \psi \in \mathcal{C}([0,T);H_{x}^{1})\cap
L_{t}^{s_{1}}([0,T);W_{x}^{1,r_{1}})\cap L_{t}^{s_{2}}([0,T);W_{x}^{1,r_{2}})%
\text{ }|\text{\ }\vartheta (\psi )\leq R\right\} .
\end{equation*}%
By Proposition \ref{P3.4} and Lemma \ref{L3-4} (applied with $v=0$), we
obtain
\begin{equation*}
\Vert \Psi _{1}(\psi (t))\Vert _{L_{t}^{s_{1}}W_{x}^{1,r_{1}}}\lesssim \Vert
\psi_{0} \Vert _{H_{x}^{1}}+T^{\theta }\Vert \psi \Vert _{L_{t}^{\infty
}H_{x}^{1}}^{2(p-1)}\Vert \psi \Vert _{L_{t}^{s_{1}}W_{x}^{1,r_{1}}},
\end{equation*}%
and
\begin{equation*}
\Vert \Psi _{2}(\psi (t))\Vert _{L_{t}^{s_{2}}W_{x}^{1,r_{2}}}\lesssim \Vert
\psi_{0} \Vert _{H_{x}^{1}}+T^{\delta }\Vert \psi \Vert _{L_{t}^{\infty
}H_{x}^{1}}^{q-2}\Vert \psi \Vert _{L_{t}^{s_{2}}W_{x}^{1,r_{2}}}.
\end{equation*}%
Following a similar argument, we also have
\begin{equation*}
\Vert \Psi _{1}(\psi (t))\Vert _{L_{t}^{\infty }H_{x}^{1}}\lesssim \Vert
\psi_{0} \Vert _{H_{x}^{1}}+T^{\theta }\Vert \psi \Vert _{L_{t}^{\infty
}H_{x}^{1}}^{2(p-1)}\Vert \psi \Vert _{L_{t}^{s_{1}}W_{x}^{1,r_{1}}}
\end{equation*}%
and
\begin{equation*}
\Vert \Psi _{2}(\psi (t))\Vert _{L_{t}^{\infty }H_{x}^{1}}\lesssim \Vert
\psi_{0} \Vert _{H_{x}^{1}}+T^{\delta }\Vert \psi \Vert _{L_{t}^{\infty
}H_{x}^{1}}^{q-2}\Vert \psi \Vert _{L_{t}^{s_{2}}W_{x}^{1,r_{2}}}.
\end{equation*}%
Thus we get for $\psi \in \mathcal{B}_{R,T}$,
\begin{equation*}
\Vert \Psi \Vert _{L^{s_{1}}W_{x}^{1,r_{1}}}+\Vert \Psi \Vert
_{L^{s_{2}}W_{x}^{1,r_{2}}}+\Vert \Psi \Vert _{L_{t}^{\infty }H_{x}^{1}}\leq
C\Vert \psi_{0} \Vert _{H_{x}^{1}}+CT^{\theta }R^{2p-1}+CT^{\delta }R^{q-1}.
\label{e3-10}
\end{equation*}%
Set $R=2C\Vert \psi _{0}\Vert _{H_{x}^{1}}$, for the energy-subcritical
case, we choose $0<T<1$ such that
\begin{equation}
CT^{\theta }R^{2p-2}+CT^{\delta }R^{q-2}\leq \frac{1}{2}.  \label{e3-11}
\end{equation}%
For the energy-critical case, we have $\theta =\delta =0$, and thus we take $%
\Vert \psi _{0}\Vert _{H_{x}^{1}}$ small enough, so that (\ref{e3-11})
holds. This implies that $\mathcal{B}_{R,T}$ is an invariant set of $\Phi $.

By Lemma \ref{L3-4}, we obtain that for $\psi ,v\in \mathcal{B}_{R,T}$
\begin{equation*}
d(\Phi (\psi )-\Phi (v))\lesssim (T^{\theta }R^{2(p-1)}+T^{\delta
}R^{q-2})d(\psi ,v)
\end{equation*}%
which implies that the operator $\Phi $ is a contraction. In particular, $%
\Phi $ has a unique fixed point in this space.

\section{Existence and multiplicity of standing waves}

\subsection{The case $\protect\gamma<0, \protect\mu>0$}

\begin{lemma}
\label{L4-19} Let $N\geq 2,\overline{p}<p<2_{\alpha }^{\ast },\overline{q}%
<q<2^{\ast }$ and $2Np-qN-2\alpha \leq 0.$ Then there is a unique $%
s^{-}(u)>0 $ such that $u_{s^{-}(u)}\in \mathcal{M}(c)=\mathcal{M}^{-}(c)$
and $E_{\gamma ,\mu }(u_{s^{-}(u)})=\sup_{s>0}E_{\gamma ,\mu }(u_{s})>0.$
\end{lemma}

\begin{proof}
Clearly, $\mathcal{M}(c)=\mathcal{M}^{-}(c)$ by (\ref{e2-7}). Fix $u\in S(c)$%
. Let
\begin{equation*}
\widetilde{k}(s):=\frac{2q}{\mu N(q-2)}\left( A(u)+\frac{N|\gamma |}{%
N+\alpha +2}B(u)\right) s^{\frac{4-N(q-2)}{2}}\text{ for }s>0.
\end{equation*}%
Clearly, $u_{s}\in \mathcal{M}(c)$ if and only if $\widetilde{k}(s)=C(u)$. A
direct calculation shows that $\lim_{s\rightarrow 0^{+}}\widetilde{k}%
(s)=\infty ,$ $\lim_{s\rightarrow +\infty }\widetilde{k}(s)=0$ and $%
\widetilde{k}(s)$ is decreasing on $(0,+\infty )$. This implies that there
exists a unique $s^{-}(u)>0$ such that $u_{s^{-}(u)}\in \mathcal{M}(c).$
Moreover, we also obtain that $g_{u}^{\prime }(s)>0$ on $(0,s^{-}(u))$ and $%
g_{u}^{\prime }(s)<0$ on $(s^{-}(u),+\infty ),$ which indicates that $%
E_{\gamma ,\mu }(u_{s^{-}(u)})=\sup_{s>0}E_{\gamma ,\mu }(u_{s})>0.$ This
completes the proof.
\end{proof}

\begin{lemma}
\label{L4-9} Let $N\geq 2,\overline{p}<p<2_{\alpha }^{\ast },\overline{q}%
<q<2^{\ast }$ and $2Np-qN-2\alpha \leq 0.$ Then the functional $E_{\gamma
,\mu }$ is bounded from below by a positive constant and coercive on $%
\mathcal{M}(c)=\mathcal{M}^{-}(c)$ for all $c>0$.
\end{lemma}

\begin{proof}
Let $u\in \mathcal{M}(c)=\mathcal{M}^{-}(c)$. Using (\ref{e2-1}) and (\ref%
{e2-19}) yields to
\begin{equation*}
A(u)\leq \frac{\mu N(q-2)}{2q}\overline{S}c^{\frac{2N-(N-2)q}{4}}A(u)^{\frac{%
(q-2)N}{4}},
\end{equation*}%
which implies that there exists a constant $D_{1}>0$ depending on $\mu $ and
$c,$ such that $A\left( u\right) \geq D_{1}$, since $\frac{(q-2)N}{4}>1.$
Then by (\ref{e2-19}), we get
\begin{eqnarray*}
E_{\gamma ,\mu }(u)=\frac{N(q-2)-4}{2N(q-2)}A(u)-\frac{\gamma
(Nq-2Np+2\alpha )}{2Np(q-2)}B(u)>\frac{N(q-2)-4}{2N(q-2)}D_{1}>0,
\end{eqnarray*}%
Hence, $E_{\gamma ,\mu }$ is bounded from below by a positive constant and
is coercive on $\mathcal{M}(c)$. The proof is complete.
\end{proof}

We now define $S_{r}(c):=S(c)\cap H_{r}^{1}(\mathbb{R}^{N}),\text{ }\mathcal{%
M}_{r}(c):=\mathcal{M}(c)\cap H_{r}^{1}(\mathbb{R}^{N})$ and $\mathcal{M}%
_{r}^{-}(c):=\mathcal{M}^{-}(c)\cap H_{r}^{1}(\mathbb{R}^{N})$. Clearly, by
Lemma \ref{L4-9} one has $m_{r}^{-}(c):=\inf_{u\in \mathcal{M}%
_{r}^{-}(c)}E_{\gamma ,\mu }(u)\geq \inf_{u\in \mathcal{M}^{-}(c)}E_{\gamma
,\mu }(u)>0.$ Next we apply Lemma \ref{L2-6} to construct a Palais--Smale
sequence $\{u_{n}\}\subset \mathcal{M}_{r}^{-}(c)$ for $E_{\gamma ,\mu }$
restricted to $S_{r}(c)$ at level $m_{r}^{-}(c)$. Our arguments are inspired
by \cite{BS3,CJ}. Observe that $\Theta =\emptyset $ is admissible. Now we
define the functional $J^{-}:S_{r}(c)\mapsto \mathbb{R}$ by $%
J^{-}(u):=E_{\gamma ,\mu }(u_{s^{-}\left( u\right) })$.

\begin{lemma}
\label{L3-6} The map $u\in S_{r}(c)\mapsto s^{-}(u)\in \mathbb{R}$ is of
class $C^{1}$.
\end{lemma}

\begin{proof}
Applying the implicit function theorem on the $C^{1}$ function $F:\mathbb{%
R\times }S_{r}(c)\rightarrow \mathbb{R}$ defined by $F(s,u)=g_{u}^{\prime
}(s)$.
\end{proof}

By the above lemma, we obtain that the functional $J^{-}$ is of class $C^{1}$%
.

\begin{lemma}
\label{L3-12} The map $T_{u}S_{r}(c)\rightarrow T_{u_{s^{-}(u)}}S_{r}\left(
c\right) $ defined by $\phi \rightarrow \phi _{s^{-}(u)}$ is an isomorphism.
\end{lemma}

\begin{proof}
Clearly the map is linear. For $\phi \in T_{u}S_{r}(c),$ we have
\begin{equation*}
\int_{\mathbb{R}^{N}}u_{s}(x)\phi _{s}(x)dx=\int_{\mathbb{R}%
^{N}}s^{N/2}u(sx)s^{N/2}\phi (sx)dx=\int_{\mathbb{R}^{N}}u(y)\phi (y)dy=0,
\end{equation*}%
which implies that $\phi _{s}\in T_{u_{s}}S_{r}(c)$ and the map is
well-defined. The rest of the proof is standard, see for example \cite[Lemma
3.6]{BS3}.
\end{proof}

\begin{lemma}
\label{L3-13} It holds $(J^{-})^{\prime }(u)[\phi ]=E_{\gamma ,\mu }^{\prime
}(u_{s^{-}(u)})[\phi _{s^{-}(u)}]$ for any $u\in S_{r}(c)$ and $\phi \in
T_{u}S_{r}(c).$
\end{lemma}

\begin{proof}
The proof is similar to that of \cite[Lemma 3.15]{CJ}, we omit it here.
\end{proof}

\begin{lemma}
\label{L4-15} Let $\mathcal{F}$ be a homotopy stable family of compact
subsets of $S_{r}(c)$ with closed boundary $\Theta $ and let
\begin{equation*}
e_{\mathcal{F}}^{-}:=\inf_{H\in \mathcal{F}}\max_{u\in H}J^{-}(u).
\end{equation*}%
Assume that $\Theta $ is contained in a connected component of $\mathcal{M}%
_{r}^{-}(c)$ and $\max \{\sup J^{-}(\Theta ),0\}<e_{\mathcal{F}}^{-}<\infty $%
. Then there exists a Palais--Smale sequence $\{u_{n}\}\subset \mathcal{M}%
_{r}^{-}(c)$ for $E_{\gamma ,\mu }$ restricted to $S_{r}(c)$ at level $e_{%
\mathcal{F}}^{-}.$
\end{lemma}

\begin{proof}
First of all, we take $\{D_{n}\}\subset \mathcal{F}$ such that $\max_{u\in
D_{n}}J^{-}(u)<e_{\mathcal{F}}^{-}+\frac{1}{n}$ and define $\eta
:[0,1]\times S_{r}(c)\rightarrow S_{r}(c)$ by%
\begin{equation*}
\eta (t,u)=u_{1-t+ts^{-}(u)}.
\end{equation*}%
Since $s^{-}(u)=1$ for any $u\in \mathcal{M}_{r}^{-}(c)$ and $\Theta \subset
\mathcal{M}_{r}^{-}(c)$, we have $\eta (t,u)=u$ for $(t,u)\in (\{0\}\times
S_{r}(c))\cup ([0,1]\times \Theta )$. Also note that $\eta $ is continuous.
Then it follows from the definition of $\mathcal{F}$ that $\mathbf{A}%
_{n}:=\eta (\{1\}\times D_{n})=\{u_{s^{-}(u)}$ $|$ $u\in D_{n}\}\in \mathcal{%
F}.$ Clearly, $\mathbf{A}_{n}\subset \mathcal{M}_{r}^{-}(c)$ for all $n\in
\mathbb{N}$. Let $v\in \mathbf{A}_{n}$. Then $v=u_{s^{-}(u)}$ for some $u\in
D_{n},$ and $J^{-}(u)=J^{-}(v)$, which shows that $\max_{\mathbf{A}%
_{n}}J^{-}=\max_{D_{n}}J^{-}.$ Thus, $\{\mathbf{A}_{n}\}\subset \mathcal{M}%
_{r}^{-}(c)$ is another minimizing sequence of $e_{\mathcal{F}}^{-}$.
Applying Lemma \ref{L2-6} yields a Palais--Smale sequence $\{\widehat{u}%
_{n}\}$ for $J^{-}$ on $S_{r}(c)$ at level $e_{\mathcal{F}}^{-}$ satisfying $%
\text{dist}_{H^{1}(\mathbb{R}^{N})}(\widehat{u}_{n},\mathbf{A}%
_{n})\rightarrow 0$ as $n\rightarrow \infty $. Setting $t_{n}=s^{-}(\widehat{%
u}_{n})$ and $u_{n}=(\widehat{u}_{n})_{s_{n}}\in \mathcal{M}_{r}^{-}(c)$. We
claim that there exists a constant $C_{0}>0$ such that
\begin{equation}
\frac{1}{C_{0}}\leq s_{n}^{2}\leq C_{0}  \label{e3-23}
\end{equation}%
for $n\in \mathbb{N}$ large enough. Indeed, note that
\begin{equation}
s_{n}^{2}=\frac{A((\widehat{u}_{n})_{s_{n}})}{A(\widehat{u}_{n})}=\frac{%
A(u_{n})}{A(\widehat{u}_{n})}.  \label{e3-24}
\end{equation}%
Since $E_{\gamma ,\mu }(u_{n})=J^{-}(\widehat{u}_{n})\rightarrow e_{\mathcal{%
F}}^{-}$ as $n\rightarrow \infty $, it follows from Lemma \ref{L4-9} that
there exists a constant $M_{0}>0$ such that
\begin{equation}
\frac{1}{M_{0}}\leq A(u_{n})\leq M_{0}.  \label{e3-25}
\end{equation}%
On the other hand, since $\{\mathbf{A}_{n}\}\subset \mathcal{M}_{r}^{-}(c)$
is a minimizing sequence for $e_{\mathcal{F}}^{-}$ and $E_{\gamma ,\mu }$ is
coercive on $\mathcal{M}_{r}^{-}(c)$, we conclude that $\{\mathbf{A}_{n}\}$
is uniformly bounded in $H^{1}(\mathbb{R}^{N})$, and combining the fact of $%
\text{dist}_{H^{1}(\mathbb{R}^{N})}\left( \widehat{u}_{n},\mathbf{A}%
_{n}\right) \rightarrow 0$ as $n\rightarrow \infty ,$ gives $\sup_{n}A(%
\widehat{u}_{n})<\infty $. Also, since $\{\mathbf{A}_{n}\}$ is compact for
every $n\in \mathbb{N}$, there exists a $v_{n}\in \mathbf{A}_{n}$ such that
dist$_{H^{1}(\mathbb{R}^{N})}(\widehat{u}_{n},\mathbf{A}_{n})=\Vert v_{n}-%
\widehat{u}_{n}\Vert _{H^{1}}.$ Then it follows from Lemma \ref{L4-9} that
for some $\sigma >0,$
\begin{equation}
A(\widehat{u}_{n})\geq A(v_{n})-A(\widehat{u}_{n}-v_{n})\geq \frac{\sigma }{2%
}.  \label{e3-26}
\end{equation}%
Hence, by (\ref{e3-24})--(\ref{e3-26}), we proved the claim.

Next, we show that $\{u_{n}\}\subset \mathcal{M}_{r}^{-}(c)$ is a
Palais--Smale sequence for $E_{\gamma ,\mu }$ on $S_{r}(c)$ at level $e_{%
\mathcal{F}}^{-}$. Denote by $\Vert \cdot \Vert _{\ast }$ the dual norm of $%
(T_{u_{n}}S_{r}(c))^{\ast }$. Then we have
\begin{equation*}
\Vert E_{\gamma ,\mu }^{\prime }(u_{n})\Vert _{\ast }=\sup_{\phi \in
T_{u_{n}}S_{r}(c),\Vert \phi \Vert \leq 1}|E_{\gamma ,\mu }^{\prime
}(u_{n})[\phi ]|=\sup_{\phi \in T_{u_{n}}S_{r}(c),\Vert \phi \Vert \leq
1}|E_{\gamma ,\mu }^{\prime }(u_{n})[(\phi _{-s_{n}})_{s_{n}}]|.
\end{equation*}%
By Lemma \ref{L3-12}, we obtain that $T_{\widehat{u}_{n}}S_{r}(c)\rightarrow
T_{u_{n}}S_{r}(c)$ defined by $\phi \rightarrow \phi _{t_{n}}$ is an
isomorphism. Moreover, it follows from Lemma \ref{L3-13} that
\begin{equation*}
(J^{-})^{\prime }(\widehat{u}_{n})[\phi _{-s_{n}}]=E_{\gamma ,\mu }^{\prime
}((\widehat{u}_{n})_{s_{n}})[(\phi _{-s_{n}})_{s_{n}}]=E_{\gamma ,\mu
}^{\prime }(u_{n})[(\phi _{-s_{n}})_{s_{n}}].
\end{equation*}%
Thus, it follows that
\begin{equation}
\Vert E_{\gamma ,\mu }^{\prime }(u_{n})\Vert _{\ast }=\sup_{\phi \in
T_{u_{n}}S_{r}(c),\Vert \phi \Vert \leq 1}|(J^{-})^{\prime }(\widehat{u}%
_{n})[\psi _{-s_{n}}]|.  \label{e3-27}
\end{equation}%
Note that $\Vert \psi _{-s_{n}}\Vert _{H^{1}}\leq C\Vert \psi \Vert
_{H^{1}}\leq C$ for some $C>0$ by (\ref{e3-23}). So by (\ref{e3-27}) we
deduce that $\{u_{n}\}\subset \mathcal{M}_{r}^{-}(c)$ is a Palais--Smale
sequence for $E_{\gamma ,\mu }$ on $\mathcal{M}_{r}^{-}(c)$ at level $e_{%
\mathcal{F}}^{-}$. The proof is complete.
\end{proof}

\begin{lemma}
\label{L4-18} Let $N\geq 2,\overline{p}<p<2_{\alpha }^{\ast },\overline{q}%
<q<2^{\ast }$ with $2Np-Nq-2\alpha \leq 0.$ Then for $|\gamma |c^{\frac{%
N+\alpha -(N-2)p}{2}}$ sufficiently small, there exists a Palais--Smale
sequence $\{u_{n}\}\subset \mathcal{M}_{r}^{-}(c)$ for $E_{\gamma ,\mu }$
restricted to $S_{r}(c)$ at the level $m_{r}^{-}(c)>0.$
\end{lemma}

\begin{proof}
Let $\{u_{n}\}\subset \mathcal{M}_{r}^{-}(c)$. By Lemma \ref{L4-15}, we
choose the set $\overline{\mathcal{F}}$ of all singletons belonging to $%
S_{r}(c)$ and $\Theta =\emptyset $, which is clearly a homotopy stable
family of compact subsets of $S_{r}(c)$ (without boundary). Note that $e_{%
\overline{\mathcal{F}}}^{-}=\inf_{H\in \overline{\mathcal{F}}}\max_{u\in
H}J^{-}(u)=\inf_{u\in S_{r}(c)}J^{-}(u)=\inf_{u\in \mathcal{M}%
_{r}^{-}(c)}E_{\gamma ,\mu }(u)=m_{r}^{-}(c).$ Then the lemma follows
directly from Lemma \ref{L4-15}. This completes the proof.
\end{proof}

When $\overline{p}<p\leq 2_{\alpha }^{\ast }$ and $\overline{q}<q<2^{\ast }$
with $2Np-Nq-2\alpha >0$, it is very difficult to verify that $\mathcal{M}%
^{0}(c)=\emptyset $ and that $E_{\gamma ,\mu }$ is bounded from below on $%
\mathcal{M}(c)$. To achieve this goal, we need to introduce a novel natural
constraint manifold inspired by \cite{CJ}.

Set
\begin{equation}
\Gamma :=\frac{q(Np-N-\alpha -2)}{2Np-Nq-2\alpha }\left( \frac{|\gamma |%
\overline{B}(2Np-Nq-2\alpha )(Np-N-\alpha )c^{\frac{N+\alpha -p(N-2)}{2}}}{%
2p(N(q-2)-4)}\right) ^{\frac{N(q-2)-4}{2(Np-N-\alpha -2)}}.  \label{e4-32}
\end{equation}%
Clearly, $\Gamma >0$ and $\Gamma \rightarrow 0$ as $|\gamma |\rightarrow 0,$
since $\overline{p}<p\leq 2_{\alpha }^{\ast }$ and $\overline{q}<q<2^{\ast }$
with $2Np-Nq-2\alpha >0.$ We define a set $V$ by%
\begin{eqnarray}
V :=\left\{ u\in S(c)\text{ }|\text{\ }\mu C(u)>\Gamma A(u)^{\frac{N(q-2)}{4}%
}\right\} =\left\{ u\in S(c)\text{ }|\text{\ }(s_{u}^{\ast
})^{2}A(u)<K\right\} ,  \label{e4-34}
\end{eqnarray}%
where
\begin{equation}
s_{u}^{\ast }:=\left( \frac{4q(Np-N-\alpha -2)A(u)}{\mu
N(q-2)(2Np-Nq-2\alpha )C(u)}\right) ^{\frac{2}{N(q-2)-4}}  \label{e4-30}
\end{equation}%
and
\begin{equation*}
K:=\left( \frac{4}{(q-2)N}\right) ^{\frac{4}{N(q-2)-4}}\left( \frac{%
2p((q-2)N-4)}{|\gamma|\overline{B}(2Np-Nq-2\alpha )(Np-N-\alpha )c^{\frac{%
N+\alpha -p(N-2)}{2}}}\right) ^{\frac{2}{Np-N-\alpha -2}}.
\end{equation*}

To show that the set $V$ is not empty, we consider the following problem:
\begin{equation}
\left\{
\begin{array}{ll}
-\Delta u+\lambda u=\mu |u|^{q-2}u & \text{ in }\mathbb{R}^{N}, \\
\int_{\mathbb{R}^{N}}|u|^{2}dx=c>0, &
\end{array}%
\right.  \tag{$E_{\mu }^{\infty }$}
\end{equation}%
where $N\geq 3$ and $\overline{q}<q<2^{\ast }$. From \cite{J} we know that $%
(E_{\mu }^{\infty })$ admits a ground state normalized solution $\omega _{0}$
with positive energy $\beta $. Then $\beta =\inf_{u\in \mathbf{N}_{\mu
}(c)}G_{\mu }(u)=G_{\mu }(\omega _{0})>0,$ where $G_{\mu }$ is the energy
functional related to $(E_{\mu }^{\infty })$ given by
\begin{equation*}
G_{\mu }(u)=\frac{1}{2}\int_{\mathbb{R}^{N}}|\nabla u|^{2}dx-\frac{\mu }{q}%
\int_{\mathbb{R}^{N}}|u|^{q}dx
\end{equation*}%
and
\begin{equation*}
\mathbf{N}_{\mu }(c)=\left\{ u\in S(c)\text{ }|\text{ }P_{\mu }(u)=0\right\}
\end{equation*}%
with $P_{\mu }(u)=\int_{\mathbb{R}^{N}}|\nabla u|^{2}dx-\frac{\mu (q-2)N}{2q}%
\int_{\mathbb{R}^{N}}|u|^{q}dx$. Clearly, it holds
\begin{equation*}
\beta =\frac{1}{2}A(\omega _{0})-\frac{\mu }{q}C(\omega _{0})=\frac{N(q-2)-4%
}{2N(q-2)}A(\omega _{0})>0.
\end{equation*}

Set
\begin{eqnarray*}
\gamma _{\ast } &:&=\left[ \left( \frac{N(q-2)-4}{2N\beta (q-2)}\right)
^{1/2}\left( \frac{2(2Np-Nq-2\alpha )}{N(Np-N-\alpha -2)(q-2)}\right) ^{%
\frac{2}{N(q-2)-4}}\right] ^{Np-N-\alpha -2} \\
&&\text{ \ }\times \left( \frac{2p(N(q-2)-4)}{\overline{B} (Np-N-\alpha
)(2Np-Nq-2\alpha )c^{\frac{N+\alpha -p(N-2)}{2}}}\right) .
\end{eqnarray*}%
For $0<\left\vert \gamma \right\vert <\gamma _{\ast }$, we have $\mu
C(\omega _{0})=\frac{2q}{N(q-2)}A(\omega _{0})=\frac{4q\beta }{N(q-2)-4}%
>\Gamma A(\omega _{0})^{\frac{N(q-2)}{4}},$ which implies that $\omega
_{0}\in V,$ namely, $V$ is nonempty.

Now we define the set $\Lambda (c):=\left\{ u\in V\text{ }|\text{\ }%
g_{u}^{\prime }(1)=0\right\} ,$ which is a natural constraint. Similar to
Pohozaev manifold $\mathcal{M}(c),$ we can split $\Lambda (c)$ into three
parts corresponding to local minima, local maxima and points of inflection
as follows%
\begin{eqnarray*}
\Lambda ^{+}(c) &:&=\left\{ u\in V\text{ }|\text{\ }g_{u}^{\prime
}(1)=0,g_{u}^{\prime \prime }(1)>0\right\} ; \\
\Lambda ^{-}(c) &:&=\left\{ u\in V\text{ }|\text{\ }g_{u}^{\prime
}(1)=0,g_{u}^{\prime \prime }(1)<0\right\} ; \\
\Lambda ^{0}(c) &:&=\left\{ u\in V\text{ }|\text{\ }g_{u}^{\prime
}(1)=0,g_{u}^{\prime \prime }(1)=0\right\} .
\end{eqnarray*}%
Then we have the following results.

\begin{lemma}
\label{L4-2} Let $N\geq 2,\overline{p}<p\leq 2_{\alpha }^{\ast },\overline{q}%
<q<2^{\ast }$ with $2Np-qN-2\alpha >0$. Then $\Lambda ^{0}(c)=\emptyset $
for all $c>0.$
\end{lemma}

\begin{proof}
Suppose that $\Lambda ^{0}(c)\neq \emptyset $. Then for $u\in \Lambda
^{0}(c) $, it follows from (\ref{e2-1}) and (\ref{e2-7}) that%
\begin{eqnarray*}
\frac{N(q-2)-4}{2}A(u) &=&\frac{|\gamma |(Np-N-\alpha )(2Np-Nq-2\alpha )}{4p}%
B(u) \\
&\leq &\frac{|\gamma |\overline{B}(Np-N-\alpha )(2Np-Nq-2\alpha )}{4p}c^{%
\frac{N+\alpha -p(N-2)}{2}}A(u)^{\frac{Np-N-\alpha }{2}},
\end{eqnarray*}%
which indicates that
\begin{equation}
A(u)\geq \left[ \frac{2p(N(q-2)-4)}{|\gamma |\overline{B}(Np-N-\alpha
)(2Np-Nq-2\alpha )}c^{-\frac{N+\alpha -p(N-2)}{2}}\right] ^{\frac{2}{%
Np-N-\alpha -2}}.  \label{e4-1}
\end{equation}%
On the other hand, by using (\ref{e2-8}) and the fact of $\mu C(u)>\Gamma
A(u)^{\frac{N(q-2)}{4}},$ we have
\begin{equation}
A(u)<\left( \frac{4}{N(q-2)}\right) ^{\frac{4}{N(q-2)-4}}\left( \frac{%
2p(N(q-2)-4)}{|\gamma |\overline{B}(2Np-Nq-2\alpha )(Np-N-\alpha )c^{\frac{%
N+\alpha -p(N-2)}{2}}}\right) ^{\frac{2}{Np-N-\alpha -2}}.  \label{e4-2}
\end{equation}%
Hence, using (\ref{e4-1}) and (\ref{e4-2}) leads to $q<\overline{q},$ which
is a contradiction. This completes the proof.
\end{proof}

For $u\in V$, we define
\begin{equation}
\widehat{s}_{u}:=\left( \frac{2qA(u)}{\mu N(q-2)C(u)}\right) ^{\frac{2}{%
N(q-2)-4}}.  \label{e4-6}
\end{equation}

\begin{lemma}
\label{L4-3} Let $N\geq 2,\overline{p}<p\leq 2_{\alpha }^{\ast },\overline{q}%
<q<2^{\ast }$ with $2Np-qN-2\alpha >0$. Then for each $u\in V$ there exist
two positive constants $s_{\ast }^{+}(u)$ and $s_{\ast }^{-}(u)$ satisfying $%
\widehat{s}_{u}<s_{\ast }^{-}(u)<s_{u}^{\ast }<s_{\ast }^{+}(u)$ such that $%
u_{s_{\ast }^{\pm }(u)}\in \Lambda ^{\pm }(c)$, where $s_{u}^{\ast }$ is as
in (\ref{e4-30}).
\end{lemma}

\begin{proof}
Define
\begin{equation*}
\widehat{h}(s)=A(u)s^{N+\alpha +2-Np}-\frac{\mu N(q-2)}{2q}C(u)s^{\frac{%
Nq+2\alpha -2Np}{2}}\quad \text{for}\quad t>0.
\end{equation*}%
Clearly, $u_{s}\in \Lambda (c)$ if and only if $\widehat{h}(s)=\frac{\gamma
(Np-N-\alpha )}{2p}B(u)$. A straightforward calculation shows that $\widehat{%
h}(\widehat{s}_{u})=0,$ $\lim_{t\rightarrow 0^{+}}\widehat{h}(s)=\infty $
and $\lim_{s\rightarrow \infty }\widehat{h}(s)=0.$ Note that
\begin{equation*}
\widehat{h}^{\prime }(s)=s^{N+\alpha +1-Np}\left[ (N+\alpha +2-Np)A(u)-\frac{%
\mu N(q-2)(Nq+2\alpha -2Np)}{4q}C(u)s^{\frac{N(q-2)-4}{2}}\right] .
\end{equation*}%
Then we obtain that $\widehat{h}(s)$ is decreasing when $0<s<s_{u}^{\ast }$
and is increasing when $s>s_{u}^{\ast }$, which implies that
\begin{eqnarray*}
\inf_{s>0}\widehat{h}(s) &=&\widehat{h}\left( \left( \frac{4q(Np-N-\alpha
-2)A(u)}{\mu N(q-2)(2Np-Nq-2\alpha )C(u)}\right) ^{\frac{2}{N(q-2)-4}}\right)
\\
&=&-\frac{N(q-2)-4}{2Np-Nq-2\alpha }\left( \frac{4q(Np-N-\alpha -2)A(u)}{\mu
N(q-2)(2Np-Nq-2\alpha )C(u)}\right) ^{\frac{2(N+\alpha +2-Np)}{N(q-2)-4}%
}A(u).
\end{eqnarray*}%
Since $u\in V$, we have $\mu C(u)>\Gamma A(u)^{\frac{N(q-2)}{4}}.$ This
indicates that
\begin{equation*}
\inf_{s>0}\widehat{h}(s)<\frac{\gamma \overline{B}(Np-N-\alpha )}{2p}A(u)^{%
\frac{Np-N-\alpha }{2}}c^{\frac{N+\alpha -p(N-2)}{2}}\leq \frac{\gamma
(Np-N-\alpha )}{2p}B(u),
\end{equation*}%
which implies that there exist two constants $s_{\ast }^{+}(u)$ and $s_{\ast
}^{-}(u)$ satisfying $\widehat{s}_{u}<s_{\ast }^{-}(u)<s_{u}^{\ast }<s_{\ast
}^{+}(u)$ such that
\begin{equation*}
\widehat{h}(s_{\ast }^{\pm }(u))=\frac{\gamma (Np-N-\alpha )}{2p}B(u).
\end{equation*}%
Namely, $u_{s_{\ast }^{\pm }(u)}\in \Lambda (c).$ By a calculation on the
second order derivatives, we find that
\begin{equation*}
g_{u_{s_{\ast }^{-}(u)}}^{\prime \prime }(1)=(s_{\ast }^{-}(u))^{Np-N-\alpha
+1}(\widehat{h})^{\prime }(s_{\ast }^{-}(u))<0\text{ and }g_{u_{s_{\ast
}^{+}\left( u\right) }}^{\prime \prime }\left( 1\right) =\left( s_{\ast
}^{+}\left( u\right) \right) ^{Np-N-\alpha +1}(\widehat{h})^{\prime }\left(
s_{\ast }^{+}\left( u\right) \right) >0.
\end{equation*}%
These imply that $u_{s_{\ast }^{\pm }\left( u\right) }\in \Lambda ^{\pm
}\left( c\right) $. This completes the proof.
\end{proof}

\begin{lemma}
\label{L4-1} Let $N\geq 2,\overline{p}<p\leq 2_{\alpha }^{\ast },\overline{q}%
<q<2^{\ast }$ with $2Np-qN-2\alpha >0$. Then the functional $E_{\gamma ,\mu
} $ is bounded from below by a positive constant and coercive on $\Lambda
^{-}(c)$ for all $c>0.$
\end{lemma}

\begin{proof}
Let $u\in \Lambda ^{-}(c)$. Then we have
\begin{equation}
A(u)-\frac{\gamma (Np-N-\alpha )}{2p}B(u)-\frac{\mu N(q-2)}{2q}C(u)=0
\label{e4-15}
\end{equation}%
and
\begin{equation}
(Np-N-\alpha -2)A(u)>\frac{\mu N(q-2)(2Np-2\alpha -qN)}{4q}C(u).
\label{e4-17}
\end{equation}%
It follows from (\ref{e2-1}) and (\ref{e4-15}) that%
\begin{eqnarray*}
A(u)=\frac{\gamma (Np-N-\alpha )}{2p}B(u)+\frac{\mu N(q-2)}{2q}C(u) \leq%
\frac{\mu N(q-2)}{2q}\overline{S}c^{\frac{2N-q(N-2)}{4}}A(u)^{\frac{N(q-2)}{4%
}},
\end{eqnarray*}%
which implies that $A(u)\geq D_{2}$ for some $D_{2}>0$. Using this, together
with (\ref{e4-15}) and (\ref{e4-17}), leads to
\begin{eqnarray*}
E_{\gamma ,\mu }(u) &=&\frac{Np-N-\alpha -2}{2(Np-N-\alpha )}A(u)-\frac{\mu
(2Np-2\alpha -qN)}{2q(Np-N-\alpha )}C(u) \\
&>&\frac{(Np-N-\alpha -2)(N(q-2)-4)}{2N(Np-N-\alpha )(q-2)}D_{2}>0.
\end{eqnarray*}%
This shows that $E_{\gamma ,\mu }$ is bounded from below by a positive
constant and coercive on $\Lambda ^{-}(c)$. This completes the proof.
\end{proof}

\begin{remark}
\label{R2} One point worth emphasizing is that we are not sure whether $%
E_{\gamma ,\mu }$ is bounded from below on $\Lambda ^{+}(c),$ which seems to
be an interesting issue.
\end{remark}

We now define%
\begin{equation*}
V_{r}:=V\cap H_{r}^{1}(\mathbb{R}^{N}),\text{ }\Lambda _{r}(c):=\Lambda
(c)\cap H_{r}^{1}(\mathbb{R}^{N})\text{ and }\Lambda _{r}^{-}(c):=\Lambda
^{-}(c)\cap H_{r}^{1}(\mathbb{R}^{N}).
\end{equation*}%
It follows from Lemma \ref{L4-1} that
\begin{equation*}
l_{r}^{-}(c):=\inf_{u\in \Lambda _{r}^{-}(c)}E_{\gamma ,\mu }(u)\geq
\inf_{u\in \Lambda ^{-}(c)}E_{\gamma ,\mu }(u)>0.
\end{equation*}%
Next, we shall establish the existence of a Palais--Smale sequence $%
\{u_{n}\}\subset \Lambda _{r}^{-}(c)$ by Lemma \ref{L2-6} in Section 2.3.
Define the functional $F^{-}:V_{r}\mapsto \mathbb{R}$ by $%
F^{-}(u):=E_{\gamma ,\mu }(u_{s_{\ast }^{-}(u)})$, where $s_{\ast }^{-}(u)$
is given in Lemma \ref{L4-3}.

As Lemmas \ref{L3-6}-\ref{L3-13}, we have the following results without
proof.

\begin{lemma}
\label{L4-10} If $\Lambda ^{0}( c) =\emptyset $, then $\Lambda ( c) =V\cap
\mathcal{M}( c) $ is smooth submanifold of codimension $2$ of $H^{1}(
\mathbb{R}^{N}) $ and a submanifold of codimension $1$ in $S(c) $.
\end{lemma}

\begin{lemma}
\label{L4-14} The map $u\in V_{r}\mapsto s_{\ast }^{-}(u)\in \mathbb{R}$ is
of class $C^{1}$.
\end{lemma}

\begin{lemma}
\label{L4-6} The map $T_{u}V_{r}\rightarrow T_{u_{s_{\ast }^{-}(u)}}V_{r}$
defined by $\phi\rightarrow \phi _{s_{\ast }^{-}(u)}$ is isomorphism.
\end{lemma}

\begin{lemma}
\label{L4-7} We have that $(F^{-})^{\prime }(u)[\phi ]=E_{\gamma ,\mu
}^{\prime }(u_{s_{\ast }^{-}(u)})[\phi _{s_{\ast }^{-}(u)}]$ for any $u\in
V_{r}$ and $\phi \in T_{u}V_{r}$.
\end{lemma}

Since $V_{r}$ has a boundary, in order to guarantee that our deformation
arguments happen inside $V_{r},$ we need the following lemma.

\begin{lemma}
\label{L4-13} Let $0<|\gamma |<\gamma _{\ast }$. If $\{v_{n}\}\subset V_{r}$
is a sequence with $v_{n}\rightarrow \bar{v}\in \partial V_{r}$ strongly in $%
H^{1}(\mathbb{R}^{N})$, then there exist a sequence $\{s_{\ast
}^{-}(v_{n})\} $ with $s_{\ast }^{-}(v_{n})\rightarrow s_{\infty }>0$
satisfying $g_{\bar{v}_{s_{\infty }}}^{\prime \prime }(1)\leq 0$ and a
constant $\overline{M}=\overline{M}(\gamma ,c)>0$ satisfying $\overline{M}%
\rightarrow \infty $ as $|\gamma |\rightarrow 0$ such that $E_{\gamma ,\mu }(%
\bar{v}_{s_{\infty }})>\overline{M}.$
\end{lemma}

\begin{proof}
Let $\{v_{n}\}\subset V_{r}$ be such that $v_{n}\rightarrow \bar{v}\in
\partial V_{r}$ strongly in $H^{1}(\mathbb{R}^{N})$ as $n\rightarrow \infty $%
. Then it holds $s_{v_{n}}^{\ast }\rightarrow s_{\bar{v}}^{\ast }\neq 0$ and
$\widehat{s}_{v_{n}}\rightarrow \widehat{s}_{\bar{v}}\neq 0$ as$%
\,n\rightarrow \infty ,$ where $s_{u}^{\ast }$ and $\widehat{t}_{u}$ are
defined in (\ref{e4-30}) and (\ref{e4-6}), respectively. Since $v_{n}\in
V_{r}$ and $\bar{v}\in \partial V_{r}$, we have $(s_{v_{n}}^{\ast
})^{2}A(v_{n})<K$ and $(s_{\bar{v}}^{\ast })^{2}A(\bar{v})=K$. Moreover, it
follows from Lemma \ref{L4-3} that $\widehat{s}_{v_{n}}<s_{\ast
}^{-}(v_{n})\leq s_{v_{n}}^{\ast }\leq s_{\ast }^{+}(v_{n}),$ which implies
that the sequence $\{s_{\ast }^{-}(v_{n})\}$ is bounded. Then we can assume
that $\{s_{\ast }^{-}(v_{n})\}$, up to a subsequence, converges to $%
s_{\infty }\neq 0$. Since $v_{n}\rightarrow \bar{v}$ strongly in $H^{1}(%
\mathbb{R}^{N})$ and $Q((v_{n})_{s_{\ast }^{-}(v_{n})})=0$, we infer that $Q(%
\bar{v}_{s_{\infty }})=0$, namely,
\begin{equation}
A(\bar{v}_{s_{\infty }})-\frac{\gamma (Np-N-\alpha )}{2p}B(\bar{v}%
_{s_{\infty }})-\frac{\mu N(q-2)}{2q}C(\bar{v}_{s_{\infty }})=0.
\label{e4-18}
\end{equation}%
Then we have%
\begin{equation*}
A(\bar{v}_{s_{\infty }})\leq \frac{\mu N(q-2)}{2q}C(\bar{v}_{s_{\infty }})=%
\frac{N\Gamma (q-2)}{2q}A(\bar{v}_{s_{\infty }})^{\frac{N(q-2)}{4}},
\end{equation*}%
which implies that
\begin{equation}
A(\bar{v}_{s_{\infty }})\geq \left( \frac{2q}{N\Gamma (q-2)}\right) ^{\frac{4%
}{N(q-2)-4}},  \label{e4-40}
\end{equation}%
where $\Gamma $ is as in (\ref{e4-32}) and $\Gamma \rightarrow 0$ as $%
\left\vert \gamma \right\vert \rightarrow 0$. Since $g_{(v_{n})_{s_{\ast
}^{-}(v_{n})}}^{\prime \prime }(1)<0$, we have $g_{\bar{v}_{s_{\infty
}}}^{\prime \prime }(1)\leq 0$, namely,
\begin{equation}
(Np-N-\alpha -2)A(\bar{v}_{s_{\infty }})\geq \frac{\mu N(q-2)(2Np-Nq-2\alpha
)}{4q}C(\bar{v}_{s_{\infty }}).  \label{e4-19}
\end{equation}%
It follows from (\ref{e4-18})--(\ref{e4-19}) that%
\begin{eqnarray*}
E_{\gamma ,\mu }(\bar{v}_{s_{\infty }}) &>&\frac{(Np-N-\alpha -2)(N(q-2)-4)}{%
2N(Np-N-\alpha )(q-2)}A(\bar{v}_{s_{\infty }}) \\
&\geq &\frac{(Np-N-\alpha -2)(N(q-2)-4)}{2N(Np-N-\alpha )(q-2)}\left( \frac{%
2q}{N\Gamma (q-2)}\right) ^{\frac{4}{N(q-2)-4}}:=\overline{M}.
\end{eqnarray*}%
Clearly, $\overline{M}\rightarrow \infty $ as $|\gamma |\rightarrow 0$. This
completes the proof.
\end{proof}

\begin{lemma}
\label{L4-8} Let $\mathcal{F}$ be a homotopy stable family of compact
subsets of $V_{r}$ with closed boundary $B$ and let
\begin{equation*}
e_{\mathcal{F}}^{-}:=\inf_{H\in \mathcal{F}}\max_{u\in H}F^{-}(u).
\end{equation*}%
Suppose that $\Theta $ is contained in a connected component of $\Lambda
_{r}^{-}(c)$ and $\max \{\sup J^{-}(\Theta ),0\}<e_{\mathcal{F}}^{-}<\infty $%
. Then there exists a Palais--Smale sequence $\{u_{n}\}\subset \Lambda
_{r}^{-}(c)$ for $E_{\gamma ,\mu }$ restricted to $S(c)$ at level $e_{%
\mathcal{F}}^{-}$ for $|\gamma |$ sufficiently small.
\end{lemma}

\begin{proof}
Take $\{D_{n}\}\subset \mathcal{F}$ such that $\max_{u\in D_{n}}F^{-}(u)<e_{%
\mathcal{F}}^{-}+\frac{1}{n}$ and $\eta :[0,1]\times V_{r}\rightarrow V_{r},$
$\eta (t,u)=u_{1-t+ts_{\ast }^{-}(u)}.$ Since $s_{\ast }^{-}(u)=1$ for any $%
u\in \Lambda _{r}^{-}(c)$ and $\Theta \subset \Lambda _{r}^{-}(c)$, we
obtain $\eta (t,u)=u$ for $(t,u)\in (\{0\}\times V_{r})\cup ([0,1]\times
\Theta )$. Also note that $\eta $ is continuous. Then, using the definition
of $\mathcal{F}$,
\begin{equation*}
\mathbf{A}_{n}:=\eta (\{1\}\times D_{n})=\left\{ u_{s(u)}\text{ }|\text{ }%
u\in D_{n}\right\} \in \mathcal{F}.
\end{equation*}
Clearly, $\mathbf{A}_{n}\subset \Lambda _{r}^{-}(c)$ for all $n\in \mathbb{N}
$. Let $v\in \mathbf{A}_{n}$, i.e. $v=u_{s_{\ast }^{-}(u)}$ for some $u\in
D_{n}$ and $F^{-}(u)=F^{-}(v)$. Then we have $\max_{\mathbf{A}%
_{n}}F^{-}=\min_{D_{n}}F^{-},$ and so $\{\mathbf{A}_{n}\}\subset \Lambda
^{-}(c)$ is another minimizing sequence of $e_{\mathcal{F}}^{-}$. Note that
the minimax level $e_{\mathcal{F}}^{-}$ depends on the parameter $\gamma $
and is decreasing as $|\gamma |\searrow 0^{+}$. Then for $|\gamma |$ small
enough, by Lemmas \ref{L2-6} and \ref{L4-13}, we obtain a Palais--Smale
sequence $\{\widetilde{u}_{n}\}$ for $F^{-}$ on $V_{r}$ at level $e_{%
\mathcal{F}}^{-}$ such that dist$_{H^{1}(\mathbb{R}^{N})}(\widetilde{u}_{n},%
\mathbf{A}_{n})\rightarrow 0$ as $n\rightarrow \infty .$ Next, by using the
standard arguments as in Lemma \ref{L4-15}, we easily obtain the conclusion.
\end{proof}

\begin{lemma}
\label{L4-17} Let $N\geq 2,\overline{p}<p<2_{\alpha }^{\ast },\overline{q}%
<q<2^{\ast }$ with $2Np-qN-2\alpha >0$. Then there exists a Palais--Smale
sequence $\{u_{n}\}\subset \Lambda _{r}^{-}(c)$ for $E_{\gamma ,\mu }$
restricted to $S(c)$ at the level $l_{r}^{-}(c)>0.$
\end{lemma}

\begin{proof}
Let $\{u_{n}\}\subset \Lambda _{r}^{-}(c)$. By Lemma \ref{L4-8}, we choose
the set $\overline{\mathcal{F}}$ of all singletons belonging to $V_{r}$ and $%
\Theta =\emptyset $, which is clearly a homotopy stable family of compact
subsets of $V_{r}$ (without boundary). Note that $e_{\overline{\mathcal{F}}%
}^{-}=\inf_{H\in \overline{\mathcal{F}}}\max_{u\in H}F^{-}(u)=\inf_{u\in
V_{r}}F^{-}(u)=\inf_{u\in \Lambda_{r}^{-}(c)}E_{\gamma ,\mu
}(u)=l_{r}^{-}(c).$ Then the lemma follows directly from Lemma \ref{L4-8}.
This completes the proof.
\end{proof}

\textbf{Now we give the proof of Theorem \ref{t4}:} Under the condition $(a)$
of Theorem \ref{t4}, by Lemma \ref{L4-18}, there exists a Palais--Smale
sequence $\{u_{n}\}\subset \mathcal{M}_{r}(c)=\mathcal{M}_{r}^{-}(c)$ for $%
E_{\gamma ,\mu }$ restricted to $S(c)$ at level $m_{r}^{-}(c)>0$, which is
bounded in $H^{1}(\mathbb{R}^{N})$ via Lemma \ref{L4-9}. So, according to
Lemma \ref{L2-9}, Problem $(P_{c})$ admits a radial solution $u^{-}$
satisfying $E_{\gamma ,\mu }(u^{-})=m_{r}^{-}(c)>0$ for some $\lambda >0$.

Under the condition $(b)$ of Theorem \ref{t4}, let $\{u_{n}\}\subset \Lambda
_{r}^{-}(c)$ be a Palais--Smale sequence for $E_{\gamma ,\mu }$ restricted
to $S(c)$ at level $l_{r}^{-}(c)>0$ whose existence is ensured by Lemma \ref%
{L4-17}. Lemma \ref{L4-1} shows that $\{u_{n}\}$ is bounded in $H^{1}(%
\mathbb{R}^{N})$. Hence, it follows from Lemma \ref{L2-9} that $(P_{c})$ has
a radial solution $u^{-}$ satisfying $E_{\gamma ,\mu }(u^{-})=l_{r}^{-}(c)>0$
for some $\lambda >0$. This completes the proof.

Finally, we show that $u^{-}$ has an exponential decay at infinity. Let
\begin{equation*}
\widehat{\Phi }(x)=Ce^{-\delta (|x|-M)}
\end{equation*}%
where $C=\max \left\{ |u(x)|\text{ }|\text{\ }|x|=M\right\} $ and $0<\delta <%
\sqrt{\lambda }$. Then we have $\Delta \widehat{\Phi }=\left( \delta ^{2}-%
\frac{\delta }{|x|}\right) \widehat{\Phi }$. Define
\begin{equation*}
\widehat{\Psi }_{R}:=\left\{
\begin{array}{ll}
0, & \text{ }x\in B_{R} \\
u^{-}-\widehat{\Phi }, & \text{ }x\in \mathbb{R}^{N}\backslash B_{R}.%
\end{array}%
\right.
\end{equation*}%
Then we have
\begin{equation*}
\int_{\mathbb{R}^{N}}(|\nabla \widehat{\Psi }_{R}|^{2}+\lambda ^{-}|\widehat{%
\Psi }_{R}|^{2})dx\leq \int_{\mathbb{R}^{N}}(\delta ^{2}-\frac{\delta }{|x|}%
-\lambda ^{-})\widehat{\Phi }\widehat{\Psi }_{R}dx+\int_{\mathbb{R}%
^{N}}g(u^{-})\widehat{\Psi }_{R}dx+o_{R}(1).
\end{equation*}%
where $g(u^{-})=\gamma (I_{\alpha }\ast |u^{-}|^{p})|u^{-}|^{p-2}u^{-}+\mu
|u^{-}|^{q-2}u^{-}$. The fact that the solution $u^{-}$ decay uniformly to
zero as $|x|\rightarrow +\infty $, we can take $R>0$ such that
\begin{equation*}
g(u^{-})\leq (\lambda ^{-}-\delta ^{2})|u^{-}(x)|\quad \text{for }|x|\geq R.
\end{equation*}%
Thus, we have
\begin{eqnarray*}
\lambda ^{-}\int_{|x|\geq R}|\widehat{\Psi }_{R}|^{2}dx &\leq &\int_{|x|\geq
R}(|\nabla \widehat{\Psi }_{R}|^{2}+\lambda ^{-}|\widehat{\Psi }_{R}|^{2})dx
\\
&\leq &\int_{|x|\geq R}(\lambda ^{-}-\delta ^{2})(u^{-}-\widehat{\Phi })%
\widehat{\Psi }_{R}dx \\
&\leq &(\lambda ^{-}-\delta ^{2})\int_{|x|\geq R}|\widehat{\Psi }%
_{R}|^{2}dx+o_{R}(1).
\end{eqnarray*}%
Consequently, $\widehat{\Psi }_{R}\equiv 0$. This implies that $|x|u^{-}\in
L^{2}(\mathbb{R}^{N})$.

\subsection{The case $\protect\gamma>0, \protect\mu<0$}

\begin{lemma}
\label{L5-0} Let $N\geq 2,\overline{p}<p<2_{\alpha }^{\ast },\overline{q}%
<q\leq 2^{\ast }$ with $qN+2\alpha -2Np\leq 0.$ Then there is a unique $%
s^{-}(u)>0$ such that $u_{s^{-}(u)}\in \mathcal{M}(c)=\mathcal{M}^{-}(c)$
and
\begin{equation*}
E_{\gamma ,\mu }(u_{s^{-}(u)})=\sup_{s>0}E_{\gamma ,\mu }(u_{t})>0.
\end{equation*}
\end{lemma}

\begin{proof}
The proof is similar to Lemma \ref{L4-19}, so we omit it.
\end{proof}

\begin{lemma}
\label{L5-8} Let $N\geq 2,\overline{p}<p<2_{\alpha }^{\ast },\overline{q}%
<q\leq 2^{\ast }$ with $qN+2\alpha -2Np\leq 0.$ Then the functional $%
E_{\gamma ,\mu }$ is bounded from below by a positive constant and coercive
on $\mathcal{M}(c)=\mathcal{M}^{-}(c)$ for all $c>0$. Furthermore, it holds%
\begin{equation*}
m^{-}(c):=\inf_{u\in \mathcal{M}^{-}(c)}E_{\gamma ,\mu }(u)=\inf_{u\in
\mathcal{M}_{r}^{-}(c)}E_{\gamma ,\mu }(u).
\end{equation*}
\end{lemma}

\begin{proof}
Let $u\in \mathcal{M}(c)=\mathcal{M}^{-}(c)$. Using (\ref{e2-3}) and (\ref%
{e2-19}) gives%
\begin{equation*}
A(u)=\frac{\gamma N}{N+\alpha +2}B(u)-\frac{|\mu |N(q-2)}{2q}C(u)\leq \frac{%
\gamma \overline{B}N c^{\frac{N+\alpha -(N-2)p}{2}}}{(N+\alpha +2)}A(u)^{%
\frac{Np-N-\alpha }{2}},
\end{equation*}%
which implies that there exists a constant $D_{3}>0$ depending on $\gamma $
and $c,$ such that $A(u)\geq D_{3}$, since $p>\overline{p}.$ Then by (\ref%
{e2-19}), we have
\begin{eqnarray*}
E_{\gamma ,\mu }(u) &=&\frac{Np-N-\alpha -2}{2(Np-N-\alpha )}A(u)+\frac{|\mu
|(2Np-2\alpha -qN)}{2q(Np-N-\alpha )}C(u) \\
&>&\frac{Np-N-\alpha -2}{2(Np-N-\alpha )}D_{3}>0,
\end{eqnarray*}%
which shows that $E_{\gamma ,\mu }$ is bounded from below by a positive
constant and coercive on $\mathcal{M}(c)$.

Next, we claim that $\inf_{u\in \mathcal{M}^{-}(c)}E_{\gamma ,\mu
}(u)=\inf_{u\in \mathcal{M}_{r}^{-}(c)}E_{\gamma ,\mu }(u).$ Since $\mathcal{%
M}_{r}^{-}(c)\subset \mathcal{M}^{-}(c),$ we have $\inf_{u\in \mathcal{M}%
^{-}(c)}E_{\gamma ,\mu }(u)\leq \inf_{u\in \mathcal{M}_{r}^{-}(c)}E_{\gamma
,\mu }(u).$ Thus, it is enough to prove that $\inf_{u\in \mathcal{M}%
^{-}(c)}E_{\gamma ,\mu }(u)\geq \inf_{u\in \mathcal{M}_{r}^{-}(c)}E_{\gamma
,\mu }(u).$ By Lemma \ref{L5-0}, we have%
\begin{equation}
\inf_{u\in \mathcal{M}^{-}(c)}E_{\gamma ,\mu }(u)=\inf_{u\in
S(c)}\sup_{0<t\leq s^{-}(u)}E_{\gamma ,\mu }(u_{s}).  \label{e5-8}
\end{equation}

Let $u\in S(c)$ and $|u|^{\ast }\in S_{r}(c)$ be Schwarz rearrangement of $%
|u|.$ Clearly, $A(|u|^{\ast })\leq A(u)$ and $C(|u|^{\ast })=C(u)$ (see \cite%
{BH}). Moreover, by the Riesz's rearrangement inequality one has $%
B(|u|^{\ast })\geq B(u).$ Then for all $t>0$ we have%
\begin{eqnarray}
E_{\gamma ,\mu }((|u|^{\ast })_{t}) &=&\frac{s^{2}}{2}A(|u|^{\ast })-\frac{%
\gamma s^{Np-N-\alpha }}{2p}B(|u|^{\ast })-\frac{\mu s^{N(q-2)/2}}{q}%
C(|u|^{\ast })  \notag \\
&\leq &\frac{s^{2}}{2}A(u)-\frac{\gamma s^{Np-N-\alpha }}{2p}B(u)-\frac{\mu
s^{N(q-2)/2}}{q}C(u)=E_{\gamma ,\mu }(u_{t}).  \label{e5-4}
\end{eqnarray}%
Note that $g_{|u|^{\ast }}^{\prime }(0)=g_{u}^{\prime }(0)=0$ and $%
g_{|u|^{\ast }}^{\prime \prime }(s)\leq g_{u}^{\prime \prime }(s)$ for $s>0.$
This implies that $0<s^{-}(|u|^{\ast })\leq s^{-}(u).$ Hence, it follows
from (\ref{e5-4}) that
\begin{equation*}
\sup_{0<s\leq s^{-}(|u|^{\ast })}E_{\gamma ,\mu }((|u|^{\ast })_{s})\leq
\sup_{0<s\leq s^{-}(u)}E_{\gamma ,\mu }(u_{s}).
\end{equation*}%
Using this, together with (\ref{e5-8}), leads to $\inf_{u\in \mathcal{M}%
^{-}(c)}E_{\gamma ,\mu }(u)\geq \inf_{u\in \mathcal{M}_{r}^{-}(c)}E_{\gamma
,\mu }(u)$. This completes the proof.
\end{proof}

Next, similar to the arguments in Section 4.1, by Lemma \ref{L5-0} we apply
Lemma \ref{L2-6} to construct a Palais--Smale sequence $\{u_{n}\}\subset
\mathcal{M}_{r}^{-}(c)$ for $E_{\gamma ,\mu }$ restricted to $S(c)$ at level
$m^{-}(c)$. Here we only state the conclusion without proof.

\begin{lemma}
\label{L5-3} Let $\mathcal{F}$ be a homotopy stable family of compact
subsets of $S_{r}(c)$ with closed boundary $\Theta $ and let
\begin{equation*}
e_{\mathcal{F}}^{-}:=\inf_{H\in \mathcal{F}}\max_{u\in H}J^{-}(u),
\end{equation*}%
where $J^{-}(u)=E_{\gamma ,\mu }(u_{t^{-}}).$ Assume that $\Theta $ is
contained in a connected component of $\mathcal{M}_{r}^{-}(c)$ and $\max
\{\sup J^{-}(\Theta ),0\}<e_{\mathcal{F}}^{-}<\infty $. Then there exists a
Palais--Smale sequence $\{u_{n}\}\subset \mathcal{M}_{r}^{-}(c)$ for $%
E_{\gamma ,\mu }$ restricted to $S(c)$ at level $e_{\mathcal{F}}^{-}$.
\end{lemma}

In view of Lemma \ref{L5-8} we have%
\begin{equation*}
e_{\mathcal{F}}^{-}=\inf_{H\in \mathcal{F}}\max_{u\in H}J^{-}(u)=\inf_{u\in
S_{r}(c)}J^{-}(u)=\inf_{u\in \mathcal{M}_{r}^{-}(c)}E_{\gamma ,\mu
}(u)=\inf_{u\in \mathcal{M}^{-}(c)}E_{\gamma ,\mu }(u)=m^{-}(c).
\end{equation*}%
Then the following result follows from Lemma \ref{L5-3}.

\begin{lemma}
\label{L5-21} Assume that $N\geq 2,\gamma >0,\mu <0$ and that one of the two
following conditions holds:\newline
$(i)$ $\overline{p}<p<2_{\alpha }^{\ast },\overline{q}<q<2^{\ast }$ with $%
2Np-Nq-\alpha q>0$ and $|\mu |$ is sufficiently small;\newline
$(ii)$ $\overline{p}<p<2_{\alpha }^{\ast },\overline{q}<q\leq 2^{\ast }$
with $2Np-Nq-\alpha q\leq 0\leq 2Np-Nq-2\alpha .$\newline
Then there exists a Palais--Smale sequence $\{u_{n}\}\subset \mathcal{M}%
_{r}^{-}(c)$ for $E_{\gamma ,\mu }$ restricted to $S(c)$ at level $%
m^{-}(c)>0 $.
\end{lemma}

For the case of $\overline{p}<p<2_{\alpha }^{\ast }$ and $\overline{q}<q\leq
2^{\ast }$ with $2Np-Nq-2\alpha <0$, it is very difficult to verify that $%
\mathcal{M}^{0}(c)=\emptyset $. Following the idea in Section 4.1, we
introduce a novel natural constraint manifold.

Set%
\begin{equation*}
\Gamma _{\ast }:=\frac{p(N(q-2)-4)}{Nq+2\alpha -2Np}\left( \frac{|\mu
|N(Nq+2\alpha -2Np)(q-2)\overline{S}c^{\frac{2N-q(N-2)}{4}}}{4q(Np-N-\alpha
-2)}\right) ^{\frac{2(Np-N-\alpha -2)}{N(q-2)-4}}.
\end{equation*}%
Clearly, $\Gamma _{\ast }>0$ and $\Gamma _{\ast }\rightarrow 0$ as $%
c\rightarrow 0,$ since $\overline{p}<p<2_{\alpha }^{\ast }$ and $\overline{q}%
<q\leq 2^{\ast }$ with $2Np-Nq-2\alpha <0.$ We define the set%
\begin{eqnarray}
U:= &&\left\{ u\in S(c)\text{ }|\text{ }\gamma B(u)>\Gamma _{\ast }A(u)^{%
\frac{Np-N-\alpha }{2}}\right\}=\left\{ u\in S(c)\text{ }|\text{ }\left(
s_{u}^{\star }\right) ^{2}A(u)<K\right\} ,  \label{e5-3}
\end{eqnarray}%
where
\begin{equation*}
s_{u}^{\star }:=\left( \frac{2p(N(q-2)-4)A(u)}{\gamma (Np-N-\alpha
)(Nq+2\alpha -2Np)B(u)}\right) ^{\frac{1}{Np-N-\alpha -2}},  \label{e5-12}
\end{equation*}%
and
\begin{equation*}
K:=\left( \frac{2}{Np-N-\alpha }\right) ^{\frac{2}{Np-N-\alpha -2}}\left(
\frac{4q(Np-N-\alpha -2)}{|\mu |N(Nq+2\alpha -2Np)(q-2)\overline{S}c^{\frac{%
2N-q(N-2)}{4}}}\right) ^{\frac{4}{N(q-2)-4}}.
\end{equation*}

Set
\begin{eqnarray*}
c_{\ast } &:&=\left[ \left( \frac{Np-N-\alpha -2}{2\kappa (Np-N-\alpha )}%
\right) ^{1/2}\left( \frac{2(Nq+2\alpha -2Np)}{(N(q-2)-4)(Np-N-\alpha )}%
\right) ^{\frac{1}{Np-N-\alpha -2}}\right] ^{\frac{2(N(q-2)-4)}{2N-q(N-2)}}
\\
&&\text{ \ \ }\times \left( \frac{4q(Np-N-\alpha -2)}{|\mu
|N(q-2)(Nq+2\alpha -2Np)\overline{S}}\right) ^{\frac{2(N(q-2)-4)}{2N-q(N-2)}%
},
\end{eqnarray*}%
where $\kappa>0$ is ground state energy for the following Choquard equation
\cite{BLL}:
\begin{equation}
\left\{
\begin{array}{ll}
-\Delta u+\lambda u=\gamma (I_{\alpha }\ast |u|^{p})|u|^{p-2}u & \text{ in }%
\mathbb{R}^{N}, \\
\int_{\mathbb{R}^{N}}|u|^{2}dx=c>0, &
\end{array}%
\right.  \tag{$E_{\gamma }^{\infty }$}
\end{equation}%
For $0<c<c_{\ast }$, we have $U$ is nonempty.

Next, we define a set $\Phi (c):=\{u\in U$ $|$\ $g_{u}^{\prime }(1)=0\},$
which appears a natural constraint. Similarly, we can split $\Phi (c)$ into
the three following parts, i.e. $\Phi (c)=\Phi ^{+}(c)\cup \Phi ^{-}(c)\cup
\Phi ^{0}(c)$. Furthermore, we have the following results.

\begin{lemma}
\label{L5-12}Let $N\geq 2,\overline{p}<p<2_{\alpha }^{\ast },\overline{q}%
<q\leq 2^{\ast }$ with $Nq+2\alpha -2Np>0$. Then $\Phi ^{0}(c)=\emptyset $
for all $c>0.$
\end{lemma}

\begin{proof}
Assume on the contrary, that is $\Phi ^{0}(c)\neq \emptyset $ for some $c>0.$
Then for $u\in \Phi ^{0}(c)$, by (\ref{e2-1}) we have
\begin{eqnarray*}
(Np-N-\alpha -2)A(u) &=&\frac{|\mu |N(q-2)(Nq+2\alpha -2Np)}{4q}C(u) \\
&\leq &\frac{|\mu |N(q-2)(Nq+2\alpha -2Np)\overline{S}}{4q}c^{\frac{2N-q(N-2)%
}{4}}A(u)^{\frac{N(q-2)}{4}},
\end{eqnarray*}%
which indicates that
\begin{equation}
A(u)\geq \left[ \frac{4q(Np-N-\alpha -2)}{|\mu |N(q-2)(Nq+2\alpha -2Np)%
\overline{S}}c^{-\frac{2N-q(N-2)}{4}}\right] ^{\frac{4}{N(q-2)-4}}.
\label{e5-22}
\end{equation}%
On the other hand, we note that
\begin{equation*}
\frac{N(q-2)-4}{2}A(u)=\frac{\gamma (Np-N-\alpha )(Nq+2\alpha -2Np)}{4p}B(u)
\end{equation*}%
and $\gamma B(u)>\Gamma _{\ast }A(u)^{\frac{Np-N-\alpha }{2}}.$ Then it
follows that
\begin{equation*}
A(u)<\left( \frac{2}{Np-N-\alpha }\right) ^{\frac{2}{Np-N-\alpha -2}}\left(
\frac{4q(Np-N-\alpha -2)}{|\mu |N(q-2)(Nq+2\alpha -2Np)\overline{S}}c^{-%
\frac{2N-q(N-2)}{4}}\right) ^{\frac{4}{N(q-2)-4}},
\end{equation*}%
which contradicts with (\ref{e5-22}). The proof is complete.
\end{proof}

\begin{lemma}
\label{L5-6} Let $N\geq 2,\overline{p}<p<2_{\alpha }^{\ast },\overline{q}%
<q\leq 2^{\ast }$ with $Nq+2\alpha -2Np>0$. Then the functional $E_{\gamma
,\mu }$ is bounded from below\ and coercive on $S(c)$ for all $c>0.$
Moreover, the functional $E_{\gamma ,\mu }$ is bounded from below by a
positive constant on $\Phi ^{-}(c)$ for all $c>0.$
\end{lemma}

\begin{proof}
By (\ref{e2-3}) and H\"{o}lder's inequality, we have
\begin{equation*}
B(u)\leq \widetilde{C}\Vert u\Vert _{\frac{2Np}{N+\alpha }}^{2p}\leq
\widetilde{C}c^{\frac{q(N+\alpha )-2Np}{N(q-2)}}\Vert u\Vert _{q}^{\frac{%
2q(Np-N-\alpha )}{N(q-2)}},
\end{equation*}%
leading to
\begin{equation*}
C(u)\geq \widetilde{C}^{-\frac{N(q-2)}{2(Np-N-\alpha )}}c^{\frac{%
2Np-q(N+\alpha )}{2(Np-N-\alpha )}}B(u)^{\frac{N(q-2)}{2(Np-N-\alpha )}}.
\end{equation*}%
Then it holds
\begin{eqnarray*}
E_{\gamma ,\mu }(u) &=&\frac{1}{2}A(u)-\frac{\gamma }{2p}B(u)+\frac{|\mu |}{q%
}C(u) \\
&\geq &\frac{1}{2}A(u)-\frac{\gamma }{2p}B(u)+\frac{|\mu |}{q}\widetilde{C}%
^{-\frac{N(q-2)}{2(Np-N-\alpha )}}c^{\frac{2Np-q(N+\alpha )}{2(Np-N-\alpha )}%
}B(u)^{\frac{N(q-2)}{2(Np-N-\alpha )}}.
\end{eqnarray*}%
which implies that $E_{\gamma ,\mu }$ is bounded from below and coercive on $%
S(c)$ for all $c>0$, since $\frac{N(q-2)}{2(Np-N-\alpha )}>1$. Indeed, the
quantity
\begin{equation*}
-\frac{\gamma }{2p}B(u)+\frac{|\mu |}{q}\widetilde{C}^{-\frac{N(q-2)}{%
2(Np-N-\alpha )}}c^{\frac{2Np-q(N+\alpha )}{2(Np-N-\alpha )}}B(u)^{\frac{%
N(q-2)}{2(Np-N-\alpha )}}
\end{equation*}%
has a negative lower bound on $S(c)$.

Let $u\in \Phi ^{-}(c)$. Then we have $A(u)-\frac{\gamma (Np-N-\alpha )}{2p}%
B(u)+\frac{|\mu |N(q-2)}{2q}C(u)=0$. It follows from (\ref{e2-3}) that%
\begin{eqnarray*}
A(u) &=&\frac{\gamma (Np-N-\alpha )}{2p}B(u)-\frac{|\mu |N(q-2)}{2q}C(u) \\
&\leq &\frac{\gamma \overline{B}(Np-N-\alpha )}{2p}c^{\frac{N+\alpha -p(N-2)%
}{2}}A(u)^{\frac{Np-N-\alpha }{2}},
\end{eqnarray*}%
which implies that $A(u)\geq D_{4}$ for some $D_{4}>0$. Using this, we have
\begin{eqnarray*}
E_{\gamma ,\mu }(u) &=&\frac{N(q-2)-4}{2N(q-2)}A(u)-\frac{\gamma (Nq+2\alpha
-2Np)}{2pN(q-2)}B(u) \\
&>&\frac{(Np-N-\alpha -2)(N(q-2)-4)}{2N(Np-N-\alpha )(q-2)}D_{4}>0.
\end{eqnarray*}%
Hence, $E_{\gamma ,\mu }$ is bounded from below by a positive constant on $%
\Phi ^{-}(c)$ for all $c>0.$ This completes the proof.
\end{proof}

\begin{lemma}
\label{L5-13} Let $N\geq 2,\overline{p}<p<2_{\alpha }^{\ast },\overline{q}%
<q\leq 2^{\ast }$ with $Nq+2\alpha -2Np>0$. Then for each $u\in U,$ we have $%
\inf_{s\geq 0}E_{\gamma ,\mu }(u_{s})<0.$
\end{lemma}

\begin{proof}
Let $u\in U.$ Then for any $s>0$, we have that
\begin{equation*}
E_{\gamma ,\mu }(u_{s})=s^{\frac{N(q-2)}{2}}\left[ \frac{1}{2}A(u)s^{\frac{%
4-N(q-2)}{2}}-\frac{\gamma }{2p}B(u)s^{\frac{2Np-Nq-2\alpha }{2}}+\frac{|\mu
|}{q}C(u)\right] .
\end{equation*}%
Let
\begin{equation*}
k(t):=\frac{1}{2}A(u)s^{\frac{4-N(q-2)}{2}}-\frac{\gamma }{2p}B(u)s^{\frac{%
2Np-Nq-2\alpha }{2}}.
\end{equation*}%
Clearly, $E_{\gamma ,\mu }(u_{s})=0$ if and only if $k\left( s\right) +\frac{%
|\mu |}{q}C(u)=0.$ It is easily seen that $k(s_{0})=0,\lim_{s\rightarrow
0^{+}}k(s)=\infty $ and $\lim_{s\rightarrow \infty }k(s)=0,$ where $%
s_{0}:=\left( \frac{pA(u)}{\gamma B(u)}\right) ^{1/(Np-N-\alpha -2)}$. By
calculating the derivative of $k(t)$, we get
\begin{equation*}
k^{\prime }(s)=s^{\frac{2-N(q-2)}{2}}\left[ \frac{\gamma (Nq+2\alpha
-2Np)s^{Np-N-\alpha -2}}{4p}B(u)-\frac{(q-2)N-4}{4}A(u)\right] .
\end{equation*}%
This indicates that $k(t)$ is decreasing when $0<s<\left( \frac{%
p(N(q-2)-4)A(u)}{\gamma (Nq+2\alpha -2Np)B(u)}\right) ^{1/(Np-N-\alpha -2)}$
and is increasing when $s>\left( \frac{p(N(q-2)-4)A(u)}{\gamma (Nq+2\alpha
-2Np)B(u)}\right) ^{1/(Np-N-\alpha -2)}$, and so we have%
\begin{eqnarray*}
\inf_{s>0}k(s) &=&-\frac{Np-N-\alpha -2}{Nq+2\alpha -2Np}\left( \frac{%
p(N(q-2)-4)A(u)}{\gamma (Nq+2\alpha -2Np)B(u)}\right) ^{\frac{4-N\left(
q-2\right) }{2\left( Np-N-\alpha -2\right) }}A(u) \\
&<&-\frac{|\mu| \overline{S}}{q}c^{\frac{2N-q(N-2)}{4}}A(u)^{\frac{N(q-2)}{4}%
}<-\frac{|\mu| }{q}C(u),
\end{eqnarray*}%
which implies that there exist two constants $\hat{s}^{(i)}>0$ $(i=1,2)$
satisfying
\begin{equation*}
\hat{s}^{(1)}<\left( \frac{p(N(q-2)-4)A(u)}{\gamma (Nq+2\alpha -2Np)B(u)}%
\right) ^{1/(Np-N-\alpha -2)}<\hat{s}^{(2)}
\end{equation*}%
such that $E_{\gamma ,\mu }\left( u_{\hat{s}^{(i)}}\right) =0$ for $i=1,2.$
Moreover, we get
\begin{equation*}
E_{\gamma ,\mu }\left[ \left( \frac{p(N(q-2)-4)A(u)}{\gamma (Nq+2\alpha
-2Np)B(u)}\right) ^{1/(Np-N-\alpha -2)}\right] <0,
\end{equation*}%
which implies that $\inf_{s\geq 0}E_{\gamma ,\mu }(u_{s})<0.$ The proof is
complete.
\end{proof}

For $u\in U$, we define
\begin{equation*}
\widetilde{s}_{u}=\left( \frac{2pA( u) }{\gamma ( Np-N-\alpha ) B( u) }%
\right) ^{1/(Np-N-\alpha -2)}.
\end{equation*}
By similar arguments in Lemma \ref{L4-3}, we have the following result.

\begin{lemma}
\label{L5-7} Let $N\geq 2,\overline{p}<p<2_{\alpha }^{\ast },\overline{q}%
<q\leq 2^{\ast }$ with $Nq+2\alpha -2Np>0$. Then for each $u\in U$ there
exist two positive constants $s_{\star }^{+}(u)$ and $s_{\star }^{-}(u)$
satisfying $\widetilde{s}_{u}<s_{\star }^{-}(u)<s_{u}^{\star }<s_{\star
}^{+}(u)$ such that $u_{s_{\star }^{\pm }(u)}\in \Phi ^{\pm }(c),$ where $%
s_{u}^{\star }$ is as in (\ref{e5-12}). Furthermore, we have $E_{\gamma ,\mu
}(u_{s_{\star }^{+}(u)})=\inf_{s\geq s_{\star }^{-}(u)}E_{\gamma ,\mu
}(u_{t})=\inf_{s\geq 0}E_{\gamma ,\mu }(u_{s})<0.$
\end{lemma}

We now define%
\begin{equation*}
U_{r}:=U\cap H_{r}^{1}(\mathbb{R}^{N}),\text{ }\Phi _{r}(c):=\Phi (c)\cap
H_{r}^{1}(\mathbb{R}^{N})\text{ and }\Phi _{r}^{\pm }(c):=\Phi ^{\pm
}(c)\cap H_{r}^{1}(\mathbb{R}^{N}).
\end{equation*}%
Similar to Lemma \ref{L5-8}, by Lemmas \ref{L5-13}--\ref{L5-7}, we define%
\begin{equation}
d^{-}(c):=\inf_{u\in \Phi ^{-}(c)}E_{\gamma ,\mu }(u)=\inf_{u\in \Phi
_{r}^{-}(c)}E_{\gamma ,\mu }(u)>0,  \label{e5-9}
\end{equation}%
and%
\begin{equation}
d^{+}(c):=\inf_{u\in \Phi ^{+}(c)}E_{\gamma ,\mu }(u)=\inf_{u\in \Phi
_{r}^{+}(c)}E_{\gamma ,\mu }(u)<0.  \label{e5-20}
\end{equation}

Since $U_{r}$ has a boundary, in order to guarantee that our deformation
arguments happen inside $U_{r},$ we need the following lemma.

\begin{lemma}
\label{L5-18} Let $N\geq 2,\overline{p}<p<2_{\alpha }^{\ast },\overline{q}%
<q\leq 2^{\ast }$ with $Nq+2\alpha -2Np>0$. If $\{w_{n}\}\subset U_{r}$ is a
sequence with $w_{n}\rightarrow w\in \partial U_{r}$ strongly in $H^{1}(%
\mathbb{R}^{N})$, then the following statements are true.\newline
$(i)$ There exist a sequence $\{s_{\star }^{-}(w_{n})\}$ with $s_{\star
}^{-}(w_{n})\rightarrow s_{\infty }^{-}>0$ satisfying $g_{w_{s_{\infty
}^{-}}}^{\prime \prime }(1)\leq 0$ and a constant $\widetilde{M}=\widetilde{M%
}(\mu ,c)>0$ satisfying $\widetilde{M}\rightarrow \infty $ as $c\rightarrow
0 $ such that $E_{\gamma ,\mu }(w_{s_{\infty }^{-}})>\widetilde{M}$ for $%
0<c<c_{\ast };$\newline
$(ii)$ There exist a sequence $\{s_{\star }^{+}(w_{n})\}$ with $s_{\star
}^{+}(w_{n})\rightarrow s_{\infty }^{+}>0$ satisfying $g_{w_{s_{\infty
}^{+}}}^{\prime \prime }(1)\geq 0$ and a constant $\widetilde{M}>0$ large
enough such that $E_{\gamma ,\mu }(w_{s_{\infty }^{+}})>\widetilde{M}$ for $%
c>0$ sufficiently small;\newline
$(iii)$ There exists a sequence $\{s_{\star }^{+}(w_{n})\}$ with $s_{\star
}^{+}(w_{n})\rightarrow \infty $ such that $E_{\gamma ,\mu
}((w_{n})_{s_{\star }^{+}(w_{n})})\rightarrow \infty $ for $0<c<c_{\ast }.$
\end{lemma}

\begin{proof}
We only prove the second and third statements, since the first one is very
similar to that of Lemma \ref{L4-13}.

Let $\{w_{n}\}\subset U_{r}$ such that $w_{n}\rightarrow w\in \partial U_{r}$
strongly in $H^{1}(\mathbb{R}^{N})$ as $n\rightarrow \infty $. Then it holds
$s_{w_{n}}^{\star }\rightarrow s_{w}^{\star }\neq 0\quad $and$\quad
\widetilde{s}_{w_{n}}\rightarrow \widetilde{s}_{w}\neq 0\quad $as$%
\,n\rightarrow \infty .$ Since $w_{n}\in U_{r}$ and $w\in \partial U_{r}$,
we have $(s_{w_{n}}^{\star })^{2}A(w_{n})<K$ and $(s_{w}^{\star })^{2}A(w)=K$%
. Moreover, by Lemma \ref{L5-7} one has $\widetilde{s}_{w_{n}}<s_{\star
}^{-}(w_{n})\leq s_{w_{n}}^{\star }\leq s_{\star }^{+}(w_{n}).$

Next, we consider two separate cases.\newline
Case $(I):\{s_{\star }^{+}(w_{n})\}$ is bounded. Then we can assume that $%
\{s_{\star }^{+}(w_{n})\}$ up to a subsequence, converges to $s_{\infty
}^{+}\neq 0$. Since $w_{n}\rightarrow w$ strongly in $H^{1}(\mathbb{R}^{N})$
and $Q((w_{n})_{s_{\star }^{+}(w_{n})})=0$, we obtain that $Q(w_{s_{\infty
}^{+}})=0$, which implies that
\begin{equation*}
A(w_{s_{\infty }^{+}})\leq \frac{\gamma (Np-N-\alpha )}{2p}B(w_{s_{\infty
}^{+}})=\frac{\Gamma _{\ast }(Np-N-\alpha )}{2p}A(w_{s_{\infty }^{+}})^{%
\frac{Np-N-\alpha }{2}},
\end{equation*}%
leading to
\begin{equation}
A(w_{s_{\infty }^{+}})\geq \left( \frac{2p}{\Gamma _{\ast }(Np-N-\alpha )}%
\right) ^{\frac{2}{Np-N-\alpha -2}}.  \label{e5-13}
\end{equation}%
Since $g_{(w_{n})_{s_{\star }^{+}(w_{n})}}^{\prime \prime }(1)>0$, we have $%
g_{w_{s_{\infty }^{+}}}^{\prime \prime }(1)\geq 0$, namely,
\begin{equation}
\frac{(Np-N-\alpha )(Nq+2\alpha -2Np)}{2p(N(q-2)-4)}B(w_{s_{\infty
}^{+}})\geq A(w_{s_{\infty }^{+}}).  \label{e5-14}
\end{equation}%
It follows from (\ref{e5-13}) and (\ref{e5-14}) that%
\begin{equation}
B(w_{s_{\infty }^{+}})>\frac{2p(N(q-2)-4)}{(Np-N-\alpha )(Nq+2\alpha -2Np)}%
\left( \frac{2p}{\Gamma _{\ast }(Np-N-\alpha )}\right) ^{\frac{2}{%
Np-N-\alpha -2}}.  \label{e5-19}
\end{equation}%
Moreover, from Lemma \ref{L5-6}, we observe that
\begin{equation}
E_{\gamma ,\mu }(w_{s_{\infty }^{+}})\geq \frac{|\mu |}{q}\widetilde{C}^{-%
\frac{N(q-2)}{2(Np-N-\alpha )}}c^{\frac{2Np-q(N+\alpha )}{2(Np-N-\alpha )}%
}B(w_{s_{\infty }^{+}})^{\frac{N(q-2)}{2(Np-N-\alpha )}}-\frac{\gamma }{2p}%
B(w_{s_{\infty }^{+}}).  \label{e5-17}
\end{equation}%
Then it follows from (\ref{e5-19})--(\ref{e5-17}) that there exists a
constant $\widetilde{M}>0$ large enough such that $E_{\gamma ,\mu
}(w_{s_{\infty }^{+}})>\widetilde{M}$ for $c>0$ sufficiently small. Thus,
the second statement are proved.\newline
Case $(II):\{s_{\star }^{+}(w_{n})\}$ is unbounded. By Lemma \ref{L5-6}, we
have
\begin{eqnarray*}
E_{\gamma ,\mu }((w_{n})_{s_{\star }^{+}(w_{n})}) &\geq &\frac{|\mu |}{q}%
\widetilde{C}^{-\frac{N(q-2)}{2(Np-N-\alpha )}}c^{\frac{2Np-q(N+\alpha )}{%
2(Np-N-\alpha )}}\left( s_{\star }^{+}(w_{n})\right) ^{\frac{N(q-2)}{2}%
}B(w_{n})^{\frac{N(q-2)}{2(Np-N-\alpha )}} \\
&&-\frac{\gamma }{2p}\left( t_{\star }^{+}(w_{n})\right) ^{Np-N-\alpha
}B(w_{n}),
\end{eqnarray*}%
which implies that $E_{\gamma ,\mu }((w_{n})_{s_{\star
}^{+}(w_{n})})\rightarrow \infty ,$ since $Nq+2\alpha -2Np>0$. This
completes the proof.
\end{proof}

Similar to Lemma \ref{L4-8}, we give the following lemma.

\begin{lemma}
\label{L5-19} Let $\mathcal{F}$ be a homotopy stable family of compact
subsets of $U_{r}$ with closed boundary $B$ and let
\begin{equation*}
e_{\mathcal{F}}^{-}:=\inf_{H\in \mathcal{F}}\max_{u\in H}I^{-}(u).
\end{equation*}%
Suppose that $\Theta $ is contained in a connected component of $\Phi
_{r}^{-}(c)$ and $\max \{\sup I^{-}(\Theta ),0\}<e_{\mathcal{F}}^{-}<\infty $%
. Then there exists a Palais--Smale sequence $\{u_{n}\}\subset \Phi
_{r}^{-}(c)$ for $E_{\gamma ,\mu }$ restricted to $U_{r}$ at level $e_{%
\mathcal{F}}^{-}$ for $c>0$ sufficiently small.
\end{lemma}

In view of (\ref{e5-9}) one has%
\begin{equation*}
e_{\mathcal{F}}^{-}=\inf_{H\in \mathcal{F}}\max_{u\in H}I^{-}(u)=\inf_{u\in
U_{r}}I^{-}(u)=\inf_{u\in \Phi _{r}^{-}(c)}E_{\gamma ,\mu }(u)=\inf_{u\in
\Phi ^{-}(c)}E_{\gamma ,\mu }(u)=d^{-}(c).
\end{equation*}%
Then by Lemma \ref{L5-19}, similar to Lemma \ref{L4-17}, we have the
following result.

\begin{lemma}
\label{L5-20} Let $N\geq 2,\overline{p}<p<2_{\alpha }^{\ast },\overline{q}%
<q\leq 2^{\ast }$ with $Nq+2\alpha -2Np>0$. Then there exists a
Palais--Smale sequence $\{u_{n}\}\subset \Phi _{r}^{-}(c)$ for $E_{\gamma
,\mu }$ restricted to $S(c)$ at the level $d^{-}(c)$ for $c>0$ sufficiently
small.
\end{lemma}

For each $u\in U_{r}$, using the argument in the proof of Lemma \ref{L5-6}
one has
\begin{equation*}
E_{\gamma ,\mu }(u)\geq \frac{|\mu |}{q}\widetilde{C}^{-\frac{N(q-2)}{%
2(Np-N-\alpha )}}c^{\frac{2Np-q(N+\alpha )}{2(Np-N-\alpha )}}B(u)^{\frac{%
N(q-2)}{2(Np-N-\alpha )}}-\frac{\gamma }{2p}B(u).
\end{equation*}%
Define
\begin{equation*}
\overline{h}(s):=\frac{|\mu |}{q}\widetilde{C}^{-\frac{N(q-2)}{2(Np-N-\alpha
)}}c^{\frac{2Np-q(N+\alpha )}{2(Np-N-\alpha )}}s^{\frac{N(q-2)}{%
2(Np-N-\alpha )}}-\frac{\gamma }{2p}s\text{ for }s>0.
\end{equation*}%
Since $Nq+2\alpha -2Np>0,$ a direct calculation shows that $\overline{h}%
(R_{0})=0$ and $\overline{h}(s)>0$ when $s>R_{0},$ where%
\begin{equation*}
R_{0}:=\left[ \frac{\gamma q}{2p|\mu |}\widetilde{C}^{\frac{N(q-2)}{%
2(Np-N-\alpha )}}c^{-\frac{2Np-q(N+\alpha )}{2(Np-N-\alpha )}}\right] ^{%
\frac{2(Np-N-\alpha )}{Nq-2Np+2\alpha }}>0.
\end{equation*}%
This shows that $E_{\gamma ,\mu }(u)>0$ if $B(u)>R_{0}.$ Moreover, it
follows from Lemma \ref{L5-7} that the set $\Phi _{r}^{+}(c)$ is contained
in $\mathbf{B}_{R_{0}}:=\left\{ u\in U_{r}\text{ }|\text{ }%
B(u)<R_{0}\right\} .$ Note that $\gamma B(u)>\Gamma _{\ast }A(u)^{\frac{%
Np-N-\alpha }{2}}$ for all $u\in U.$ Then we have $E_{\gamma ,\mu }(u)>0$ if
$A(u)\geq r_{0}:=\left( \frac{\gamma R_{0}}{\Gamma _{\ast }}\right) ^{\frac{2%
}{Np-N-\alpha }}.$ Clearly, $r_{0}\rightarrow \infty $ as $c\rightarrow 0.$
Now we define
\begin{equation*}
\mathbf{A}_{r_{0}}:=\left\{ u\in U_{r}\text{ }|\text{ }A(u)<r_{0}\right\} .
\end{equation*}%
Clearly, $\mathbf{B}_{R_{0}}\subset \mathbf{A}_{r_{0}}.$ Furthermore, we
have the following result.

\begin{lemma}
\label{L5-22} Let $N\geq 2,\overline{p}<p<2_{\alpha }^{\ast },\overline{q}%
<q\leq 2^{\ast }$ with $Nq+2\alpha -2Np>0$. Then we have $-\infty <\overline{%
m}(c):=\inf_{u\in \mathbf{A}_{r_{0}}}E_{\gamma ,\mu }(u)<0$ for all $c>0$.
Furthermore, it holds
\begin{equation*}
\overline{m}(c)=\inf_{u\in \Phi _{r}^{+}(c)}E_{\gamma ,\mu }(u)=d^{+}(c).
\end{equation*}
\end{lemma}

\begin{proof}
For $u\in \mathbf{A}_{r_{0}}$, by (\ref{e2-3}), we have%
\begin{equation}
E_{\gamma ,\mu }(u)\geq \frac{1}{2}A(u)-\frac{\gamma \overline{B}}{2p}c^{%
\frac{N+\alpha -p(N-2)}{2}}A(u)^{\frac{Np-N-\alpha }{2}}.  \notag
\end{equation}%
leading to $\overline{m}(c)>-\infty $. Moreover, for any $u\in U$, according
to Lemma \ref{L5-7} and the fact of $\Phi _{r}^{+}(c)\subset \mathbf{B}%
_{R_{0}}\subset \mathbf{A}_{r_{0}},$ there exists a constant $s_{\star
}^{+}(u)>0$ such that $A(u_{s_{\star }^{+}(u)})<r_{0}$ and $E_{\gamma ,\mu
}(u_{s_{\star }^{+}(u)})<0,$ which implies that $\overline{m}(c)<0.$

It is clear that $\overline{m}(c)\leq \inf_{u\in \Phi _{r}^{+}(c)}E_{\gamma
,\mu }(u),$ since $\Phi _{r}^{+}(c)\subset \mathbf{B}_{R_{0}}\subset \mathbf{%
A}_{r_{0}}$. On the other hand, if $u\in \mathbf{A}_{r_{0}}\subset U_{r}$,
then there exists a constant $s_{\star }^{+}(u)>0$ such that $u_{s_{\star
}^{+}(u)}\in \Phi _{r}^{+}(c)\subset \mathbf{A}_{r_{0}}$ and
\begin{equation*}
E_{\gamma ,\mu }(u_{s_{\star }^{+}(u)})=\min \left\{ E_{\gamma ,\mu }(u_{s})%
\text{ }|\text{ }s>0\text{ and }A(u_{s})<r_{0}\right\} \leq E_{\gamma ,\mu
}(u).
\end{equation*}%
This implies that $\inf_{\Phi _{r}^{+}(c)}E_{\gamma ,\mu }(u)\leq \overline{m%
}(c)$. So we have $\overline{m}(c)=\inf_{u\in \Phi _{r}^{+}(c)}E_{\gamma
,\mu }(u),$ and together with (\ref{e5-20}), leading to%
\begin{equation*}
\overline{m}(c)=\inf_{u\in \Phi _{r}^{+}(c)}E_{\gamma ,\mu }(u)=d^{+}(c).
\end{equation*}%
This completes the proof.
\end{proof}

We now give the existence of a local minimizer for $E_{\gamma ,\mu }$
restricted to $\mathbf{A}_{r_{0}}.$

\begin{theorem}
\label{T5-1} Let $N\geq 2,\overline{p}<p<2_{\alpha }^{\ast },\overline{q}%
<q\leq 2^{\ast }$ with $Nq+2\alpha -2Np>0$. Then for $c>0$ small enough, the
functional $E_{\gamma ,\mu }$ has a local minimizer $u^{+}$ on $S(c)$
satisfying $E_{\gamma ,\mu }(u^{+})=\overline{m}(c)<0$ and
\begin{equation*}
\left[ \frac{2p(N(q-2)-4)}{\Gamma _{\ast }(Np-N-\alpha )(Nq+2\alpha -2Np)}%
\right] ^{\frac{2}{Np-N-\alpha -2}}\leq A(u^{+})<r_{0}.
\end{equation*}
\end{theorem}

\begin{proof}
Let $\{v_{n}\}$ be a minimizing sequence for $E_{\gamma ,\mu }$ restricted
to $\mathbf{A}_{r_{0}}$. Since $(v_{n})_{s_{\star }^{+}(v_{n})}\in \Phi
_{r}^{+}(c)\subset \mathbf{B}_{R_{0}}\subset \mathbf{A}_{r_{0}}$, we have $%
A((v_{n})_{s_{\star }^{+}(v_{n})})<r_{0}$ and
\begin{equation*}
E_{\gamma ,\mu }((v_{n})_{s_{\star }^{+}(v_{n})})=\min \left\{ E_{\gamma
,\mu }((v_{n})_{s})\text{ }|\text{ }s>0\text{ and }A((v_{n})_{s})<r_{0}%
\right\} \leq E_{\gamma ,\mu }(v_{n}).
\end{equation*}%
In this way we obtain a new minimizing sequence $\{\psi _{n}\}:=\left\{
(v_{n})_{t_{\star }^{+}(v_{n})}\right\} $ for $E_{\gamma ,\mu }$ restricted
to $\mathbf{A}_{r_{0}}$. Since $E_{\gamma ,\mu }(u)>0$ for all $u\in U_{r}$
with $A(u)\geq r_{0},$ by Ekeland's variational principle, we can find a new
minimizing sequence $\left\{ u_{n}^{+}\right\} \subset \mathbf{A}_{r_{0}}$
for $\overline{m}\left( c\right) $ such that%
\begin{equation}
\Vert u_{n}^{+}-\psi _{n}\Vert _{H^{1}}^{2}\rightarrow 0\text{ as }%
n\rightarrow \infty ,  \label{e5-18}
\end{equation}%
which is also a Palais-Smale sequence for $E_{\gamma ,\mu }$ on $S(c)$.
Since $\left\{ \psi _{n}\right\} \subset \Phi _{r}^{+}(c)\subset \mathbf{A}%
_{r_{0}}$, it follows from (\ref{e5-18}) that $Q(u_{n}^{+})\rightarrow 0$ as
$n\rightarrow \infty .$ Thus, by Lemmas \ref{L2-9} and \ref{L5-18} $%
(ii)-(iii)$, we obtain that $u_{n}^{+}\rightarrow u^{+}$ strongly in $H^{1}(%
\mathbb{R}^{N})$ up to a subsequence, and $u^{+}$ is an interior local
minimizer for $E_{\gamma ,\mu }$ restricted to $\mathbf{A}_{r_{0}}$
satisfying $E_{\gamma ,\mu }(u^{+})=\overline{m}(c)<0.$

For any $u\in \Phi _{r}^{-}\left( c\right) ,$ by (\ref{e5-3}), one has%
\begin{equation*}
A(u)<\left[ \frac{2p(N(q-2)-4)}{\Gamma _{\ast }(Np-N-\alpha )(Nq+2\alpha
-2Np)}\right] ^{\frac{2}{Np-N-\alpha -2}}<\left( \frac{\gamma R_{0}}{\Gamma
_{\ast }}\right) ^{\frac{2}{Np-N-\alpha }}
\end{equation*}%
for $c>0$ sufficiently small, which implies that $u\in \mathbf{A}_{r_{0}}.$
This indicates that $\Phi _{r}^{-}(c)\subset \mathbf{A}_{r_{0}}.$ Since $%
E_{\gamma ,\mu }(u)>0$ for all $u\in \Phi _{r}^{-}(c)$ and $E_{\gamma ,\mu
}(u^{+})<0,$ we obtain that%
\begin{equation*}
\left[ \frac{2p(N(q-2)-4)}{\Gamma _{\ast }(Np-N-\alpha )(Nq+2\alpha -2Np)}%
\right] ^{\frac{2}{Np-N-\alpha -2}}\leq A(u^{+})<r_{0}.
\end{equation*}%
This completes the proof.
\end{proof}

First of all, we prove Theorem \ref{t6} $(a)$. Let $\{u_{n}\}\subset
\mathcal{M}_{r}^{-}(c)$ be a Palais--Smale sequence for $E_{\gamma ,\mu }$
restricted to $S(c)$ at level $m^{-}(c)>0$ whose existence is ensured by
Lemma \ref{L5-21}. It follows from Lemma \ref{L5-8} that $\{u_{n}\}$ is
bounded in $H^{1}(\mathbb{R}^{N})$. Hence, by Lemma \ref{L2-9}, $(P_{c})$
has a ground state solution $u^{-}$ which is radially symmetric and
satisfies $E_{\gamma ,\mu }(u^{-})=m^{-}(c)>0$ for some $\lambda ^{-}>0$.

Next, we give the proof of Theorem \ref{t6} $(b)$. By Lemmas \ref{L5-6} and %
\ref{L5-20}, there exists a bounded Palais--Smale sequence $\{u_{n}\}\subset
\Phi _{r}^{-}(c)$ for $E_{\gamma ,\mu }$ restricted to $S(c)$ at the level $%
d^{-}(c).$ Thus, it follows from Lemmas \ref{L2-9} and \ref{L5-18} $(i)$
that $(P_{c})$ admits a positive radial solution $u^{-}\in \Phi ^{-}(c)$
satisfying $E_{\gamma ,\mu }(u^{-})=d^{-}(c)>0$ for some $\lambda >0$.
Therefore, together with Theorem \ref{T5-1}, $(P_{c})$ has two positive
radial solutions $(\lambda ^{\pm },u^{\pm })\in \mathbb{R}^{+}\times H^{1}(%
\mathbb{R}^{N})$ satisfying $E_{\gamma ,\mu }(u^{-})>0>E_{\gamma ,\mu
}(u^{+})$ and $A(u^{-})<A(u^{+}).$ This completes the proof.

\section{Dynamic properties of standing waves}

\subsection{Stability/Strong instability of standing waves}

\textbf{Firstly, we give the proof of Theorem \ref{t8}:} Following the
classical arguments of Cazenave and Lions \cite{CL}. Assume that there exist
an $\varepsilon _{0}>0$, a sequence of initial data $\left\{
u_{n}^{0}\right\} \subset U_{r}$ and a time sequence $\left\{ t_{n}\right\}
\subset \mathbb{R}^{+}$ such that the unique solution $\psi _{n}$ of the
Cauchy problem (\ref{e1-0}) with initial data $u_{n}^{0}=\psi _{n}(0,\cdot )$
satisfies
\begin{equation*}
\text{dist}_{H^{1}}(u_{n}^{0},\mathcal{M}_{c}^{r_{0}})<\frac{1}{n}\text{ and
dist}_{H^{1}}(\psi _{n}(t_{n},\cdot ),\mathcal{M}_{c}^{r_{0}})\geq
\varepsilon _{0}.
\end{equation*}%
Without loss of generality, we may assume that $\left\{ u_{n}^{0}\right\}
\subset U_{r}$. Since $\text{dist}_{H^{1}}(u_{n}^{0},\mathcal{M}%
_{c}^{r_{0}})\rightarrow 0$ as $n\rightarrow \infty $, the conservation laws
of the energy and mass imply that $\psi _{n}(t_{n},\cdot )$ is a minimizing
sequence for $d^{+}(c)$ provided $\psi _{n}(t_{n},\cdot )\subset \mathbf{A}%
_{r_{0}}$. Indeed, if $\psi _{n}(t_{n},\cdot )\subset (H^{1}\backslash
\mathbf{A}_{r_{0}})$, then by the continuity there exists $\bar{t}_{n}\in
\lbrack 0,t_{n})$ such that $\left\{ \psi _{n}(\bar{t}_{n},\cdot )\right\}
\subset \partial \mathbf{A}_{r_{0}}$. Hence, by Theorem \ref{T5-1} one has
\begin{equation*}
E_{\gamma ,\mu }(\psi _{n}(\bar{t}_{n},\cdot ))\geq \inf_{u\in \partial _{%
\mathbf{A}_{r_{0}}}}E_{\gamma ,\mu }(u)>\inf_{u\in \mathbf{A}%
_{r_{0}}}E_{\gamma ,\mu }(u)=d^{+}(c),
\end{equation*}%
which is a contradiction. Therefore, $\left\{ \psi _{n}(t_{n},\cdot
)\right\} $ is a minimizing sequence for $d^{+}(c)$. Then there exists $%
v_{0}\in \mathcal{M}_{c}^{r_{0}}$ such that $\psi _{n}(t_{n},\cdot
)\rightarrow v_{0}$ in $H^{1}$, which contradicts with dist$_{H^{1}}(\psi
_{n}(t_{n},\cdot ),\mathcal{M}_{c}^{r_{0}})\geq \varepsilon _{0}.$ This
completes the proof.

Define
\begin{equation*}
\mathcal{G}_{1}:=\left\{ u\in V\text{ }|\text{\ }E_{\gamma ,\mu
}(u)<\inf_{\Lambda ^{-}(c)}E_{\gamma ,\mu }\right\} ,\quad \mathcal{G}%
_{2}:=\left\{ u\in U\text{ }|\text{\ }E_{\gamma ,\mu }(u)<\inf_{\Phi
^{-}(c)}E_{\gamma ,\mu }\right\} ,
\end{equation*}%
and
\begin{equation*}
\mathcal{G}_{3}:=\left\{ u\in S(c)\text{ }|\text{\ }E_{\gamma ,\mu
}(u)<\inf_{\mathcal{M}(c)}E_{\gamma ,\mu }\right\} .
\end{equation*}

\begin{proposition}
\label{L6-1} Assume that the assumptions of Theorem \ref{t4} $(b)$ are
satisfied. Let $u\in \mathcal{G}_{1}$, then $g_{u}(s)$ has a local maximum
point $s_{\ast }^{-}(u)$. If $0<s_{\ast }^{-}(u)<\min \{s_{u}^{\ast
},1\}<\max \{s_{u}^{\ast },1\}<s_{\ast }^{+}(u)$ and $|x|u\in L^{2}(\mathbb{R%
}^{N})$, then the solution $\psi $ of the Cauchy problem (\ref{e1-0}) with
initial datum $u$ blows-up in finite time.
\end{proposition}

\begin{proof}
Following the classical arguments of Glassey \cite{G0} and Berestycki and
Cazenave \cite{BC}. For every $u\in V$, the $g_{u}(s)$ has a unique local
maximum point $s_{\ast }^{-}(u)$ on $(0,s_{u}^{\ast })$ and $g_{u}(s)$ is
strict decreasing and concave on $(s_{\ast }^{-}(u),s_{u}^{\ast })$. We
claim that if $u\in V$, then
\begin{equation}
Q(u)\leq E_{\gamma ,\mu }(u)-\inf_{u\in \Lambda ^{-}(c)}E_{\gamma ,\mu }.
\label{e6-3}
\end{equation}%
By the fact that $0<s_{\ast }^{-}(u)<\min \{s_{u}^{\ast },1\}<\max
\{s_{u}^{\ast },1\}<s_{\ast }^{+}(u)$, we have $Q(u)<0$, and
\begin{eqnarray*}
E_{\gamma ,\mu }(u)=g_{u}(1) &\geq &g_{u}(s_{\ast }^{-}(u))-(s_{\ast
}^{-}(u)-1)g_{u}^{\prime }(1) \\
&=&E_{\gamma ,\mu }(u_{s_{\ast }^{-}(u)})-|Q(u)|(1-s_{\ast }^{-}(u)) \\
&\geq &\inf_{u\in \Lambda ^{-}(c)}E_{\gamma ,\mu }(u)-|Q(u)| \\
&=&\inf_{u\in \Lambda ^{-}(c)}E_{\gamma ,\mu }(u)+Q(u),
\end{eqnarray*}%
which proves the claim.

Now let us consider the solution $\psi $ with initial datum $u$. Since by
the assumption that $0<s_{\ast }^{-}(u)<\min \{s_{u}^{\ast },1\}<\max
\{s_{u}^{\ast },1\}<s_{\ast }^{+}(u)$, and that the map $u\mapsto s_{\ast
}^{-}(u)$ is continuous, we also deduce that $0<s_{\ast }^{-}(\psi (\tau
))<\min \{s_{\psi (\tau )}^{\ast },1\}<\max \{s_{\psi (\tau )}^{\ast
},1\}<s_{\ast }^{+}(\psi (\tau ))$ for every $|\tau |$ small, say $|\tau |<%
\overline{\tau }$. By (\ref{e6-3}) and recalling the assumption $E_{\gamma
,\mu }(u)<\inf_{u\in \Lambda ^{-}(c)}E_{\gamma ,\mu }$, we have
\begin{eqnarray*}
Q(\psi (\tau )) &\leq &E_{\gamma ,\mu }(\psi (\tau ))-\inf_{u\in \Lambda
^{-}(c)}E_{\gamma ,\mu } \\
&=&E_{\gamma ,\mu }(u)-\inf_{u\in \Lambda ^{-}(c)}E_{\gamma ,\mu }:=-\delta
<0,
\end{eqnarray*}%
for every such $\tau $, and hence $0<s_{\ast }^{-}(\psi (\overline{\tau }%
))<\min \{s_{\psi (\overline{\tau })}^{\ast },1\}<\max \{s_{\psi (\overline{%
\tau })}^{\ast },1\}<s_{\ast }^{+}(\psi (\overline{\tau }))$. By using the
continuity argument one has
\begin{equation*}
Q(\psi (t))\leq -\delta \text{ for }t\in (T_{\min },T_{\max }).
\end{equation*}%
Since $|x|u\in L^{2}(\mathbb{R}^{N})$, it follows from the Virial identity
\cite[Proposition 6.5.1]{C} that the function
\begin{equation*}
f(t)=\int_{\mathbb{R}^{N}}|x|^{2}|\psi (t,x)|^{2}dx
\end{equation*}%
is of class $C^{2}$ with $f^{\prime \prime }(t)=8Q(\psi (t))\leq -8\delta <0$
for $t\in (T_{\min },T_{\max })$. Integrating twice in time gives
\begin{equation*}
0\leq f(t)\leq f(0)+f^{\prime }(0)t-4\delta t^{2}.
\end{equation*}%
Since the right hand side becomes negative for $t$ sufficiently large, it is
necessary that both $T_{\min }$ and $T_{\max }$ are bounded, which in turn
implies final time blow-up. The proof is complete.
\end{proof}

By using similar arguments, we have the following two propositions.

\begin{proposition}
\label{L6-2}Assume that the assumptions if Theorem \ref{t6} $(b)$ are
satisfied. Let $u\in \mathcal{G}_{2}$, then $g_{u}(s)$ has a local maximum
point $s_{\star }^{-}(u)$. If $0<s_{\star }^{-}(u)<\min \{s_{u}^{\star
},1\}<\max \{s_{u}^{\star },1\}<s_{\star }^{+}(u)$ and $|x|u\in L^{2}(%
\mathbb{R}^{N})$, then the solution $\psi $ of the Cauchy problem (\ref{e1-0}%
) with initial datum $u$ blows-up in finite time.
\end{proposition}

\begin{proposition}
\label{L6-4}Assume that the assumptions of either Theorem \ref{t4} $(a)$ or
Theorem \ref{t6} $(a)$ are satisfied. Let $u\in \mathcal{G}_{3}$, then $%
g_{u}(s)$ has a unique global maximum point $s^{-}(u)$. If $0<s^{-}(u)<1$
and $|x|u\in L^{2}(\mathbb{R}^{N})$, then the solution $\psi $ of the Cauchy
problem (\ref{e1-0}) with initial datum $u$ blows-up in finite time.
\end{proposition}

\textbf{Now we give the proof of Theorem \ref{t7}:} We only prove the
conclusion $(ii)$, since the others are similar. Since $u^{-}$ has
exponential decay at infinity, we have $|x|u^{-}\in L^{2}(\mathbb{R}^{N})$.
For every $s>1$, let $\psi _{t}$ be the solution to the Cauchy problem (\ref%
{e1-0}) with initial datum $(u^{-})_{s}$. We have $(u^{-})_{s}\rightarrow
u^{-}$ as $s\rightarrow 1$, and hence it is sufficient to prove that $\psi
_{t}$ blows-up in finite time. Clearly, there exists a $\overline{\mu }>0$
such that $s_{\ast }^{-}((u^{-})_{s})=\frac{1}{s}<1<s_{(u^{-})_{s}}^{\ast }$
for $\mu <\overline{\mu }$, and
\begin{equation*}
E_{\gamma ,\mu }((u^{-})_{s})<E_{\gamma ,\mu }((u^{-})_{s_{\ast
}^{-}(u^{-})})=\inf_{\Lambda ^{-}(c)}E_{\gamma ,\mu }.
\end{equation*}%
Therefore, it follows from Lemma \ref{L6-1} that the solution $\psi _{t}$
blows-up in finite time. This completes the proof.

\subsection{Blow up rate}

\begin{lemma}
\label{L6-3} Assume that $N\geq 2$, $\alpha \in (\max \{0,N-4\},N)$, $2\leq
p<2_{\alpha }^{\ast }$ and $\overline{q}<q<2^{\ast }$ with $2(Np-N-\alpha
)>(N-2)(q-2)p.$ For any $0<t<\tau <T<\infty $, we have the following
estimates:
\begin{eqnarray*}
&&\Vert \nabla g_{1}(\psi )\Vert _{L_{t}^{\frac{4p}{N+\alpha -(N-4)p}%
}((t,\tau );L_{x}^{\frac{2Np}{2Np-N-\alpha }})} \\
&\lesssim &(\tau -t)^{\frac{N+\alpha -(N-2)p}{2p}}\Vert \psi \Vert
_{L_{t}^{\infty }((t,\tau );L_{x}^{2})}^{\frac{(N+\alpha -(N-2)p)(p-1)}{p}%
}\Vert \nabla \psi \Vert _{L_{t}^{\infty }((t,\tau );L_{x}^{2})}^{\frac{%
(Np-N-\alpha )(p-1)}{p}}\Vert \nabla \psi \Vert _{L_{t}^{\frac{4p}{%
Np-N-\alpha }}((t,\tau );L_{x}^{\frac{2Np}{N+\alpha }})},
\end{eqnarray*}%
and
\begin{eqnarray*}
&&\Vert \nabla g_{2}(\psi )\Vert _{L_{t}^{\frac{4p}{N+\alpha -(N-4)p}%
}((t,\tau );L_{x}^{\frac{2Np}{2Np-N-\alpha }})} \\
&\lesssim &(\tau -t)^{\frac{N+\alpha -(N-2)p}{2p}}\Vert \psi \Vert
_{L_{t}^{\infty }((t,\tau );L_{x}^{2})}^{\frac{2(N+\alpha )+(q-4)Np}{2p}%
}\Vert \nabla \psi \Vert _{L_{t}^{\infty }((t,\tau );L_{x}^{2})}^{\frac{%
2(Np-N-\alpha )-(N-2)(q-2)p}{2p}}\Vert \nabla \psi \Vert _{L_{t}^{\frac{4p}{%
Np-N-\alpha }}((t,\tau );L_{x}^{\frac{2Np}{N+\alpha }})}.
\end{eqnarray*}
\end{lemma}

\begin{proof}
It follows from the H\"{o}lder inequality and (\ref{e2-4}) that
\begin{eqnarray*}
\Vert \nabla g_{1}(\psi )\Vert _{L_{x}^{\frac{2Np}{2Np-N-\alpha }}}
&\lesssim &\Vert \psi \Vert _{L_{x}^{\frac{2Np}{N+\alpha }}}^{2(p-1)}\Vert
\nabla \psi \Vert _{L_{x}^{\frac{2Np}{N+\alpha }}} \\
&\lesssim &\Vert \psi \Vert _{L_{x}^{2}}^{\frac{(N+\alpha -(N-2)p)(p-1)}{p}%
}\Vert \nabla \psi \Vert _{L_{x}^{2}}^{\frac{(Np-N-\alpha )(p-1)}{p}}\Vert
\nabla \psi \Vert _{L_{x}^{\frac{2Np}{N+\alpha }}}.
\end{eqnarray*}%
Then for any $0<t<\tau <T<+\infty $, we get
\begin{eqnarray*}
&&\Vert \nabla g_{1}(\psi )\Vert _{L_{t}^{\frac{4p}{N+\alpha -(N-4)p}%
}((t,\tau );L_{x}^{\frac{2Np}{2Np-N-\alpha }})} \\
&\lesssim &(\tau -t)^{\frac{N+\alpha -(N-2)p}{2p}}\Vert \psi \Vert
_{L_{t}^{\infty }((t,\tau );L_{x}^{2})}^{\frac{(N+\alpha -(N-2)p)(p-1)}{p}%
}\Vert \nabla \psi \Vert _{L_{t}^{\infty }((t,\tau );L_{x}^{2})}^{\frac{%
(Np-N-\alpha )(p-1)}{p}}\Vert \nabla \psi \Vert _{L_{t}^{\frac{4p}{%
Np-N-\alpha }}((t,\tau );L_{x}^{\frac{2Np}{N+\alpha }})}.
\end{eqnarray*}%
Similarly, we have
\begin{eqnarray*}
\Vert \nabla g_{2}(\psi )\Vert _{L_{x}^{\frac{2Np}{2Np-N-\alpha }}}
&\lesssim &\Vert \psi \Vert _{L_{x}^{\frac{Np(q-2)}{Np-N-\alpha }%
}}^{q-2}\Vert \nabla \psi \Vert _{L_{x}^{\frac{2Np}{N+\alpha }}} \\
&\lesssim &\Vert \psi \Vert _{L_{x}^{2}}^{\frac{2(N+\alpha )+(q-4)Np}{2p}%
}\Vert \nabla \psi \Vert _{L_{x}^{2}}^{\frac{2(Np-N-\alpha )-(N-2)(q-2)p}{2p}%
}\Vert \nabla \psi \Vert _{L_{x}^{\frac{2Np}{N+\alpha }}},
\end{eqnarray*}%
and further
\begin{eqnarray*}
&&\Vert \nabla g_{2}(\psi )\Vert _{L_{t}^{\frac{4p}{N+\alpha -(N-4)p}%
}((t,\tau );L_{x}^{\frac{2Np}{2Np-N-\alpha }})} \\
&\lesssim &(\tau -t)^{\frac{N+\alpha -(N-2)p}{2p}}\Vert \psi \Vert
_{L_{t}^{\infty }((t,\tau );L_{x}^{2})}^{\frac{2(N+\alpha )+(q-4)Np}{2p}%
}\Vert \nabla \psi \Vert _{L_{t}^{\infty }((t,\tau );L_{x}^{2})}^{\frac{%
2(Np-N-\alpha )-(N-2)(q-2)p}{2p}}\Vert \nabla \psi \Vert _{L_{t}^{\frac{4p}{%
Np-N-\alpha }}((t,\tau );L_{x}^{\frac{2Np}{N+\alpha }})},
\end{eqnarray*}%
for any $0<t<\tau <T<+\infty $. This completes the proof.
\end{proof}

\textbf{Now we give the proof of Theorem \ref{t10}:} We choose the $L^{2}$%
-admissible $(\frac{4p}{Np-N-\alpha },\frac{2Np}{N+\alpha })$. By applying
the Proposition \ref{P3.4} $(i)$ and Lemma \ref{L6-3}, we have
\begin{eqnarray*}
&&\Vert \nabla \psi \Vert _{L_{t}^{\infty }((t,\tau );L_{x}^{2})}+\Vert
\nabla \psi \Vert _{L_{t}^{\frac{4p}{Np-N-\alpha }}((t,\tau );L_{x}^{\frac{%
2Np}{N+\alpha }})} \\
&\leq &C\Vert \nabla \psi \Vert _{L_{x}^{2}}+C(\tau -t)^{\frac{N+\alpha
-(N-2)p}{2p}}\Vert \psi \Vert _{L_{t}^{\infty }((t,\tau );L_{x}^{2})}^{\frac{%
(N+\alpha -(N-2)p)(p-1)}{p}}\Vert \nabla \psi \Vert _{L_{t}^{\infty
}((t,\tau );L_{x}^{2})}^{\frac{(Np-N-\alpha )(p-1)}{p}}\Vert \nabla \psi
\Vert _{L_{t}^{\frac{4p}{Np-N-\alpha }}((t,\tau );L_{x}^{\frac{2Np}{N+\alpha
}})} \\
&&+C(\tau -t)^{\frac{N+\alpha -(N-2)p}{2p}}\Vert \psi \Vert _{L_{t}^{\infty
}((t,\tau );L_{x}^{2})}^{\frac{2(N+\alpha )+(q-4)Np}{2p}}\Vert \nabla \psi
\Vert _{L_{t}^{\infty }((t,\tau );L_{x}^{2})}^{\frac{2(Np-N-\alpha
)-(N-2)(q-2)p}{2p}}\Vert \nabla \psi \Vert _{L_{t}^{\frac{4p}{Np-N-\alpha }%
}((t,\tau );L_{x}^{\frac{2Np}{N+\alpha }})} \\
&\leq &C\Vert \nabla \psi \Vert _{L_{x}^{2}}+C(\tau -t)^{\frac{N+\alpha
-(N-2)p}{2p}}\left( 1+\Vert \nabla \psi \Vert _{L_{t}^{\infty }((t,\tau
);L_{x}^{2})}+\Vert \nabla \psi \Vert _{L_{t}^{\frac{4p}{Np-N-\alpha }%
}((t,\tau );L_{x}^{\frac{2Np}{N+\alpha }})}\right) ^{\frac{(Np-N-\alpha
)(p-1)+p}{p}}.
\end{eqnarray*}%
Denote $G_{t}(\tau ):=1+\Vert \nabla \psi \Vert _{L_{t}^{\infty }((t,\tau
);L_{x}^{2})}+\Vert \nabla \psi \Vert _{L_{t}^{\frac{4p}{Np-N-\alpha }%
}((t,\tau );L_{x}^{\frac{2Np}{N+\alpha }})}$. Then, for any $0<t<\tau
<+\infty $, there exists $C_{0}>0$ such that
\begin{equation*}
G_{t}(\tau )\leq C_{0}(1+\Vert \nabla \psi \Vert _{L_{x}^{2}})+C_{0}(\tau
-t)^{\frac{N+\alpha -(N-2)p}{2p}}G_{t}^{\frac{(Np-N-\alpha )(p-1)+p}{p}%
}(\tau ).
\end{equation*}%
If $\psi (t,x)$ is a blow-up solution for the Cauchy problem (\ref{e1-0}),
then $G_{t}(\tau )$ is continuous and nondecreasing on $(t,T_{\max })$ and $%
\lim_{t\rightarrow T_{\max }}G_{t}(\tau )=+\infty $. Moreover, for $\tau >t$%
, we have $G_{t}(\tau )\rightarrow 1+\Vert \nabla \psi \Vert _{L_{x}^{2}}$
as $\tau \rightarrow t$. Thus, there exists $\tau _{0}\in (t,T_{\max })$
such that $G_{t}(\tau _{0})=(C_{0}+1)(1+\Vert \nabla \psi \Vert
_{L_{x}^{2}}) $. So, it follows that
\begin{eqnarray*}
1+\Vert \nabla \psi \Vert _{L_{x}^{2}} &=&G_{t}(\tau _{0})-C_{0}(1+\Vert
\nabla \psi \Vert _{L_{x}^{2}}) \\
&\leq &C_{0}(\tau _{0}-t)^{\frac{N+\alpha -(N-2)p}{2p}}(C_{0}+1)^{\frac{%
(Np-N-\alpha )(p-1)+p}{p}}(1+\Vert \nabla \psi \Vert _{L_{x}^{2}})^{\frac{%
(Np-N-\alpha )(p-1)+p}{p}} \\
&\leq &(C_{0}+1)^{\frac{(Np-N-\alpha )(p-1)+2p}{p}}(T-t)^{\frac{N+\alpha
-(N-2)p}{2p}}(1+\Vert \nabla \psi \Vert _{L_{x}^{2}})^{\frac{(Np-N-\alpha
)(p-1)+p}{p}}.
\end{eqnarray*}%
Thus for all $0<t<\tau <+\infty $, we obtain that
\begin{equation*}
1+\Vert \nabla \psi \Vert _{L_{x}^{2}}\geq \frac{1}{(C_{0}+1)^{\frac{%
(Np-N-\alpha )(p-1)+2p}{(Np-N-\alpha )(p-1)}}(T-t)^{\frac{N+\alpha -(N-2)p}{%
2(Np-N-\alpha )(p-1)}}}.
\end{equation*}%
This completes the proof.

\section{Acknowledgments}

The authors thank Louis Jeanjean for some valuable suggestions that help
them to improve a first version of the paper and, in particular, to show
that under the assumptions of Theorem \ref{t6} $(b)$, the functional $%
E_{\gamma ,\mu }$ is bounded from below and coercive on $S(c)$. J. Sun was
supported by the National Natural Science Foundation of China (Grant No.
11671236) and Shandong Provincial Natural Science Foundation (Grant No.
ZR2020JQ01), and T.F. Wu was supported, in part, by the Ministry of Science
and Technology, Taiwan (Grant No. 110-2115-M-390-006-MY2).

\section{Declarations}

\textbf{Conflict of interest} The authors confirm that there is no conflict of interest.

\end{document}